\documentclass[11pt]{article}

\usepackage[letterpaper, top=25.4mm, bottom=25.4mm, left=25.4mm, right=25.4mm, includefoot]{geometry}

\usepackage[utf8]{inputenc}
\usepackage[T1]{fontenc}
\usepackage{lmodern}
\usepackage{libertine}
\usepackage{comment}

\usepackage{titling}
\usepackage{xspace}
\usepackage{soul} 

\usepackage{ifthen}
\usepackage{textpos}

\usepackage{tikz}
\usepackage{xparse}

\newcommand{\colorname}[1]{%
  \ifnum#1=1 red\else%
    \ifnum#1=2 blue\else%
      \ifnum#1=3 green\else%
        \ifnum#1=4 magenta\else%
          \ifnum#1=5 orange\else%
            black%
          \fi%
        \fi%
      \fi%
    \fi%
  \fi%
}

\usepackage[inline]{enumitem}

\usepackage{dsfont}
\usepackage{setspace}

\usepackage{bm}
\usepackage{microtype}
\usepackage{amsmath}
\usepackage{amssymb}
\usepackage{amsfonts}
\usepackage{mathtools}
\usepackage[mathscr]{euscript}

\usepackage{thmtools} 
\usepackage{thm-restate}

\usepackage[hyphens]{url}
\usepackage{amsthm}

\usepackage{xcolor}
\usepackage{subcaption}
\usepackage{graphicx}
\usepackage{wrapfig}

\usepackage{hyperref}
\usepackage{aliascnt} 

\usepackage{cleveref}

\usepackage{nicefrac}
\usepackage[textwidth=2.2cm]{todonotes}
\usepackage{marginnote}

\colorlet{myGreen}{green!50!black}
\colorlet{myLightgreen}{green}
\colorlet{myRed}{red!90!black}
\definecolor{myBlue}{rgb}{0.25, 0.0, 1.0}
\definecolor{myLightBlue}{rgb}{0.39, 0.58, 0.93}
\colorlet{myViolet}{myBlue!55!myRed}
\definecolor{myOrange}{rgb}{1.0, 0.66, 0.07}

\definecolor{CornflowerBlue}{rgb}{0.39, 0.58, 0.93}
\definecolor{DarkGoldenrod}{rgb}{0.72, 0.53, 0.04}
\definecolor{BritishRacingGreen}{rgb}{0.0, 0.26, 0.15}
\definecolor{DarkMagenta}{rgb}{0.55, 0.0, 0.55}
\definecolor{AO}{rgb}{0.0, 0.5, 0.0}
\definecolor{BostonUniversityRed}{rgb}{0.8, 0.0, 0.0}
\definecolor{myRed}{rgb}{0.8, 0.0, 0.0}
\definecolor{DarkMidnightBlue}{rgb}{0.0, 0.2, 0.4}
\definecolor{DarkTangerine}{rgb}{1.0, 0.66, 0.07}
\definecolor{AppleGreen}{rgb}{0.55, 0.71, 0.0}
\definecolor{BrightUbe}{rgb}{0.82, 0.62, 0.91}
\definecolor{Amethyst}{rgb}{0.6, 0.4, 0.8}
\definecolor{DarkGray}{rgb}{0.52, 0.52, 0.51}
\definecolor{Gray}{rgb}{0.66, 0.66, 0.66}
\definecolor{BananaYellow}{rgb}{1.0, 0.88, 0.21}
\definecolor{Amber}{rgb}{1.0, 0.75, 0.0}
\definecolor{LightGray}{rgb}{0.83, 0.83, 0.83}
\definecolor{PrincetonOrange}{rgb}{1.0, 0.56, 0.0}
\definecolor{DeepCarrotOrange}{rgb}{0.91, 0.41, 0.17}
\definecolor{CarrotOrange}{rgb}{0.93, 0.57, 0.13}
\definecolor{MidnightBlue}{rgb}{0.1, 0.1, 0.44}
\definecolor{Magenta}{rgb}{0.50, 0.0, 0.50}
\definecolor{BrightPink}{rgb}{1.0, 0.0, 0.5}
\definecolor{BrilliantRose}{rgb}{1.0, 0.33, 0.64}
\definecolor{ChromeYellow}{rgb}{1.0, 0.65, 0.0}
\definecolor{HotMagenta}{rgb}{1.0, 0.11, 0.81}

\definecolor{DarkTangerine}{rgb}{1.0, 0.66, 0.07}
\definecolor{darkyellow}{rgb}{.7, .6, 0.0}
\definecolor{CornflowerBlue}{rgb}{0.39, 0.58, 0.93}
\definecolor{DarkGoldenrod}{rgb}{0.72, 0.53, 0.04}
\definecolor{BritishRacingGreen}{rgb}{0.0, 0.26, 0.15}
\definecolor{AO}{rgb}{0.0, 0.5, 0.0}
\definecolor{MidnightBlack}{rgb}{0.1,0.1,.34}
\definecolor{MidnightBlue}{rgb}{0.1,0.1,0.43}
\definecolor{Black}{rgb}{0,0, 0}
\definecolor{Blue}{rgb}{0, 0 ,1}
\definecolor{Red}{rgb}{1, 0 ,0}
\definecolor{White}{rgb}{1, 1, 1}
\definecolor{DeepMagenta}{rgb}{0.8, 0.0, 0.8}
\definecolor{grey}{rgb}{.6, .6, .6}
\definecolor{darkgrey}{rgb}{.33, .33, .33}
\definecolor{Mygreen}{rgb}{.0, .7, .0}
\definecolor{Yellow}{rgb}{.55,.55,0}
\definecolor{Mustard}{rgb}{1.0, 0.86, 0.35}
\definecolor{applegreen}{rgb}{0.55, 0.71, 0.0}
\definecolor{darkturquoise}{rgb}{0.0, 0.81, 0.82}
\definecolor{celestialblue}{rgb}{0.29, 0.59, 0.82}
\definecolor{green_yellow}{rgb}{0.68, 1.0, 0.18}
\definecolor{crimsonglory}{rgb}{0.75, 0.0, 0.2}
\definecolor{darkmagenta}{rgb}{0.30, 0.0, 0.30}
\definecolor{magenta}{rgb}{0.50, 0.0, 0.50}
\definecolor{internationalorange}{rgb}{1.0, 0.31, 0.0}
\definecolor{darkorange}{rgb}{1.0, 0.55, 0.0}
\definecolor{ao}{rgb}{0.0, 0.5, 0.0}
\definecolor{awesome}{rgb}{1.0, 0.13, 0.32}
\definecolor{darkcyan}{rgb}{0.0, 0.50, 0.50}
\definecolor{violet}{rgb}{0.93, 0.51, 0.93}
\definecolor{brown}{rgb}{0.65, 0.16, 0.16}
\definecolor{orange}{rgb}{1.0, 0.65, 0.0}
\definecolor{DarkGreen}{rgb}{0,.5,0}
\definecolor{BostonUniversityRed}{rgb}{0.8, 0.0, 0.0}

\vfuzz \hfuzz
\setlist[itemize]{topsep=0pt,partopsep=0pt,itemsep=0pt,parsep=0pt}
\setlist[itemize,1]{label={\small\textbullet}}
\setlist[itemize,2]{label={\tiny\textbullet}}
\setlist[itemize,3]{label=$\cdot$}
\setlist[enumerate]{topsep=0pt,partopsep=0pt,itemsep=0pt,parsep=0pt}
\setlist[enumerate,1]{label=\roman*)}
\setlist[enumerate,2]{label=\alph*)}
\setlist[enumerate,3]{label=\arabic*)}

\hypersetup{
colorlinks=true,
linkcolor=CornflowerBlue!65!black,
citecolor=CornflowerBlue!65!black,
urlcolor=CornflowerBlue!65!black,
bookmarksopen=true,
bookmarksnumbered,
bookmarksopenlevel=2,
bookmarksdepth=3
}

\newtheorem*{cmc*}{Coarse Menger's Conjecture}
\crefname{cmc*}{lemma}{Conjectures}
\crefformat{cmc*}{#2Coarse Menger's Conjecture~#1#3}
\Crefformat{cmc*}{#2Coarse Menger's Conjecture~#1#3}

\newcommand{\DeclareCleverTheorem}[4]{%
  \newaliascnt{#1}{#2}%
  \newtheorem{#1}[#1]{#3}%
  \aliascntresetthe{#1}
  \crefname{#1}{#3}{#4}%
  \crefformat{#1}{##2#3~##1##3}%
  \Crefformat{#1}{##2#3~##1##3}%
  \newtheorem*{#1*}{#3}
  \crefname{#1*}{#3}{#4}%
  \crefformat{#1*}{##2~#1##3}%
  \Crefformat{#1*}{##2~#1##3}%
}

\DeclareCleverTheorem{proposition}{environment}{Proposition}{Propositions}
\DeclareCleverTheorem{corollary}{environment}{Corollary}{Corollaries}
\DeclareCleverTheorem{theorem}{environment}{Theorem}{Theorems}
\DeclareCleverTheorem{conjecture}{environment}{Conjecture}{Conjectures}
\DeclareCleverTheorem{observation}{environment}{Observation}{Observations}
\DeclareCleverTheorem{example}{environment}{Example}{Examples}
\DeclareCleverTheorem{remark}{environment}{Remark}{Remarks}
\DeclareCleverTheorem{notation}{environment}{Notation}{Notations}
\DeclareCleverTheorem{question}{environment}{Question}{Questions}
\DeclareCleverTheorem{problem}{environment}{Problem}{Problems}
\DeclareCleverTheorem{definition}{environment}{Definition}{Definitions}
\DeclareCleverTheorem{lemma}{environment}{Lemma}{Lemmas}

\crefname{figure}{figure}{figures}
\crefformat{figure}{#2Figure~#1#3}
\Crefformat{figure}{#2Figure~#1#3}

\crefname{equation}{equation}{Equations}
\crefformat{equation}{#2Equation~#1#3}
\Crefformat{equation}{#2Equation~#1#3}

\crefname{chapter}{chapter}{chapters}
\crefformat{chapter}{#2Chapter~#1#3}
\Crefformat{chapter}{#2Chapter~#1#3}

\crefname{section}{section}{sections}
\crefformat{section}{#2Section~#1#3}
\Crefformat{section}{#2Section~#1#3}

\crefname{algorithm}{algorithm}{algorithms}
\crefformat{algorithm}{#2Algorithm~#1#3}
\Crefformat{algorithm}{#2Algorithm~#1#3}

\newtheorem{claim}{Claim}
\crefname{claim}{claim}{claims}
\crefformat{claim}{#2Claim~#1#3}
\Crefformat{claim}{#2Claim~#1#3}

\usetikzlibrary{calc}
\usetikzlibrary{fit}
\usetikzlibrary{decorations}
\usetikzlibrary{decorations.pathmorphing}
\usetikzlibrary{decorations.text}
\usetikzlibrary{shapes,hobby}
\usetikzlibrary{positioning}

\tikzset{
	position/.style args={#1:#2 from #3}{
		at=($(#3)+(#1:#2)$)
	}
}

\tikzset{
  v:main/.style = {draw, circle, scale=0.8, thick,fill=black,inner sep=0.7mm},
  v:ghost/.style = {inner sep=0pt,scale=1},
  >={latex},
  e:marker/.style = {line width=8.5pt,line cap=round,opacity=0.35,color=DarkGoldenrod},
  e:main/.style = {line width=1pt},
}

\RequirePackage{stmaryrd}
\usepackage{textcomp}
\DeclareUnicodeCharacter{2286}{\subseteq}
\DeclareUnicodeCharacter{2192}{\ifmmode\to\else\textrightarrow\fi}
\DeclareUnicodeCharacter{2203}{\ensuremath\exists}
\DeclareUnicodeCharacter{183}{\cdot}
\DeclareUnicodeCharacter{2200}{\forall}
\DeclareUnicodeCharacter{2264}{\leq}
\DeclareUnicodeCharacter{2265}{\geq}
\DeclareUnicodeCharacter{8614}{\mathbin{\mapsto}}
\DeclareUnicodeCharacter{8656}{\Leftarrow}
\DeclareUnicodeCharacter{8657}{\Uparrow}
\DeclareUnicodeCharacter{8658}{\Rightarrow}
\DeclareUnicodeCharacter{8659}{\Downarrow}
\DeclareUnicodeCharacter{8669}{\rightsquigarrow}
\newcommand{\eqdef}{\stackrel{{\scriptsize\rm def}}{=}}
\DeclareUnicodeCharacter{8797}{\eqdef}
\DeclareUnicodeCharacter{8870}{\vdash}
\DeclareUnicodeCharacter{8873}{\Vdash}
\DeclareUnicodeCharacter{22A7}{\models}
\DeclareUnicodeCharacter{9121}{\lceil}
\DeclareUnicodeCharacter{9123}{\lfloor}
\DeclareUnicodeCharacter{9124}{\rceil}
\DeclareUnicodeCharacter{2208}{\in}
\DeclareUnicodeCharacter{9126}{\rfloor}
\DeclareUnicodeCharacter{9655}{\triangleright}
\DeclareUnicodeCharacter{9665}{\triangleleft}
\DeclareUnicodeCharacter{9671}{\diamond}
\DeclareUnicodeCharacter{9675}{\circ}
\DeclareUnicodeCharacter{10178}{\bot}
\DeclareUnicodeCharacter{10214}{} 
\DeclareUnicodeCharacter{10215}{} 
\DeclareUnicodeCharacter{10229}{\longleftarrow}
\DeclareUnicodeCharacter{10230}{\longrightarrow}
\DeclareUnicodeCharacter{10231}{\longleftrightarrow}
\DeclareUnicodeCharacter{10232}{\Longleftarrow}
\DeclareUnicodeCharacter{10233}{\Longrightarrow}
\DeclareUnicodeCharacter{10234}{\Longleftrightarrow}
\DeclareUnicodeCharacter{10236}{\longmapsto}
\DeclareUnicodeCharacter{10238}{\Longmapsto} 
\DeclareUnicodeCharacter{10503}{\Mapsto}    
\DeclareUnicodeCharacter{10971}{\mathrel{\not\hspace{-0.2em}\cap}}
\DeclareUnicodeCharacter{65294}{\ldotp}
\DeclareUnicodeCharacter{65372}{\mid}

\newcommand{\N}{\mathbb{N}}

\newcommand{\rep}{\mathsf{rep}\xspace}

\newcommand{\bd}{\mathsf{bd}\xspace}

\newcommand{\remove}[1]{}

\newcommand{\belsf}{\downarrow}
\newcommand{\abvsf}{\uparrow}

\newcommand{\repr}{\mathsf{rep}}

\newenvironment{claimproof}[1][Proof.]{%
	\begin{proof}[#1]%
	}{%
	\end{proof}%
}

\newcommand{\Af}{\mathbb{A}}

\newcommand{\Hh}{\mathcal{H}}
\newcommand{\Ff}{\mathcal{F}}

\newcommand{\Bb}{\mathcal{B}}
\newcommand{\Pp}{\mathcal{P}}
\newcommand{\Qq}{\mathcal{Q}}
\newcommand{\Rr}{\mathcal{R}}
\newcommand{\Ll}{\mathcal{L}}
\newcommand{\Cc}{\mathcal{C}}
\newcommand{\Zz}{\mathcal{Z}}

\newcommand{\wh}[1]{\widetilde{#1}}

\newcommand{\Oh}{\mathcal{O}}

\newcommand{\holesf}{\mathsf{pocket}}
\newcommand{\exsf}{\mathsf{ex}}
\newcommand{\sisf}{\mathsf{si}}
\newcommand{\insf}{\mathsf{in}}
\newcommand{\outsf}{\mathsf{out}}

\newcommand{\FDisk}{\Theta}

\newcommand{\Ball}{\mathsf{Ball}}
\newcommand{\dist}{\mathsf{dist}}
\newcommand{\per}{\mathsf{per}}
\newcommand{\inte}{\mathsf{int}}
\newcommand{\exte}{\mathsf{ext}}
\newcommand{\ground}{\mathsf{ground}}

\newcommand{\rend}{\mathfrak{D}}

\renewcommand{\leq}{\leqslant}
\renewcommand{\geq}{\geqslant}

\renewcommand{\setminus}{-}

\newcommand{\link}{\mathsf{link}}
\newcommand{\sep}{\mathsf{separation}}
\newcommand{\adh}{\mathsf{adhesion}}
\newcommand{\cov}{\mathsf{cover}}
\newcommand{\vig}{\mathsf{vigilance}}
\newcommand{\depth}{\mathsf{depth}}

\newcommand{\sfh}{\mathsf{h}}
\newcommand{\sfc}{\mathsf{c}}

\title{A coarse Menger's Theorem for planar and bounded genus graphs\thanks{V.B. and M.P. were supported by the project BOBR that is funded from the European Research Council (ERC) under the European Union's Horizon 2020 research and innovation programme with grant agreement No. 948057. E.P. was supported by the ERC grant BUKA (No. 101126229).}}

\author{
Václav Blažej\\{\small University of Warsaw}\\\href{mailto:v.blazej@uw.edu.pl}{\small v.blazej@uw.edu.pl}
\and
Michał Pilipczuk\\{\small University of Warsaw}\\\href{mailto:michal.pilipczuk@mimuw.edu.pl}{\small michal.pilipczuk@mimuw.edu.pl}
\and
Evangelos Protopapas\\{\small University of Warsaw}\\\href{mailto:eprotopapas@mimuw.edu.pl}{\small eprotopapas@mimuw.edu.pl}
}

\date{}

\begin{document}
\maketitle

\begin{abstract}
	\noindent Menger's Theorem is a fundamental result in graph theory. It states that if in a graph $G$ with distinguished sets of terminal vertices $S$ and $T$ there are no $k$ pairwise vertex-disjoint $S$-$T$ paths, then there is a set of less than $k$ vertices that intersects every $S$-$T$ path. In this work, we give a coarse variant of this result for planar and bounded genus graphs. Precisely, we prove that for every surface $\Sigma$ there is a function $f\colon \N\times \N\to \N$ such that for every pair of integers $d,k\in \N$ and a $\Sigma$-embeddable graph $G$ with distinguished sets of terminal vertices $S$ and $T$, if $G$ does not contain a family of $k$ $S$-$T$ paths that are pairwise at distance larger than $d$, then there is a set $X$ consisting of at most $f(d,k)$ vertices of $G$ such that every $S$-$T$ path is at distance at most $d$ from a vertex of $X$.
	This partially answers questions of Nguyen, Scott, and Seymour~\cite{AS4:WeakMenger}, who proved that such a result cannot hold in general graphs.
	
    A key ingredient of our proof is a structure theorem from the developing ``colorful'' graph minor theory, where the focus is on studying the structure in a graph relative to some fixed subsets of annotated vertices. In our case, these annotated vertices are $S$ and $T$.
\end{abstract}

\begin{textblock}{20}(-1.6, 3.4)
	\includegraphics[width=35px]{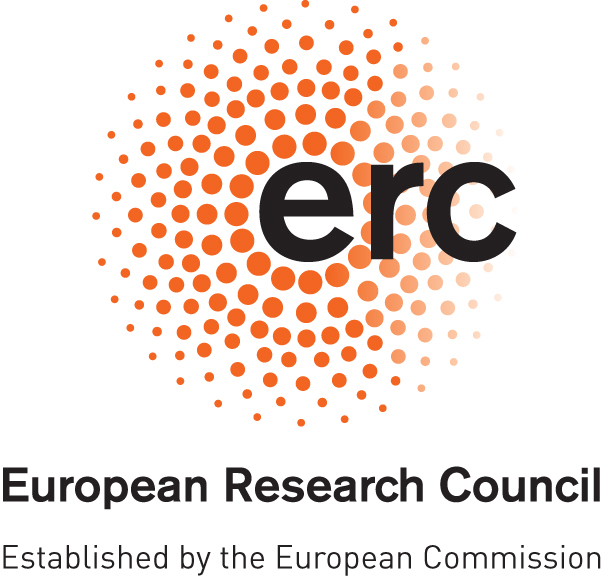}%
\end{textblock}
\begin{textblock}{20}(-1.6, 4.3)
	\includegraphics[width=35px]{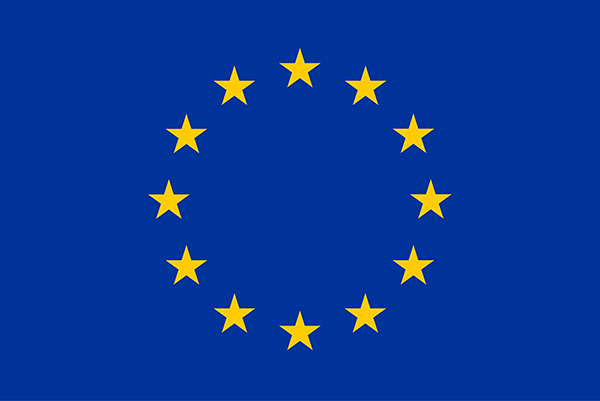}%
\end{textblock}


\thispagestyle{empty}

\newpage
\thispagestyle{empty}
\tableofcontents
\thispagestyle{empty}
\newpage

\setcounter{page}{1}

\section{Introduction}\label{sec:intro}

Menger's Theorem~\cite{Menger27} is one of the most fundamental results in graph theory. In the undirected setting, it states that given an undirected graph $G$ with prescribed vertex subsets $S,T\subseteq V(G)$, if one cannot find $k$ vertex-disjoint $S$-$T$-paths (i.e., paths connecting a vertex of $S$ with a vertex of $T$), then in fact this is because there is a small cut separating $S$ and $T$: a set $X\subseteq V(G)$ of size smaller than $k$ that intersects every $S$-$T$-path. The basic flow/cut dichotomy provided by Menger's Theorem led to countless duality results in structural graph theory, and particularly in Graph Minors theory. Typically, such duality results connect the non-existence of a certain obstruction with the existence of a decomposition. The obstructions are usually weaved from families of disjoint paths, whereas cuts can be used to decompose the graph.

Recently, there has been a rising interest in {\em{coarse graph theory}}: an attempt of building a structure theory for graphs understood as metric spaces, where the conditions that objects should be disjoint or intersecting are replaced by requiring them to be far or close, respectively. We refer to the foundational article by Georgakopoulos and Papasoglu~\cite{GeorgakopoulosP25} for a broad introduction to the topic, and to~\cite{CoarseBalSeps,AhnGHK25,GridCounter,FatK2t,SmallCounter,CoarseHalin,AlbrechtsenHJKW24,AlbrechtsenJKW25,BergerS24,BonnetLPP26,ChangCTZ25,ChepoiDNRV12,Davies25,DaviesHH25,DaviesHIM24,Distel25,DistelGHLM26,EsperetGG25,HatzelP25,Hickingbotham25,AS1:CoarseTw,AS2:pwQI,AS3:FatTree,AS4:WeakMenger,AS5:pw,AS6:disc,AS0:StrongMenger} for a selection of recent advances. In this spirit, the natural coarse analogue of Menger's Theorem would read as follows.

\begin{cmc*}
	There exist functions $f,g\colon \N^2\to \N$ such that the following holds. Let $G$ be a graph, $S,T\subseteq V(G)$ be vertex subsets, and $d,k\in \N$. Suppose that one cannot find $k$ $S$-$T$-paths in $G$ that are pairwise at distance more than $d$ apart. Then there is a vertex subset $X\subseteq V(G)$ with $|X|\leq f(d,k)$ such that every $S$-$T$-path in $G$ is at distance at most $g(d, k)$ from some vertex of $X$.
\end{cmc*}

Coarse Menger's Conjecture was proposed independently by Georgakopoulos and Papasoglu~\cite{GeorgakopoulosP25} and by Albrechtsen, Huynh, Jacobs, Knappe, and Wollan~\cite{AlbrechtsenHJKW24}. The original phrasing postulated that $f(d, k)= k-1$, as in the classic Menger's Theorem. We will call this formulation the {\em{Strong Coarse Menger's Conjecture}}, while the formulation above, which allows $|X|$ to be bounded by any function of $d$ and $k$, we call the {\em{Weak Coarse Menger's Conjecture}}.

Due to the fundamental role of Menger's Theorem in the Graph Minors project, establishing its coarse analogue has been recognized as a pivotal step towards lifting the advances of Graph Minors to the coarse setting. It therefore came as an immense disappointment that the Coarse Menger's Conjecture is actually \textbf{false}, both in the strong~\cite{AS0:StrongMenger} and in the weak formulation~\cite{AS4:WeakMenger}, and even for $d=2$ and $k=3$. However, the statement holds in a number of important cases:
\begin{itemize}
	\item Strong Coarse Menger's Conjecture holds for $k=2$~\cite{AlbrechtsenHJKW24,GeorgakopoulosP25}.
	\item Weak Coarse Menger's Conjecture holds for $d=1$ in graphs of bounded degree~\cite{HendreyNST24,GartlandKL23}, and more generally in graphs excluding a topological minor. 
	\item In graphs of bounded treewidth, a standard argument shows that Weak Coarse Menger's Conjecture does in fact hold. Somewhat surprisingly, Strong Coarse Menger's Conjecture fails already in graphs of treewidth $6$~\cite{AS0:StrongMenger}, but holds in graphs of bounded pathwidth~\cite{AS5:pw}.
\end{itemize}
Recently, Nguyen, Scott, and Seymour~\cite{AS4:WeakMenger} hypothesized that the Coarse Menger's Conjecture might hold in planar graphs, or even in graphs embeddable in a fixed surface, and possibly even in the strong formulation. They also asked whether the Weak Coarse Menger's Conjecture could be true in graphs excluding a fixed minor (the construction of \cite{AS0:StrongMenger} shows that one cannot hope for the strong formulation in this setting). Towards these goals, in~\cite{AS6:disc} they proved that Strong Coarse Menger's Conjecture indeed holds in planar graphs when all the vertices of $S\cup T$ lie on a single face.

\paragraph*{Our contribution.}

In this work, we establish that Weak Coarse Menger's Conjecture holds in planar graphs, and more generally, in surface-embeddable graphs.


\begin{restatable}{theorem}{planarCMC}\label{thm:main}
	For every surface $\Sigma$ there exists a function $f\colon \N^2\to \N$ such that the following holds. Let $G$ be a graph embeddable in $\Sigma$, $S,T\subseteq V(G)$ be vertex subsets, and $k,d\in \N$. Suppose that one cannot find $k$ $S$-$T$-paths in $G$ that are pairwise at distance more than $d$ apart. Then there is a vertex subset $X\subseteq V(G)$ with $|X|\leq f(k,d)$ such that every $S$-$T$-path in $G$ is at distance at most $d$ from some vertex of $X$.
\end{restatable}

Note that we obtain the same upper bound on the distance between $S$-$T$-paths and $X$ in the dual (cut) case, as the lower bound on the distance between $S$-$T$-paths in the primal (flow) case; or equivalently, we have $g(d,k)=d$. As for the function $f$ bounding the size of $X$, our proof shows that
\begin{align*}
    &f(d, k)\in d^{\Oh(1)} \cdot 2^{\Oh(k \log k)} & \qquad& \textrm{when }\Sigma\textrm{ is the sphere, and}\\
    &f(d, k)\in 2^{d^{\Oh(1)} \cdot 2^{\Oh((k+g) \log (k+g)))}} & \qquad& \textrm{otherwise, where }g\textrm{ denotes the genus of }\Sigma.
\end{align*}
Let us point out that our proof techniques cannot yield a bound that is subexponential in $k,$ as they rely on the \emph{Unique Linkage Theorem} of Robertson and Seymour~\cite{GraphMinorsXXI,AdlerKKLST2017Irrelevant}, for which already on planar graphs there is a construction that gives a single-exponential lower bound~\cite{AdlerK2019Alower}.

We also remark that recently, Davies~\cite{Davies25} and independently Chang, Conroy, Tan, and Zheng~\cite{ChangCTZ25} proved that string graphs (intersection graphs of connected subsets of the plane) are quasi-isometric to planar graphs. Since the assertion postulated by the Coarse Menger's Conjecture is preserved under quasi-isometry, we obtain the following.

\begin{corollary}
	The class of string graphs satisfies the Weak Coarse Menger's Conjecture.
\end{corollary}

Finally, we remark that our proof of \cref{thm:main} can be easily turned into a fixed-parameter algorithm that given $G$, $S$, $T$, $k$, and $d$, returns one of the outcomes in time $\Oh_{k,d}(\|G\|^{\Oh(1)})$; here, $\|G\|$ is the number of edges of $G$ and the $\Oh_{k,d}(\cdot)$-notation hides factors that depend on $k$ and $d$. However, if one considers the two computational problems:
\begin{itemize}
	\item Given $G,S,T,k,d$, are there $k$ $S$-$T$-paths in $G$ that are pairwise at distance more than $d$ apart?
	\item Given $G,S,T,\ell,d$, is there a set $X\subseteq V(G)$ with $|X|\leq \ell$ such that every $S$-$T$-path in $G$ is at distance at most $d$ from some vertex of $X$?
\end{itemize}
then both these problems can be expressed by sentences of logic $\mathsf{FO}+\mathsf{sdp}$, proposed by Golovach, Stamoulis, and Thilikos~\cite{GolovachST23}, of length $\Oh_{k,d}(1)$ and $\Oh_{\ell,d}(1)$, respectively. Golovach et al. proved that the model-checking problem for $\mathsf{FO}+\mathsf{sdp}$ on graphs embeddable on a fixed surface is fixed-parameter tractable: given a graph $G$ embeddable on a surface $\Sigma$ and a sentence $\varphi$ of $\mathsf{FO}+\mathsf{sdp}$, it can be decided in time $\Oh_{\varphi,\Sigma}(\|G\|^2)$ whether $\varphi$ holds in $G$. By applying this to the sentences expressing the problems above, we conclude that they can be solved in time $\Oh_{k,d,\Sigma}(\|G\|^2)$ and $\Oh_{\ell,d,\Sigma}(\|G\|^2)$, respectively, on graphs embeddable on $\Sigma$.

\section*{Overview of the approach} 

We now give a brief and informal overview of the techniques used to establish \cref{thm:main}.
Throughout most of the overview we focus on the case of planar graphs; that is, when $\Sigma$ is the sphere. The nature of our techniques makes this the main case of interest, while the extension to more complicated surfaces is a technically non-trivial, but conceptually rather natural lift. This structure is also reflected in the later section: throughout most of the paper we work towards the proof of the planar case, which is concluded in \cref{sec:planar}, and then in \cref{sec:genus} we lift all the tools to graphs embeddable in more general surfaces. 

Hence, from now on we assume we are working with a graph $G$ embedded on the sphere, sets of vertices $S,T\subseteq V(G)$ which we shall call {\em{terminals}}, and parameters $d,k \in \N$. Our goal is to either find a {\em{$d$-scattered $S$-$T$-linkage}} of order $k$ in $G$ --- a family of $k$ $S$-$T$-paths that are pairwise at distance more than $d$ from each other --- or a set of $f(d,k)$ balls of radius $d$ in $G$ whose union intersects every $S$-$T$-path.

The main idea of the proof is that for each $Z\in \{S,T\}$, there will be a fundamental case distinction: whether $Z$ is ``localized'' or ``widespread'' on the embedding of $G$. Informally, by ``localized'' we mean that all the vertices of $Z$ are close to a single face of $G$, or maybe to a bounded number of faces of $G$. And by ``widespread'' we mean that the vertices of $Z$ can be found on many different faces of $G$ that are far from each other. The intuition is that if $Z$ is ``widespread'', then it should be very easy to route a scattered linkage to $Z$, essentially by routing it to different vertices of $Z$ that are pairwise far from each other on the embedding. On the other hand, if $Z$ is ``localized'', then we can analyze the local structure around the few faces to which $Z$ is close using standard packing/hitting arguments, developed for various Erd\H{o}s-P\'osa duality statements in the theory of Graph Minors.

\paragraph*{Bidimensionality.}

However, how do we quantify whether a set $Z\subseteq V(G)$ is ``localized'' or ``widespread'' in a planar graph~$G$? The answer comes via a parameter {\em{bidimensionality}}, introduced by Thilikos and Wiederrecht~\cite{ThilikosW2024Bidimensionality} in the context of the project of constructing an annotated, or colorful counterpart of the theory of Graph Minors. In essence, this project is about studying the topological structure in graphs relative to some set of annotated vertices $Z\subseteq V(G)$, or to several such annotated sets $Z_1,\ldots,Z_t\subseteq V(G)$. In our case, $S$ and $T$ would be the two colors of annotations.

Formally, the {\em{bidimensionality}} of a set of vertices $Z$ in a graph $G$ is the largest $k$ such that $G$ contain a minor-model of a $k\times k$ grid that is {\em{$Z$-rooted}}: the model (branch set) of every vertex of the grid contains a vertex of $Z$. An important example is the following: If $G$ is planar and all the vertices of $Z$ lie on one face, then the bidimensionality of $Z$ in $G$ is at most $2$. More generally, the following statement is roughly\footnote{The right-to-left implication is true. To make the left-to-right implication also true, one needs to add some additional assumption on the connectivity of the graph; for instance, that there is no separation of small order with both sides being large.} true: If $G$ is planar, then a set of vertices $Z\subseteq V(G)$ has bounded bidimensionality if and only if $Z$ can be covered by a bounded number of bounded-radius balls in the face-vertex metric. (Here, the face-vertex metric is the metric on the vertices where vertices lying on the same face are considered to be at distance~$1$.) Therefore, this parameter is a perfect tool for distinguishing the ``localized'' and the ``widespread'' case. Different arguments will be used depending on whether the bidimensionality of $S$ and $T$ in $G$ is small or large, leading to four (up to symmetry) scenarios: small/small (2 variants), small/large, large/large. We remark that bidimensionality, as a parameter, actually never appears in our proof, but keeping it in mind is instructive for guiding the intuition.
We refer to \cite{ProtopapasTW2026Colorful,GorskyPW2026Quickly} for a more in-depth discussion on the general structure of graphs equipped with vertex subsets of bounded bidimensionality.

The main ``hammer'' we use in our proof of \cref{thm:main} is a powerful structure theorem due to Paul, Protopapas, Thilikos, and Wiederrecht~\cite{PaulPTW2024Obstructions}. This result, for a given graph $G$ that excludes a fixed minor, has large treewidth, and is equipped with annotations $Z_1,\ldots,Z_t$, exposes an ``infrastructure'' in $G$ that witnesses, for each set $Z_i$, whether $Z_i$ has small or large bidimensionality. The result is stated in full formality as \cref{lst_fi} and requires a substantial amount of preliminaries, but let us describe it informally now. 

The infrastructure takes the form of a {\em{walloid}}, presented in~\cref{sec:renditions}: a huge cylindrical wall in $G$ with various features attached to it. In the case when $G$ is planar, the walloid can be divided into {\em{flap segments}} and {\em{vortex segments}}. Each flap segment features a {\em{pocket}}: a disk placed at the top side of the segment that is enclosed from above by a ``rainbow'' consisting of many disjoint paths. Note that thus, the pockets are pairwise far from each other in the face-vertex metric. Each vortex segment has attached to its top side a large cylindrical grid (called the {\em{nest}}), whose inner disk is the {\em{vortex}} of the segment. We require that the subgraph drawn in the vortex, even together with the whole nest, has bounded pathwidth. The statement of \cref{lst_fi} could be now informally summarized as follows: If a planar graph $G$ has large treewidth, and $Z_1,\ldots,Z_t$ are sets of annotated vertices in $G$, then $G$ contains a large walloid with a bounded number of vortex segments such that for each $i\in \{1,\ldots,t\}$, one of the following holds:
	\begin{description}
		\item \quad{\bf{Large bidimensionality:}} $Z_i$ is featured in the pockets of many distinct flap segments.
		\item \quad{\bf{Small bidimensionality:}} $Z_i$ is entirely contained in the union of all the vortices.
	\end{description}

\paragraph*{Working with a walloid.} In the context of the proof of \cref{thm:main},
we apply \cref{lst_fi} with $Z_1=S$ and $Z_2=T$. If we obtain the large/large case, then we may immediately find a large $d$-scattered $S$-$T$-linkage just using the exposed infrastructure, by connecting various pockets through the walloid. Importantly, the scatteredness of the obtained linkage is also guaranteed by the infrastructure: the rows and columns of the walloid provide separation between the constructed paths, thus witnessing that they are far apart. 

For the small/large and the small/small case, we have to analyze the situation around vortices. For every vortex $\Delta$, we apply classic ``harvesting'' arguments, known from working with Erd\H{o}s-P\'osa problems, to prove that for each $Z\in \{S,T\}$, we can either find many $d$-scattered paths that start in $\Delta\cap Z$ and escape the vortex and its nest (we shall call them informally {\em{escaping paths}}), or we can hit all such escaping paths using a bounded number of balls of radius $d$. Similarly, we can either find many $d$-scattered $S$-$T$-paths residing inside the vortex and its nest, or we can hit all such paths using a bounded number of balls of radius $d$. In both these claims, it is crucial that the vortex has bounded pathwidth.

Since the number of vortex segments is bounded, we may take the family of all the balls constructed for all the vortices as the candidate hitting set for $S$-$T$-paths. There are now four scenarios when this set may still not hit all $S$-$T$-paths:
\begin{description}
	\item \textbf{Small/small case 1:} For some vortex $\Delta$, we managed to construct a large $d$-scattered linkage of escaping paths that start in $\Delta\cap S$, as well as a large $d$-scattered linkage of escaping paths that start in $\Delta\cap T$. But then we may connect those paths within the walloid to form a large $d$-scattered $S$-$T$-linkage.
	\item \textbf{Small/small case 2:}  For two different vortices $\Delta,\Delta'$, we managed to construct a large $d$-scattered linkage of paths escaping $\Delta$ that start in $\Delta\cap S$, as well as a large $d$-scattered linkage of paths escaping $\Delta'$ that start in $\Delta'\cap T$. Again, we may connect those paths within the walloid to form a large $d$-scattered $S$-$T$-linkage.
	\item \textbf{Small/large case:} For some vortex $\Delta$, we managed to construct a large $d$-scattered linkage of escaping paths that start in $\Delta\cap S$, whereas $T$ is featured in the pockets of multiple flap segments. Then we may connect the escaping paths to the pockets within the walloid in order to form a large $d$-scattered $S$-$T$-linkage.
	\item \textbf{Large/small case:} Symmetric to the previous case with the roles of $S$ and $T$ swapped.
\end{description}

\begin{figure}[ht]
    \centering
    \includegraphics[scale=1.15]{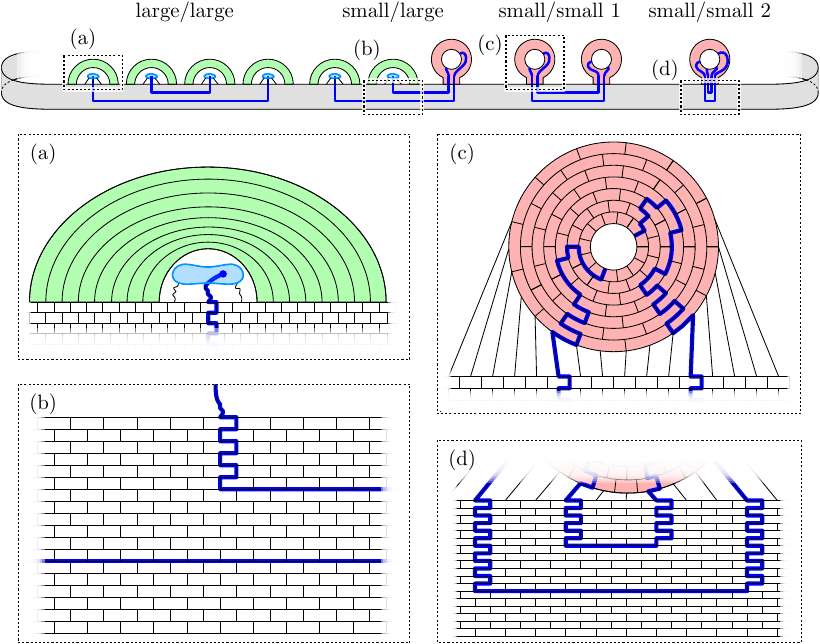}
    \caption{
        Overview of the case analysis in one walloid.
        (a) rainbow of a flap segment;
        (b) scattered $S$-$T$-linkage to another segment;
        (c) nest of a vortex segment;
        (d) scattered $S$-$T$-linkage in one vortex segment.
    }%
    \label{fig:intro}
\end{figure}

There is one important technical caveat. When connecting the linkages of escaping paths through the walloid, in order to ensure that the final linkage is $d$-scattered, we need to have a good control over how the escaping linkages exactly connect to the main cylindrical wall of the walloid. To achieve this, we use the Linkage Combing Lemma of Golovach, Stamoulis, and Thilikos~\cite{GolovachST23Combing}. Intuitively, this result allows us to assume that when an escaping linkage travels through the nest of the corresponding vortex, from some point on it agrees with the radial paths of the nest.

\paragraph*{Further complications.}

As the experienced reader probably suspects, in the description above we have cheated multiple times. The largest liberty we have taken is that \cref{lst_fi} does not really just assume that $G$ has large treewidth, but it provides a local decomposition of the graph relative to some tangle of huge order, controlled by a huge wall in the graph. And it is not really the case that vortices have bounded pathwidth, but they admit path decompositions of small adhesion where the bags can be large, but are separated from the tangle by small separators (consisting of the two adhesions incident to the bag). In fact, the reasoning presented above provides us with a certain Local Structure Theorem: we may either find a large $d$-scattered $S$-$T$-linkage in $G$, or a set consisting of a bounded number of radius-$d$ balls that hits every $S$-$T$-path, except possibly for those that are entirely hidden in the bags of vortices. By standard means, we then turn this Local Structure Theorem into a suitable Global Structure Theorem, which provides a tree decomposition of the whole graph where every bag is equipped with a hitting set (consisting of a bounded number of balls of radius $d$) for $S$-$T$-paths that are ``relevant'' for this bag. The final stroke in the proof of the planar case of \cref{thm:main} is a standard greedy bottom-up argument applied on this tree decomposition.

\paragraph*{Beyond planar graphs.} Finally, we discuss the techniques necessary to lift the reasoning presented above from the planar case to the case of graphs embeddable in some fixed surface $\Sigma$ of positive genus.

In case $G$ has an embedding on $\Sigma$ with large {\em{representativity}}, \cref{lst_fi} provides a walloid that additionally contains segments representing all the topological features (handles and cross-caps) of $\Sigma$.
These segments can be essentially ignored in the further analysis, and the whole argumentation goes through with minimal changes. In fact, in our proof we essentially show how to reduce this setting to the planar case, by finding a disk $\Delta$ in the surface that contains all the relevant features of the walloid, so that the whole reasoning can be performed within $\Delta$.

However, a priori we do not know that the embedding in the surface has large representativity. To deal with embeddings of small representativity, we inductively cut the embedding along non-contractible curves of small length, effectively splitting the graphs into multiple components, each embeddable on a surface of Euler-genus strictly less than that of $\Sigma.$
This approach will inevitably introduce a set of {\em{apices}} $A \subseteq V(G)$. The apices are not embedded on the surface and can have arbitrary neighborhoods within the vortices (and nowhere else), but their number $|A|$ is bounded.
Note that in principle, the existence of apices might influence the distances between the components from within the vortices, and even create short connections between two different vortices of the same component.
Apices have to be therefore treated carefully when analyzing the structure around every vortex, but at least they do not influence the distances outside the nests of the vortices.
However, at the end, since the number of apices is bounded, we may always just include the radius-$d$ balls around them in the constructed hitting set. Consequently, in all the harvesting arguments applied in the vicinity of the vortices, we only harvest paths disjoint from the balls around the apices.

Although this approach may appear conceptually simple at first glance, the technical challenges, particularly around the cutting of the surface, become somewhat complex.

\paragraph*{Independent work.} We have been informed by Chun-Hung Liu that he has independently obtained a proof of Weak Coarse Menger's Conjecture for $H$-minor-free graphs, for any fixed graph $H$; see~\cite{Liu26}.
To the best of our understanding, while there are large high-level similarities between the two approaches, Liu's proof appears to rely primarily on classical Graph Minor Structure Theory, whereas ours builds on ``modern'' techniques based on the bidimensionality of terminal sets.

\paragraph*{Acknowledgements.} The authors thank Jadwiga Czy\.zewska for inspiring discussions during the initial phases of this project.

\section{Preliminaries}\label{sec:prelims}

We use $\N$ to denote the set of non-negative integers, and for $c \in \N$, we write $\N_{\geq c}\coloneqq\{ x \in \N \mid x \geq c\}$.
For $a, b \in \N$, we denote by $[a, b]$ the set $\{ x \in \N \mid a \leq x \leq b\}$.
Notice that $[a, b]$ is empty when $b < a$.
For $c \in \N$, we use $[c]$ as a shorthand for $[1, c]$.

We use standard graph notation and terminology.
For a graph $G$, we denote $|G|\coloneqq |V(G)|$ and $\|G\|\coloneqq |E(G)|$.
Given a subgraph $H$ of $G$ and a vertex subset $X \subseteq V(G),$ we say that \emph{$H$ avoids $X$} if $V(H) \cap X = \emptyset.$
We extend this definition to a collection $\Hh$ of sugraphs of $G$ in the expected way: \emph{$\Hh$ avoids $X$} if for every $H \in \Hh,$ $H$ avoids $X.$
The distance (shortest path) metric in $G$ will be denoted by $\dist_G(\cdot,\cdot)$. We extend the distance metric to subgraphs of $G$ in the obvious way: for two subgraphs $H,H'\subseteq G$, we define $\dist_G(H,H')\coloneqq \min\{\dist_G(u,u')\colon u\in V(H), u'\in V(H')\}$. Finally, for a vertex $u$ and a distance parameter $d\in \N$, the radius-$d$ ball around $u$ in $G$ is the set
\[\Ball_G^d(u)\coloneqq \{v\in V(G)\mid \dist_G(u,v)\leq d\}.\]
We extend this notation to vertex subsets in $G$ by setting
\[\Ball_G^d(A)\coloneqq \bigcup_{u\in A}\Ball_G^d(u),\qquad \textrm{for }A\subseteq V(G).\]

Let us introduce some standard concepts from graph theory that will be used throughout the paper.

\paragraph{Separations and linkages.}

Let $G$ be a graph.
A \emph{separation} of $G$ is a pair $(A,B)$ of subsets of $V(G)$ such that $V(G)=A\cup B$ and there is no edge with one endpoint in $A\setminus B$ and the other in $B\setminus A$. The quantity $|A \cap B|$ is the \emph{order} of $(A, B)$.
Given sets $X, Y \subseteq V(G)$, a set $S$ is an \emph{$X$-$Y$-separator} if no component of $G-S$ contains both a vertex of $X$ and a vertex of $Y$.

For $x, y \in V(G),$ a path $P$ in $G$ is said to be an \emph{$x$-$y$-path} if the endpoints of $P$ are $x$ and $y.$
Similarly, for $X, Y \subseteq V(G),$ a path $P$ in $G$ is said to be an \emph{$X$-$Y$-path} if $P$ is an $x$-$y$-path for some $x\in X$ and $y\in Y.$
A \emph{linkage} in $G$ is a family of pairwise vertex-disjoint paths in $G$, and an \emph{$X$-$Y$-linkage} is one that consists of $X$-$Y$-paths.
The \emph{order} of a linkage $\Pp$ is the number $|\Pp|$ of paths in $\Pp.$ We sometimes slightly abuse the notation and we identify a linkage $\Pp=\{ P_1,P_2,\dots,P_k\}$ and the graph $\bigcup_{i\in[k]}P_i.$

We say that a path $P$ in $G$ is \emph{internally disjoint} from a set $X \subseteq V(G)$ if $V(P) \cap X$ does not contain any vertex that is not an endpoint of $P$.

\paragraph{Tree-decompositions.}

A \emph{rooted tree} is a tree with one distinguished vertex called the \emph{root}. This naturally imposes the parent/child and ancestor/descendant relations on the vertices of the tree. We follow the convention that whenever a rooted tree $T$ underlies a decomposition of some graph $G$, then the vertices of $T$ are called \emph{nodes}.

A \emph{rooted tree-decomposition} of a graph $G$ consists of a rooted tree $T$ and a bag function $\beta$ mapping every node $t$ of $T$ to its \emph{bag} $\beta(t)\subseteq V(G)$ so that the following conditions are satisfied:
\begin{itemize}
	\item for every edge $e$ of $G$, there exists a node $x$ of $T$ such that $\beta(x)$ contains both the endpoints of $e$; and
	\item for every vertex $u$ of $G$, the set of nodes of $T$ whose bags contain $u$ induces a non-empty, connected subtree of $T$.
\end{itemize}
The \emph{adhesion} of a tree decomposition $(T,\beta)$ is the maximum size of an intersection of two neighboring bags, that is, $\max_{st\in E(T)} \beta(s)\cap \beta(t)$.

\paragraph{Grids and walls.} 

An \emph{$(n \times m)$-grid} is the graph with vertex set $[n] \times [m]$ and edge set
$$\{ (i,j)(i,j+1) \colon i \in [n], j \in [m-1] \} \cup \{ (i,j)(i+1,j) \colon i \in [n-1], j \in [m] \}.$$
We call the path where vertices appear as $(i,1), (i,2), \ldots, (i,m)$ the \emph{$i$-th row} and the path where vertices appear as $(1,j), (2,j), \ldots, (n,j)$ the \emph{$j$-th column} of the grid.

Let  $t, z \in \N_{\geq 3}.$
An \emph{elementary $(t \times z)$-wall} is obtained from the $(t \times 2z)$-grid by removing a matching $M$ which contains all odd numbered edges of its odd numbered columns and all even numbered edges of its even numbered columns, and then deleting all vertices of degree one, see \cref{fig:wall:grid,fig:wall:wall}.
We say that a graph $G'$ is a \emph{subdivision} of a graph $G$ if it can be obtained from $G$ by iteratively replacing any edge $uv$ of $G$ by a $u$-$v$-path whose internal vertices are disjoint from $G$.
A \emph{$(t \times z)$-wall} is a subdivision of an elementary $(t \times z)$-wall.
By a wall $W$ in a graph $G$ we mean a wall $W$ that is a subgraph of $G$.

Notice that any $(t \times z)$-wall $W$ consists of $t$ pairwise-disjoint paths $P_1, \dots, P_t$ and $z$ pairwise-disjoint paths $Q_1, \dots, Q_z$ such that $W = \bigcup_{i\in[t]}P_i \cup \bigcup_{i\in[z]}Q_i,$ $P_i\cap Q_j$ is a non-empty path for all $i\in[t]$ and $j\in[z],$ each $P_i$ meets the $Q_j$s in order, and each $Q_j$ meets the $P_i$s in order; see \cref{fig:wall:paths}.
We call the $P_i$s the \emph{horizontal paths} of $W$ and the $Q_j$s the \emph{vertical paths} of $W.$

\begin{figure}[h]
    \centering
    \begin{subfigure}[t]{0.33\textwidth}
        \centering
        \includegraphics[scale=1.1,page=1]{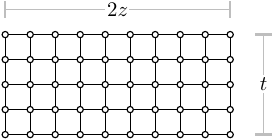}
        \caption{$(t \times 2z)$-elementary grid}\label{fig:wall:grid}
    \end{subfigure}
    \begin{subfigure}[t]{0.33\textwidth}
        \centering
        \includegraphics[scale=1.1,page=3]{wall.pdf}
        \caption{$(t \times z)$-elementary wall}\label{fig:wall:wall}
    \end{subfigure}
    \begin{subfigure}[t]{0.32\textwidth}
        \centering
        \includegraphics[scale=1.1,page=4]{wall.pdf}
        \caption{paths $P_i$ in red and $Q_j$ dotted}\label{fig:wall:paths}
    \end{subfigure}
    \caption{
        Elementary grid is used to create a $(t \times z)$-elementary wall that has horizontal paths $P_i$ and vertical paths $Q_j$.
        Grids and walls are any subdivisions of elementary grids and walls.
    }\label{fig:wall}
\end{figure}

\paragraph{Well-linked sets.}

A set $X \subseteq V(G)$ in a graph $G$ is \emph{well-linked} if for every separation $(A, B)$ of $G,$ we have that \[|A \cap X| \leq |A \cap B|\qquad\textrm{or}\qquad |B \cap X| \leq |A \cap B|.\]
Suppose $|X| \geq 3k$ for some $k\in \N$.
Then, given a separation $(A, B)$ of $G$ of order less than $k,$ we have that $|A\cap X|<k$, implying $|B\cap X|>2k$, or that $|B\cap X|<k$, implying $|A\cap X|>2k$. In the former case we say that $B$ is the \emph{$X$-majority side} of $(A,B)$, and in the latter case $A$ is the $X$-majority side. Note that these two conditions are mutually exclusive, so exactly one of them holds.

We also need to be able to detect a large wall in a graph that ``respects'' a given well-linked set of large order in the graph.
This was done by Thilikos and Wiederrecht~\cite{ThilikosW2024Excluding} by turning a proof of Kawarabayashi, Thomas, and  Wollan~\cite{KawarabayashiTW2020Quickly} into an algorithm.
We need the following definitions.

Given a $(t \times z)$-wall $W$ in a graph $G$, for some $t, z \in \N_{\geq 3},$ we can observe that for every separation $(A, B)$ of $G$ of order less than $\min \{ t, z \}$ there is a unique side, say $B,$ such that $B \setminus A$ contains the whole vertex set of both a vertical and a horizontal path of $W.$
We call this the \emph{$W$-majority side} of $(A, B).$
With this observation in mind, we say that a $(t \times z)$-wall $W$ of a graph $G$ is \emph{controlled} by a well-linked set $X$ of $G$ if for every separation $(A, B)$ of $G$ of order less than $\min \{ t, z \},$ $B$ is the $W$-majority side of $(A, B)$ if and only if $B$ is the $X$-majority side of $(A, B)$ as well.

With these definitions in place, we can state the result of Thilikos and Wiederrecht~\cite{ThilikosW2024Excluding}.

\begin{proposition}[Thilikos and Wiederrecht {\cite[Theorem 4.2.]{ThilikosW2024Excluding}}]\label{thm_algogrid}
	There exist constants $c,d\in \N$ and an algorithm that, given a graph $G$, an integer $k\geq 3$, and a well-linked set $X \subseteq V(G)$ of size at least $c\cdot k^{20}$, computes in time $2^{\Oh(k^{d})} \cdot |G|^2 \|G\| \log|G|$ a $(k\times k)$-wall $W$ in $G$ controlled by $X.$
\end{proposition}

\section{Renditions of planar graphs in the sphere}\label{sec:renditions}

For notational simplicity we assume that every considered planar graph $G = (V, E)$ is associated with some fixed embedding in the sphere $\Sigma$.
Moreover, we do not distinguish between a vertex in $V$ and the point of $\Sigma$ used to draw it, or between an edge in $E$ and the arc in $\Sigma$ representing it in the embedding.

A \emph{closed} (respectively \emph{open}) \emph{disk} $\Delta$ in $\Sigma$ is any set homeomorphic to $\{ (x, y)\in \mathbb{R}^2 \mid x^{2} + y^{2} \leq 1\}$ (respectively, to $\{ (x, y)\in \mathbb{R}^2 \mid x^{2} + y^{2} < 1\}$).
Whenever we say disk without specifying whether it is open or closed, we always mean a closed disk.
For a disk $\Delta$, we write
\begin{itemize}
\item $\bd(\Delta)$ for the {\em{boundary}} of $\Delta$ in $\Sigma$;
\item $\inte(\Delta)\coloneqq \Delta\setminus \bd(\Delta)$ for the {\em{interior}} of $\Delta$; and
\item $\exte(\Delta)\coloneqq \Sigma\setminus \Delta$ for the {\em{exterior}} of $\Delta$. 
\end{itemize}

Any disk $\Delta \subseteq \Sigma$ we consider throughout the paper will also satisfy $\bd(\Delta) \cap E = \emptyset$; that is, the boundary of $\Delta$ intersects $G$ only in vertices.
We shall always assume this to be the case without any further clarification.

\paragraph{$\Sigma$-renditions.}

A \emph{$\Sigma$-rendition} of a planar graph $G = (V, E)$ is a set $\rend$ of disks in $\Sigma$ satisfying the following:
\begin{enumerate}
    \renewcommand{\theenumi}{\roman{enumi}} 
    \renewcommand{\labelenumi}{\theenumi)}
	\item The disks of $\rend$ have pairwise disjoint interiors.
	\item If $\Delta_{1}, \Delta_{2} \in \rend$ are distinct, then $\Delta_{1} \cap \Delta_{2} \subseteq V.$
	\item Every edge of $G$ (except possibly its endpoints) is a subset of the interior of one of the disks in $\rend.$
	\item\label{c:periphery} For every disk $\Delta \in \rend$, we have $|\bd(\Delta) \cap V| \leq 2.$
\end{enumerate}

The \emph{ground} vertices of $\rend$ are those vertices of $G$ that do not belong to the interior of any disk of $\rend.$
By $\ground(\rend)$ we denote the set of all ground vertices of $\rend$; the graph $G$ will be always clear from the context.
The disks of $\rend$ will be often called the \emph{cells} of $\rend$.
For a cell $\Delta\in \rend$, we define the {\em{periphery}} of $\Delta$ to be the set $\per(\Delta)\coloneqq \bd(\Delta)\cap V=\bd(\Delta)\cap \ground(\rend)$.
Thus, condition \ref{c:periphery} asserts that every cell has a periphery of size at most $2.$

\begin{figure}[ht]
	\centering
	\begin{subfigure}{0.32\textwidth}
		\centering
		\includegraphics[page=1,scale=1.1]{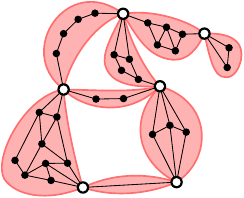}
		\caption{}\label{subfig:renditions_ok}
	\end{subfigure}
	\begin{subfigure}{0.32\textwidth}
		\centering
		\includegraphics[page=2,scale=1.1]{renditions.pdf}
		\caption{}\label{subfig:renditions_mistakes}
	\end{subfigure}
	\begin{subfigure}{0.32\textwidth}
		\centering
		\includegraphics[page=3,scale=1.1]{renditions.pdf}
		\caption{}\label{subfig:renditions_grounded}
	\end{subfigure}
	\caption{
		A graph $G$ drawn on the sphere $\Sigma$ and
		two families of closed disks in $\Sigma$.
		\textbf{(\subref{subfig:renditions_ok})}
		This family is a $\Sigma$-rendition of $G$.
		\textbf{(\subref{subfig:renditions_mistakes})}
		This family breaks every condition for being a $\Sigma$-rendition of $G$.
		\textbf{(\subref{subfig:renditions_grounded})}
        A $\rend$-aligned disk $\Delta$, a $\rend$-grounded red cycle, and a $\rend$-grounded blue dotted path.
	}%
	\label{fig:renditions}
\end{figure}

Observe that every planar graph $G$ has a trivial $\Sigma$-rendition where every vertex is a ground vertex and every edge belongs to its own cell with periphery of size $2$.

\medskip
For the remainder of this section, we fix a planar graph $G=(V,E)$ and 
 a $\Sigma$-rendition $\rend$ of $G$.

\paragraph{Aligned disks and grounded subgraphs.}

We call a disk $\Delta$ in $\Sigma$ \emph{$\rend$-aligned} if $\bd(\Delta) \cap V \subseteq \ground(\rend)$.
In particular, the cells of $\rend$ are $\rend$-aligned by definition.
Given a $\rend$-aligned disk $\Delta$, we define its \emph{periphery} as $\per(\Delta)\coloneqq \bd(\Delta)\cap V$ and we denote by $G \cap \Delta$ the subgraph of $G$ consisting of all the vertices and edges drawn in $\Delta$.

Let $Q \subseteq G$ be either a cycle or path in $G$.
We say that $Q$ is \emph{$\rend$-grounded} if either $Q$ is a path of positive length with both endpoints in $\ground(\rend)$, or $Q$ is a cycle that contains edges of $G \cap \Delta_{1}$ and $G \cap \Delta_{2}$ for at least two distinct cells $\Delta_{1}, \Delta_{2} \in \rend$.
A $2$-connected subgraph $H$ of $G$ is said to be \emph{$\rend$-grounded} if every cycle in $H$ is $\rend$-grounded.

\paragraph{Linear decompositions and depth.}

Let $\Delta$ be a $\rend$-aligned disk in $\Sigma$.
Further, let $\langle x_1, x_2, \dots, x_{n-1}, x_n\rangle$ be the enumeration of the vertices of $\per(\Delta)$ in the order of encountering them when traversing $\bd(\Delta)$ in the clockwise direction, with the first vertex $x_1$ chosen arbitrarily from $\per(\Delta)$. (Here, we assume that we have fixed an orientation of $\Sigma$.)

A \emph{linear decomposition} of $G \cap \Delta$ is a sequence $\langle B_1, B_2, \dots, B_{n-1}, B_n \rangle$ of sets such that
\begin{itemize}
	\item[(V1)] $B_i \subseteq V(G\cap \Delta)$ and $x_i \in B_i$ for all $i \in [n]$;
	\item[(V2)] $\bigcup_{i \in [n]} B_i = V(G \cap \Delta)$;
	\item[(V3)] for every edge $uv \in E(G \cap \Delta)$, there exists $i \in [n]$ such that $u,v \in B_i$; and
	\item[(V4)] for every vertex $v \in V(G \cap \Delta)$, the set $\{ i \in [n] \mid v \in B_i\}$ forms an interval in $[n]$.
\end{itemize}
The \emph{adhesion} of a linear decomposition is the quantity $\max\{ |B_i \cap B_{i+1}| : i \in[n-1] \}$.
The \emph{width} of a linear decomposition is the quantity $\max\{ |B_i| : i \in [n] \}$.
The \emph{depth} of $\Delta$ is the smallest integer $k$ such that $G \cap \Delta$ has a linear decomposition of adhesion at most $k$ and the \emph{width} of $\Delta$ is the smallest integer $k$ such that $G \cap \Delta$ has a linear decomposition of width at most $k$.

\paragraph{Radial linkages and (railed) nests.}

Finally, we introduce some terminology for grid-like objects. Let us fix two disks $\Delta^{\insf} \subseteq \Delta^{\outsf}$ in $\Sigma$, not necessarily $\rend$-aligned.

We say that a linkage $\Pp$ in $G$ is a \emph{$(\Delta^{\insf}, \Delta^{\outsf})$-radial linkage} if each path in $\Pp$ has one endpoint drawn in $\Delta^{\insf}$ and the other drawn in $\Sigma \setminus \inte(\Delta^{\outsf}),$ see \Cref{fig:radial_linkages}.
We moreover say that $\Pp$ is 
\begin{itemize}
	\item \emph{confined} if every path in $\Pp$ is disjoint from $\exte(\Delta^{\outsf})$, and
	\item \emph{fully confined} if it is confined and, furthemore, every path in $\Pp$ is disjoint from $\inte(\Delta^{\insf}).$
\end{itemize} 
Note that if $\Pp$ is confined, then one of the two endpoints of every path in $\Pp$ lies in $\bd(\Delta^{\outsf});$ and if $\Pp$ is fully confined, then additionally the other endpoint must lie in $\bd(\Delta^{\insf}).$ Finally, for a set $X\subseteq V(G)$, we say that a radial linkage $\Pp$ is {\em{$X$-rooted}} if for each $P\in\Pp$, the endpoint of $P$ in $\Delta^\insf$ belongs to $X$.

\begin{figure}[h]
    \centering
    \begin{minipage}{.45\textwidth}
        \centering
        \includegraphics[scale=1.1]{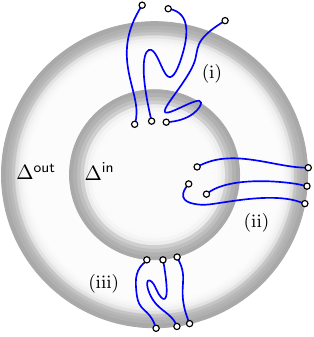}
        \captionof{figure}{
            (i) A $(\Delta^{\insf}, \Delta^{\outsf})$-radial linkage that is moreover,
            (ii) confined,
            (iii) fully confined.
        }\label{fig:radial_linkages}
    \end{minipage}
    ~~
    \begin{minipage}{.45\textwidth}
        \centering
        \includegraphics[scale=1.1]{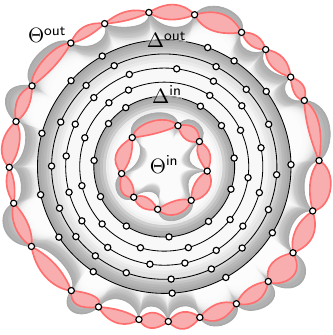}
        \captionof{figure}{
            A $(\Delta^{\insf}, \Delta^{\outsf})$-nest of order $5$ sandwiched by a pair of $\rend$-aligned disks $(\Theta^{\insf}, \Theta^{\outsf}).$
        }\label{fig:nest}
    \end{minipage}
\end{figure}

Next, a \emph{$(\Delta^{\insf}, \Delta^{\outsf})$-nest} of \emph{order} $s \in \N_{\geq 2}$ is a sequence $\Cc = (C_{1}, \ldots, C_{s})$ of pairwise disjoint cycles in $G$ such that for every $i \in [s]$, $C_{i}$ bounds a disk $\Delta_{i}$ in $\Sigma$ in such a way that $\Delta^{\insf} = \Delta_{1} \subseteq \Delta_{2} \subseteq \ldots \subseteq \Delta_{s} = \Delta^{\outsf}$, see \Cref{fig:nest}.
Note that we insist that $\Delta^{\insf}$ is bounded by the cycle $C_1$ and $\Delta^{\outsf}$ is bounded by the cycle $C_s$, so in this case $\Delta^{\insf}$ and $\Delta^{\outsf}$ will {\em{not}} be $\rend$-aligned.
Also, $\Delta^{\insf}$ is the disk bounded by $C_1$ that does not contain $C_s$ and $\Delta^{\outsf}$ is the disk bounded by $C_s$ that contains $C_1$, hence the pair $(\Delta^{\insf},\Delta^{\outsf})$ is uniquely defined by the nest $\Cc$.
Therefore, we often speak about nests without specifying the disks, as they are implied, and we call $\Delta^\insf$ and $\Delta^\outsf$ the {\em{inner disk}} and the {\em{outer disk}} of $\Cc$ respectively.

We say that a $(\Delta^{\insf}, \Delta^{\outsf})$-radial linkage $\Pp$ is \emph{orthogonal} to a $(\Delta^{\insf}, \Delta^{\outsf})$-nest $\Cc=(C_1,\ldots,C_s)$ if for every $P \in \Pp$ and every $i \in [s]$, $C_{i} \cap P$ consists of a single connected component (which is then a subpath of $C_i$ and a subpath of $P$).
Whenever $\Pp$ is orthogonal to $\Cc$, $\Cc$ is of order $s \in \N_{\geq 2}$, and $\Pp$ is of order $p \in \N_{\geq 1}$, we call the pair $(\Cc, \Pp)$ a \emph{$(\Delta^{\insf}, \Delta^{\outsf})$-railed nest} of \emph{order} $(s, p)$.
If $\Pp$ is (fully) confined, then we call $(\Cc, \Pp)$ \emph{(fully) confined} as well; and if $\Pp$ is $X$-rooted, then we call $(\Cc,\Pp)$ $X$-rooted as well.

It will be sometimes convenient to speak about railed nests delimited by $\rend$-aligned disks. For this, we consider the following definition: for a pair of $\rend$-aligned disks $\FDisk^\insf\subseteq \FDisk^\outsf$, we say that a $(\Delta^{\insf}, \Delta^{\outsf})$-railed nest $(\Cc,\Pp)$ is \emph{sandwiched} by $(\FDisk^\insf,\FDisk^\outsf)$ if
\begin{itemize}
	\item $\FDisk^\insf\subseteq \Delta^\insf$ and $\bd(\FDisk^{\insf}) \cap \bd(\Delta^{\insf}) = \emptyset$;
	\item $\Delta^\outsf\subseteq \FDisk^\outsf$ and $\bd(\Delta^{\outsf}) \cap \bd(\FDisk^{\outsf})=\emptyset$; and
	\item $\Pp$ is a $(\FDisk^\insf, \FDisk^\outsf)$-radial linkage.
\end{itemize}

\section{Structure relative to terminals}

\subsection{Segments and walloids}\label{sec:walloids}

We first revisit the definition of a wall in order to define a slightly more general notion. This notion will be one of the building blocks towards a generalized structure that can appropriately model the infrastructure we need for our $\Sigma$-renditions.

\paragraph{Wall segments.}

Let $r, t \in \N_{\geq 4}$.
An \emph{elementary $(r \times t)$-wall segment} is obtained from the $(r \times 2t)$-grid by removing the matching $M$ consisting of all odd-numbered edges of its odd-numbered columns and all even-numbered edges of its even-numbered columns, see \cref{fig:wall_segment}.
Notice that any elementary $(r \times t)$-wall segment $W_{1}$ consists of $r$ pairwise disjoint paths $P_{1}, \ldots, P_{r}$ --- the \emph{horizontal paths} of $W_{1}$ --- and $t$ pairwise disjoint paths $Q_{1}, \ldots, Q_{t}$ --- the \emph{vertical paths} of $W_{1}$ --- which are defined in the obvious way from the unique $(r \times t)$-wall subgraph $W$ of $W_{1},$ so that each $P_{i},$ $i \in [r],$ respectively $Q_{j},$ $j \in [r],$ contains the $i$-th horizontal path of $W,$ respectively the $j$-th horizontal path of $W.$
We refer to $P_{1}$ as the \emph{top} and to $P_{r}$ as the \emph{bottom} horizontal path of $W_{1}$ respectively.
The \emph{top} and \emph{bottom boundary vertices} of $W_{1}$ are the vertices of the original $(r \times 2t)$-grid that are incident to edges of $M$ and belong to the top and bottom path of $W_{1}$ respectively.
For $i \in [t]$, the $i$-th \emph{top boundary vertex} of $W_{1}$ is the single vertex in $P_{1} \cap Q_{i}$ that is matched in $M$ and the $i$-th \emph{bottom boundary vertex} of $W_{1}$ is the single vertex in $P_{r} \cap Q_{i}$ that is matched in $M$.
For $i \in [r]$, the $i$-th \emph{left boundary vertex} of $W_{1}$ is the vertex $(i, 1)$ and the $i$-th \emph{right boundary vertex} of $W_{1}$ is the vertex $(i, 2t)$, where we identify the vertex set of $W_1$ with $[r]\times [2t]$ naturally.

\paragraph{Annulus walls.}

Let $r, t \in \N_{\geq 4}$.
An \emph{elementary $(r \times t)$-annulus wall} $W$ is obtained from an elementary $(r \times t)$-wall segment $W_{1}$ by adding an edge connecting the $i$-th left and the $i$-th right boundary vertex of $W_{1}$, for each $i \in [r]$, see \cref{fig:annulus_segment}.
Notice that the top and bottom boundary vertices of $W$ are the top and bottom boundary vertices of $W_{1}$.
Moreover, the $t$ horizontal paths $P_{1}, \ldots, P_{r}$ of $W_{1}$ are completed into $r$ cycles $C_{1}, \ldots, C_{r}$ of $W$ such that $V(P_{i}) = V(C_{i})$, for each $i \in [r]$.
We call these cycles the \emph{base cycles} of $W$.
We call the cycle $C_{1}$ the \emph{inner cycle} and the cycle $C_{r}$ the \emph{outer cycle} of $W$ respectively.
An \emph{$(r \times t)$-annulus wall} is a subdivision of an elementary $(r \times t)$-annulus wall.
When $r = t$, we simply write \emph{$t$-annulus wall}.

\paragraph{Flap segments.}

Let $r, t \in \N_{\geq 4}$.
An \emph{elementary $(r \times t)$-flap segment} of \emph{arity} $q \in [3]$ is obtained from an elementary $(r \times (2t + q))$-wall segment $W_{1}$ with top boundary vertices $v_{1}, \ldots, v_{t}$, $t_{1}, \ldots, t_{q}$, $u_{1}, \ldots, u_{t}$ in left to right order by adding all the edges $\{ v_i u_{t-i+1}\mid i\in [t] \}$, see \cref{fig:flap_segment}.
An \emph{$(r \times t)$-flap segment} $W$ is a subdivision of an elementary $(r \times t)$-flap segment.

We call the vertices $t_{1}, \ldots, t_{q}$ the \emph{sockets} of $W$.
Notice that the added edges form a $\{v_{1}, \ldots, v_{t}\}$-$\{u_{1}, \ldots, u_{t}\}$-linkage in $W$ which we call the \emph{rainbow} of $W$.
Further, we define the \emph{pocket cycle} of $W$, denoted $C^\holesf_W$, as the cycle of $W$ that consists of the (subdivision of the) edge $v_tu_1$ from the rainbow and the $v_t$-$u_1$-path that is a subpath of the top path of $W_{1}$. The {\em{pocket}} of $W$, denoted $\Delta^\holesf_W$, is the face of $W$ bounded by $C^\holesf_W$.

\paragraph{Vortex segments.}

Let $r,t \in \N_{\geq 4}$.
An \emph{elementary $(r \times t)$-vortex segment} $W$ is obtained from the disjoint union of an elementary $(r \times t)$-wall segment $W_{1}$ with top boundary vertices $v_{1}, \ldots, v_{t}$ in left to right order, and an elementary $((t + 2) \times t)$-annulus wall $W_{2}$ with bottom boundary vertices $u_{1}, \ldots, u_{t}$ in left to right order, by making $v_{i}$ adjacent to $u_{i}$ for each $i \in [t]$, see \cref{fig:vortex_segment}.
An \emph{$(r \times t)$-vortex segment} $W$ is a subdivision of an elementary $(r \times t)$-vortex segment.

We refer to the inner (resp. outer) cycle of $W_{2}$ as the \emph{inner cycle} (resp. \emph{outer cycle}) of $W$.
Also, let 
\begin{itemize}
	\item $\Cc$ be the set of base cycles of $W_{2},$ excluding the two extremal ones: the inner and the outer cycle; and
	\item $\Rr$ be a set of paths, each obtained from a distinct column of $W_{2}$, say with endpoint $u_{i}$, by appending the (subdivision of the) edge $v_{i}u_{i}.$
\end{itemize} 
Clearly, $(\Cc, \Rr)$ is a railed nest of order $(t, t)$.
We call this nest the {\em{nest}} of $W$.

\begin{figure}[p]
    \newcommand{\herevspacing}{\par\vspace{3em}}
	\centering
	\begin{minipage}{.5\textwidth}
		\centering
		\includegraphics[page=1,scale=1.1]{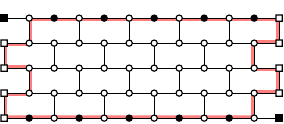}
		\captionof{figure}{A $(5 \times 6)$-wall segment}%
		\label{fig:wall_segment}
	\end{minipage}%
	\begin{minipage}{.5\textwidth}
		\centering
		\includegraphics[page=2,scale=1.1]{basic_segments.pdf}
		\captionof{figure}{A $(5 \times 6)$-annulus wall}%
		\label{fig:annulus_segment}
	\end{minipage}
    \herevspacing
	\centering
	\begin{minipage}{.6\textwidth}
		\centering
		\includegraphics[page=3,scale=1.1]{basic_segments.pdf}
		\captionof{figure}{A $(5 \times 4)$-flap segment of arity 3}%
		\label{fig:flap_segment}
	\end{minipage}%
	\begin{minipage}{.4\textwidth}
		\centering
		\includegraphics[page=4,scale=1.1]{basic_segments.pdf}
		\captionof{figure}{A $(4 \times 5)$-vortex segment}%
		\label{fig:vortex_segment}
	\end{minipage}
    \herevspacing
    \centering
    \includegraphics[page=1,scale=1.08]{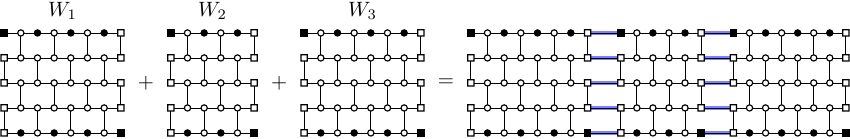}
    \caption{Concatenation of $W_1$, $W_2$, and $W_3$.}\label{fig:segment_concatenation}
    \herevspacing
    \centering
    \includegraphics[page=1,scale=1.08]{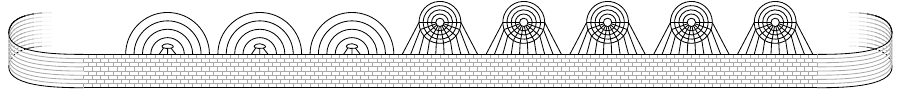}
    \caption{A $(9,11,3,5)$-walloid}\label{fig:walloid}
\end{figure}

\paragraph{Base, circumference, and home.}

We collectively refer to (elementary) wall, flap, and vortex segments as \emph{(elementary) $(r \times t)$-segments}.
In any segment $W$, the \emph{base} of $W$ is the subdivision of the wall segment $W_{1}$ (as denoted in the definition) that is a subgraph of $W$. The left and right boundary vertices of a segment $W$ are the left and right boundary vertices of its base. Note that an $(r \times t)$-segment has $r$ left boundary vertices and $r$ right boundary vertices.

Finally, for an elementary segment $W$ we define its {\em{circumference}} to be the cycle $C_W$ defined as follows:
\begin{itemize}
    \item If $W$ is a wall segment, then $C_W$ is the unique cycle contained in the union of: the left column of $W$, the right column of $W$, the top row of $W$, and the bottom row of $W$, see $C_W$ marked in \cref{fig:wall_segment}.
	\item If $W$ is a flap segment, then $C_W$ is the unique cycle contained in the union of: the left column of $W_1$, the right column of $W_1$, the bottom row of $W_1$, and the rainbow edge connecting the leftmost and the rightmost top boundary vertex of $W_1$, see \cref{fig:flap_segment}.
	\item If $W$ is a vortex segment, then $C_W$ is the unique cycle contained in the union of: the left column of~$W_1$, the right column of $W_1$, the bottom row of $W_1$, and the $4$-edge path connecting the leftmost and the rightmost top boundary vertex of $W_1$ whose internal vertices lie on the outer cycle of $W_2$, see \cref{fig:vortex_segment}.
\end{itemize}
Thus, in all cases, $C_W$ is a cycle contained in $W$ such that one of the disks bounded by $C_W$ contains all of $C_W$, possibly except for one of the left/right boundary vertices of $W_1$ and the incident edge.
We denote this disk by $\Delta_W$ and call it the {\em{home}} of the segment $W$.
We extend this terminology to non-elementary segments in the expected way, by applying the subdivision also to the circumference.

\paragraph{(Cylindrical) concatenation of segments.}

Let $r, t \in \N_{\geq 4}$ and $\ell \in \N_{\geq 1}$.
Given a sequence $W_{1}, \ldots, W_{\ell}$ of elementary $(r \times t)$-segments, their \emph{concatenation} is obtained from the disjoint union of $W_{1}, \ldots, W_{\ell}$ by adding an edge between the $j$-th right boundary vertex of $W_{i}$ and the $j$-th left boundary vertex of $W_{i+1}$, for all $i \in [\ell - 1]$ and $j \in [r]$.
Moreover, their \emph{cylindrical concatenation} is obtained from their concatenation by adding an edge between the $j$-th left boundary vertex of $W_{1}$ and the $j$-th right boundary vertex of $W_{\ell}$, for all $j \in [r]$, see \cref{fig:segment_concatenation}.
Notice that $W$ contains, as a subgraph, both the concatenation as well as the cylindrical concatenation of the bases of all the segments $W_1,\ldots,W_\ell$.
We refer to this annulus wall as the \emph{base annulus} of~$W$.
We extend these definitions to non-elementary segments in the expected way.

Let us note that the cylindrical concatenation of some collection of $(r\times t)$-segments, where $r,t\in \N_{\geq 4}$, is always a $3$-connected planar graph (or a subdivision of one).
Therefore, it has a unique embedding in the sphere $\Sigma$.

\paragraph{Walloids.}

Let $r,t \in \N_{\geq 4}$ and $a,b  \in \N$.
An \emph{elementary $(r, t, a,b)$-walloid} (\cref{fig:walloid}) is the cylindrical concatenation of $W_{0}, \ldots, W_{a+b}$, where 
\begin{itemize}
	\item $W_{0}$ is an elementary $(r \times r)$-wall segment,
	\item $W_{1}, \ldots, W_{a}$ are elementary $(r \times t)$-flap segments, and
	\item $W_{a + 1}, \ldots, W_{a+b}$ are elementary $(r \times t)$-vortex segments.
\end{itemize}
An \emph{$(r, t,a,b)$-walloid} is a subdivision of an elementary $(r, t,a,b)$-walloid. An \emph{(elementary) $(r,t)$-walloid} is an (elementary) $(r,t,a,b)$-walloid for some $a,b\in \N$.

We call $W_0,\ldots,W_{a+b}$ the \emph{segments} of $W$.
A set of segments of $W$ is \emph{consecutive} in $W$ if they appear consecutively in the above order.
For $i \in [a]$, we refer to $W_{i}$ as the \emph{$i$-th flap segment} of $W$ and for $j \in [b]$, we refer to $W_{a + j}$ as the \emph{$j$-th vortex segment} of $W$.
Note that in the absence of both flap and vortex segments, walloids are simply annulus walls. The {\em{breadth}} of a walloid is $b$, the number of vortex segments.

Similarly to walls, we define when a walloid is controlled by a well-linked set.
Let $G$ be a graph, $r,t \in \N_{\geq 4}$, and $W \subseteq G$ be an $(r, t)$-walloid.
As before, we may observe that for every separation $(A, B)$ of $G$ of order less than $r,$ there is a unique side, say $B,$ such that $B \setminus A$ contains the vertex set of both a cycle and a vertical path of the base annulus of $W.$
In this case, we call $B$ the \emph{$W$-majority side} of $(A, B).$
Given a well-linked set $X$ of $G$ of order at least $3r,$ we say that $W$ is \emph{controlled} by $X$ if for every separation $(A, B)$ of $G$ of order less than $r,$ $B$ is the $W$-majority side of $(A, B)$ if and only if $B$ is the $X$-majority side of $(A, B)$ as well.

Finally, given an $(r', t')$-walloid $W' \subseteq G$ such that $3\leq r' \leq r,$ we say that $W'$ is \emph{controlled} by $W$ if for every separation $(A, B)$ of $G$ of order less than $r',$ $B$ is the $W'$-majority side of $(A, B)$ if and only if $B$ is the $W$-majority side of $(A, B)$ as well.

\paragraph{Facial cycles of walloids.}

Let $W$ be an $(r,t,a,b)$-walloid, for some $r,t\in \N_{\geq 4}$ and $a,b\in \N$.
By definition, $W$ is a subdivision of a $3$-connected planar graph and therefore it has a unique (up to homeomorphism) embedding in $\Sigma$.
Let us classify the facial cycles (i.e., cycles bounding the faces) of this embedding.

A \emph{brick} of $W$ is any facial cycle of $W$ that contains at most $6$ degree-$3$ vertices of $W$ and is not an inner cycle of a flap or a vortex segment.
Among the facial cycles of $W$ that are not bricks, $a$ of them correspond to the inner cycles of the flap segments of $W$ and $b$ of them correspond to the inner cycles of the vortex segments of $W$. Note that for every segment $W'$ of $W$, the home of $W'$ contains all the bricks contained in $W'$ and all the faces of $W$ bounded by those bricks.

Two facial cycles of $W$ remain unclassified.
One of the two corresponds to the outer cycle of the base annulus of $W$ which we call the \emph{simple cycle} of $W$, denoted by $C^{\sisf}_{W}$. The only other unclassified facial cycle shall be called the \emph{exceptional cycle} of $W$ and denoted by $C^{\exsf}_{W}$. This is the only facial cycle that contains an endpoint from the top path of the base of every segment of $W$. The two faces of $W$ bounded by $C^{\sisf}_{W}$ and $C^{\exsf}_{W}$ will be denoted by $\Delta^{\sisf}_{W}$ and $\Delta^{\exsf}_{W}$, respectively; note that these are disks in $\Sigma$.

\subsection{Walloids in $\Sigma$-renditions}

Let $G$ be a planar graph and $\rend$ be a $\Sigma$-rendition of $G$.
Also, let $W$ be an $(r, t,a,b)$-walloid in $G$ that is $\rend$-grounded, for some $r,t\in \N_{\geq 4}$ and $a,b\in \N$. Recall that the segments of $W$ are $W_0,W_1,\ldots,W_{a+b}$, where $W_1,\ldots,W_a$ are flap segments and $W_{a+1},\ldots,W_{a+b}$ are the vortex segments.

Consider a flap segment $W_i$ of $W$, $i\in [a]$, say of arity $q \in [3]$. For brevity, we write $\Delta^\holesf_i\coloneqq \Delta^\holesf_{W_i}$ for the pocket of $W_i$.
We shall say that $W_i$ \emph{hosts} a cell $\Delta \in \rend$ if $|\per(\Delta)| = q$, $\Delta$ is drawn in the pocket $\Delta^{\holesf}_i$, and moreover, in $G$ there exists a confined $(\Delta, \Delta^{\holesf}_i)$-radial linkage $\Pp_i$  whose endpoints on $\bd(\Delta^{\holesf}_i)$ are the sockets of $W_i$, see \cref{fig:flap_details}.

\begin{figure}[h!]
	\centering
	\begin{minipage}{.58\textwidth}
		\centering
		\includegraphics[page=1,scale=1.10]{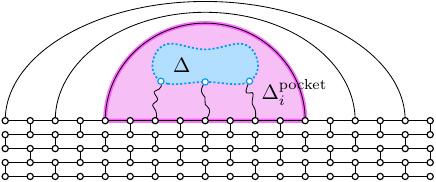}
		\caption{Flap segment $W_i$ of $W$ of arity $q=3$.}\label{fig:flap_details}
	\end{minipage}
	\begin{minipage}{.40\textwidth}
		\centering
		\includegraphics[page=2,scale=1.10]{flaps_vortices.pdf}
		\caption{Vortex segment $W_{a+j}$ of $W$.}\label{fig:vortex_details}
	\end{minipage}
\end{figure}

Next, consider the vortex segments $W_{a+1}, \ldots, W_{a+b}$; recall that $W_{a+j}$ is the $j$-th vortex segment, for $j\in [b]$.
Suppose there is a sequence of $\langle (\FDisk^{\insf}_{j},\FDisk^{\outsf}_{j})\colon j\in [b] \rangle$ of $\rend$-aligned disks such that, for each $j \in [b]$,
\begin{itemize}
	\item $\per(\FDisk^{\insf}_{j}) \subseteq V(C_{0})$, where $C_{0}$ is the inner cycle of the $j$-th vortex segment of $W$,
	\item $\per(\FDisk^{\outsf}_{j}) \subseteq V(C_{t + 1})$, where $C_{t + 1}$ is the outer cycle of the $j$-th vortex segment of $W$, and
	\item $\FDisk^{\insf}_{j} \subseteq \FDisk^{\outsf}_{j}$ and disks $\{\FDisk^{\outsf}_{j}\colon j\in [b]\}$ are pairwise disjoint.
\end{itemize}
We shall refer to $\FDisk^{\insf}_{j}$ as the \emph{inner vortex disk} and to $\FDisk^{\outsf}_{j}$ as the \emph{outer vortex disk} for the $j$-th vortex segment see \cref{fig:vortex_details}.
Note that the nest of the $j$-th vortex segment is by definition sandwiched by $(\FDisk^\insf_{j}, \FDisk^\outsf_{j}).$ The sequence $\langle (\FDisk^{\insf}_{j},\FDisk^{\outsf}_{j})\colon j\in [b] \rangle$ as above shall be called a \emph{sandwich sequence} for $W$ and we say that the walloid $W$ is \emph{$\rend$-well-grounded} if it admits a sandwich sequence.

We say that the walloid $W$ has \emph{depth} $d$ (with respect to the $\Sigma$-rendition $\rend$) if every vortex segment of $W$ has an outer vortex disk of depth $d.$
Whenever this is the case we shall implicitly assume that any outer vortex disk we choose is a disk of depth (at most) $d.$

A \emph{cell-coloring} of $\rend$ is a function $\chi \colon \rend \to \N_{\geq 1}.$
We define $\chi(\rend) \coloneqq \{ \chi(\Delta) \mid \Delta \in \rend \}$.
The maximum integer in $\chi(\rend)$ is the \emph{capacity} of $\chi$.

\paragraph{Tight renditions.}

We also require some additional connectivity properties for our renditions.
For this we introduce a couple of new definitions.

\medskip
Let $G$ be a planar graph and $\rend$ be a $\Sigma$-rendition of $G$.
We define the \emph{$\rend$-torso} of $G$ as the graph $T$ with vertex set being the ground vertices of $\rend,$ and $uv \in E(T),$ whenever $u$ and $v$ belong to the periphery of a common cell of $\rend.$
Notice that $T$ is clearly planar, and we implicitly consider the $\Sigma$-embedding for $T$ to be the one induced by the $\Sigma$-embedding of $G$ in the obvious way.

We say that $\rend$ is \emph{tight} if the following conditions are met:
\begin{enumerate}
	\item For all $\Delta \in \rend,$ there is an $x$-$y$ path in $G \cap \Delta$ between any two distinct vertices $x, y \in \per(\Delta).$
	\item The $\rend$-torso of $G$ is $2$-connected.
\end{enumerate}
Below we explain how one can obtain a tight $\Sigma$-rendition for a given planar graph in linear time.

\begin{observation}\label{obs:tight_rendition} There is an algorithm that, given a planar graph $G$ with a $t$-annulus wall $W \subseteq G,$ for some $t \in \N_{\geq 4}$, computes a tight $\Sigma$-rendition $\rend$ of $G$ with $W$ being $\rend$-grounded in time $\Oh(|G| + \|G\|).$
\end{observation}
\begin{proof} Let $C_{W}$ be the connected component of $G$ that contains $W.$
	We start from the trivial $\Sigma$-rendition $\rend$ of $G$ where every vertex of $C_{W}$ is a ground vertex and every edge of $C_{W}$ belongs to its own cell with periphery of size $2.$
	Next (by possibly modifying the $\Sigma$-embedding of $G$), we identify a disk $\Delta$ that contains the drawing of all other components of $G$ (if any) and only those, in a way that $\bd(\Delta) \cap G = \emptyset.$
	In particular this means that the periphery of $\Delta$ has size $0.$
	Clearly $\rend$ satisfies the first condition above and $W$ is $\rend$-grounded.
	Moreover, trivially, the $\rend$-torso $T$ of $G$ is connected but may not be $2$-connected.
	Suppose that $T$ is not $2$-connected, otherwise we conclude with $\rend.$

	We proceed to modify $\rend$ to the desired tight $\Sigma$-rendition as follows: Since $W$ is $\rend$-grounded, by property~i) above, and the definition of $T,$ there is a $t$-annulus wall $W' \subseteq T$ such that $W$ is a subdivision of $W'.$
	We compute the set $\mathcal{B}$ of blocks ($2$-connected components) of $T$ by using the classic linear-time algorithm of Tarjan~\cite{Tarjan1971Depth}.
	Since $W'$ is a subdivision of a $3$-connected graph it follows that there exists a unique block $B \in \mathcal{B}$ that contains $W'.$
	Let $u$ be a cut-vertex of $T$ that belongs to $B$ and $(X, Y)$ be a separation of $T$ such that $X \cap Y = \{ u \},$ $X \setminus \{ u \}$ contains the vertex set of the component of $T \setminus u$ that contains $V(B) \setminus \{ u \}$ while $Y \setminus \{ u \}$ contains the rest.
	By standard planarity arguments, it follows that there is a closed curve $J$ in $\Sigma$ that intersects the drawing of $T$ only at $u,$ bounds a $\rend$-aligned disk $\Delta_{J}$ internally disjoint from $X,$ disjoint from $\Delta,$ and such that every vertex of $Y$ is drawn in $\Delta_{J}.$

	We now may define $\rend'$ by replacing all cells of $\rend$ contained in $\Delta_{J}$ by $\Delta_{J},$ for all cut-vertices $u$ and closed curves $J$ as defined above.
	Clearly $\rend'$ now satisfies both tightness conditions and $W$ is $\rend'$-grounded as desired.
\end{proof}

The purpose of introducing these definitions is to import a powerful result from the work of Paul, Protopapas, Thilikos, and Wiederrecht \cite{PaulPTS2025LocalIndex, PaulPTW2024Obstructions}, presented as \cref{lst_fi} below.
The statement is an abbreviated version of Lemma 8.1 from \cite{PaulPTS2025LocalIndex}, adjusted for our setting. The intuition is the following: Given a planar graph $G$ with several types of terminals highlighted (in our case, the terminal sets $S$ and $T$ will make up two types of terminals), the statement provides an infrastructure in $G$ that exposes whether every terminal set is ``local'' or ``global,'' and how the terminal sets can be ``linked'' to each other in the embedding.

\begin{proposition}[{\cite[Lemma 8.1]{PaulPTS2025LocalIndex}}]\label{lst_fi} There exist functions $\mathsf{annulus}_{\ref{lst_fi}} \colon \N^{4} \to \N$, $\mathsf{breadth}_{\ref{lst_fi}} \colon \N^{2} \to \N$, and $\depth_{\ref{lst_fi}} \colon \N^{3} \to \N$ such that for all $r,t \in \N_{\geq 4}$ and $\ell, q \in \N_{\geq 1}$ the following holds.
	
Suppose $G$ is a planar graph with a tight $\Sigma$-rendition, $W \subseteq G$ is a $\rend$-grounded $\mathsf{annulus}_{\ref{lst_fi}}(r, t, \ell, q)$-annulus wall, and $\chi$ is a cell-coloring of $\rend$ of capacity at most $\ell.$

Then there exists a $\rend$-well-grounded $(r, t, a, b)$-walloid $W' \subseteq G$ for some $a \in \mathbb{N}$ and $b \leq \mathsf{breadth}_{\ref{lst_fi}}(\ell, q),$ with a sandwich sequence $\langle (\FDisk^{\insf}_{j},\FDisk^{\outsf}_{j})\colon j\in [b] \rangle$, such that the following conditions hold:
	\begin{enumerate}
		\item $W'$ is controlled by $W$;
		\item $W'$ has depth at most $\depth_{\ref{lst_fi}}(t, \ell, q)$ with respect to $\rend$; and
		\item For every $\alpha \in \chi(\rend)$, if $\rend_{\alpha} = \{ \Delta \in \rend \mid \chi(\Delta) = \alpha\}$, then either
		\begin{itemize}
			\item $\bigcup \rend_{\alpha} \subseteq \bigcup_{j \in [b]} \inte(\FDisk^{\insf}_{j})$; or
			\item there is a sequence of $q$ consecutive flaps segments of $W'$, each hosting a cell from $\rend_{\alpha}$.
		\end{itemize}
	\end{enumerate}
	Moreover, it holds that
	\begin{align*}
		\mathsf{annulus}_{\ref{lst_fi}}(r, t, \ell, q) \ &\in \ (r + (\ell q + 1)^{q \cdot 2^{\Oh(\ell)}} \cdot t)^{2^{\Oh(\ell)}}\\
		\mathsf{breadth}_{\ref{lst_fi}}(\ell, q) \ &\in \ q \cdot 2^{\Oh(\ell)}\text{, and}\\
		\depth_{\ref{lst_fi}}(t, \ell, q) \ &\in \ (\ell q + 1)^{q \cdot 2^{\Oh(\ell)}} \cdot t.
	\end{align*}
\end{proposition}

Let us briefly explain how \cref{lst_fi} relates to \cite[Lemma 8.1]{PaulPTS2025LocalIndex}.
The $\Sigma$-renditions we use here, tailored to planar graphs, can be viewed as a special (and significantly simpler) instance of the general $\Sigma$-decompositions used in \cite{PaulPTS2025LocalIndex} for $H$-minor-free graphs.
This simplification suffices for our purposes, since we restrict attention exclusively to planar graphs.

In particular, the tightness condition required for $\Sigma$-decompositions in \cite{PaulPTS2025LocalIndex} is formulated in a more general way.
There, the authors introduce the notion of a torsoid associated with a $\Sigma$-decomposition and require it to be $3$-connected.
In the planar setting considered here, and under our notion of a $\Sigma$-rendition, this requirement coincides with our tightness condition.

Furthermore, the conclusion of \cite[Lemma 8.1]{PaulPTS2025LocalIndex} yields a new $\Sigma$-decomposition representing the given cell-coloring that is ``coarser'' than the original one.
Here, coarser means that all cells contained in the inner vortex disk of each vortex segment (in the sense of our definition of a walloid) are replaced by a single vortex cell, possibly with large periphery.

In our setting, we omit this cell-replacement step, so that the resulting $\Sigma$-decomposition coincides with the original one. This is consistent with the proof of \cite[Lemma 8.1]{PaulPTS2025LocalIndex} and entails no loss of generality.

\subsection{Classifying terminals}

We are now in the position to setup a lemma that describes the structure of a planar graph of large treewidth relative to a pair of prespecified sets of terminals. This is essentially a combination of all the tools mentioned above, in particular of \cref{lst_fi}, ready for the usage in later sections.

\medskip
Let $G$ be a planar graph, $\rend$ be a $\Sigma$-rendition of $G$, and $W$ be a $\rend$-grounded walloid in $G$.
Given a set $X \subseteq V(G)$, we define
$$\rend_{X} \coloneqq \{ \Delta \in \rend \mid \text{$G \cap \Delta$ contains an $X$-$\per(\Delta)$-path} \}.$$
Moreover, for $c\in \N$ and $\mathfrak{R}\subseteq \rend$, we say that $W$ \emph{$c$-represents} $\mathfrak{R}$ if there is a sequence of $c$ consecutive flap segments of~$W$, each hosting a cell belonging to $\mathfrak{R}.$
\Cref{fig:classifying_terminals} illustrates the following lemma.

\begin{lemma}\label{structure_relative_terminals}\ There exist functions $\mathsf{annulus}_{\ref{structure_relative_terminals}} \colon \N^3 \to \N,$ $\depth_{\ref{structure_relative_terminals}} \colon \N^{2} \to \N$ and $\mathsf{breadth}_{\ref{structure_relative_terminals}} \colon \N \to \N$ such that, for all $r,t \in \N_{\geq 4}$ and $q \in \N_{\geq 1}$ the following holds.

	Suppose $G$ is a planar graph with an $\mathsf{annulus}_{\ref{structure_relative_terminals}}(r, t, q)$-annulus wall $W \subseteq G$ and $S, T \subseteq V(G)$ is a pair of vertex subsets.
	Also, suppose $F \subseteq V(G)$ is a vertex subset of size strictly less than $q.$
	
	Then, there is a $\Sigma$-rendition $\rend$ of $G$ and a $\rend$-well-grounded $(r, t, a, b)$-walloid $W' \subseteq G$ for some $a \in \mathbb{N}$ and $b \leq \mathsf{breadth}_{\ref{structure_relative_terminals}}(q),$ with a sandwich sequence $\langle (\FDisk^{\insf}_{j},\FDisk^{\outsf}_{j})\colon j\in [b] \rangle$, such that the following conditions~hold:
	\begin{enumerate}
		\item $W'$ is controlled by $W$;
		\item $W'$ has depth at most $\depth_{\ref{lst_fi}}(t, q)$ with respect to $\rend$;
		\item $F \subseteq \bigcup_{j \in [b]} \inte(\Theta^\insf_j)$; and
		\item for each $\mathfrak{R} \in \{ \rend_{S} \cap \rend_{T}, \rend_{S} \setminus \rend_{T}, \rend_{T} \setminus \rend_{S} \}$, either $\bigcup \mathfrak{R} \subseteq \bigcup_{j \in [b]} \inte(\FDisk^{\insf}_{j})$ or $W'$ $q$-represents $\mathfrak{R}.$
	\end{enumerate}
	Moreover, it holds that
	\begin{align*}
		\mathsf{annulus}_{\ref{structure_relative_terminals}}(r, t, q) &\in (r + 2^{\Oh(q \log q)} \cdot t)^{\Oh(1)},\\
		\mathsf{breadth}_{\ref{structure_relative_terminals}}(q) &\in \Oh(q), \quad \text{and}\\
		\depth_{\ref{structure_relative_terminals}}(t, q) &\in 2^{\Oh(q \log q)} \cdot t.
	\end{align*}
\end{lemma}
\begin{proof} The lemma follows as an easy application of \cref{lst_fi}.
	From \cref{obs:tight_rendition} we get a tight $\Sigma$-rendition $\rend$ of $G$ with $W$ being $\rend$-grounded.
	Next, we define a cell-coloring $\chi$ of $\rend$ of capacity at most $5$ so that for every cell $\Delta \in \rend$,
	\begin{itemize}
		\item $\chi(\Delta) = 1$ if $G \cap \Delta$ contains a vertex of $F$;
		\item $\chi(\Delta) = 2$ if $G \cap \Delta$ contains no vertex of $F$ and both an $S$-$\per(\Delta)$-path and a $T$-$\per(\Delta)$-path;
		\item $\chi(\Delta) = 3$ if $G \cap \Delta$ contains no vertex of $F$ and an $S$-$\per(\Delta)$-path but no $T$-$\per(\Delta)$-path;
		\item $\chi(\Delta) = 4$ if $G \cap \Delta$ contains no vertex of $F$ and a $T$-$\per(\Delta)$-path but no $S$-$\per(\Delta)$-path; and
		\item $\chi(\Delta) = 5$ otherwise.
	\end{itemize}
	Now the claim follows directly by applying \cref{lst_fi}, since $|F| < q$ and therefore, the set of cells $\{ \Delta \in \rend \mid \chi(\Delta) = 1 \}$ cannot be $q$-represented by $W'.$
\end{proof}

\begin{figure}[ht]
    \centering
    \includegraphics[scale=1.1]{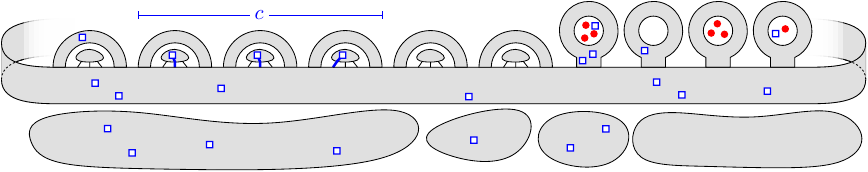}
    \caption{
        Considering $\mathfrak{R}$ of \Cref{structure_relative_terminals}, each set of cells is either entirely contained in the inner vortex disks (full red circles) or is $c$-represented by the walloid (empty blue squares).
    }%
    \label{fig:classifying_terminals}
\end{figure}

We note that the small set $F$ of ``forbidden'' vertices featured in the statement of \Cref{structure_relative_terminals} will not be of relevance in the planar case, but will turn useful in the lift to the setting of graphs embeddable in more complicated surfaces. This will be also the case in all the similar statements until \cref{sec:genus}.

\section{Finding scattered paths}\label{sec:finding}

The goal of this section is to identify several cases when the infrastructure provided by \cref{structure_relative_terminals} allows us to immediately find a large $S$-$T$-linkage consisting of paths that are far apart. We start with some terminology and simple observations.

For the remainder of this section, we fix a planar graph $G$, a rendition $\rend$ of $G$ in the sphere $\Sigma$, sets of terminals $S,T\subseteq V(G)$, a distance parameter $d\in \N,$ and a $\rend$-well-grounded $(r, t, a, b)$-walloid in $G,$ for some $r, t \in \N_{\geq 4},$ $a \in \N,$ and $b \in \N_{\geq 1}.$
The value of $r$ and $t$ will vary, depending on the requirements of each subsequent statement.
Also fix a sandwich sequence $\Zz$ for $W$ and a vertex subset $F \subseteq V(G)$ that is contained in the union of the interiors of the inner vortex disks in $\Zz$ of the vortex segments of $W.$

Suppose $\Hh$ is a family of subgraphs of $G$.
We shall say that $\Hh$ is \emph{$d$-scattered} in $G$ if for every two distinct subgraphs $H,H' \in \Hh$, we have $\dist_{G}(H,H') > d$. Note that $0$-scattered simply means pairwise disjoint. We say that a railed nest $(\Cc,\Pp)$ in $G$ is \emph{$d$-scattered} if both $\Cc$ and $\Pp$ are $d$-scattered (as families of subgraphs of $G$). Finally, we say that a set of integers $I$ is {\em{$d$-scattered}} if $|i-j|>d$ for all distinct $i,j\in I$.

The following definition will be useful for arguing scatteredness.
Let $W'$ be a segment of $W$. We call a subgraph $H$ of $G$ {\em{$d$-hidden}} in $W'$ if $H$ is entirely contained in the disk $\Delta_{W'}$ (the home of $W'$) and the distance in $G$ between $H$ and the circumference $C_{W'}$ is at least $d$. The following observation is nearly obvious.

\begin{observation}\label{obs:union-hidden}
	Let $W',W''$ be two distinct segments of $W$, and $H',H''$ be subgraphs of $G$ such that $H'$ is $d$-hidden in $W'$ and $H''$ is $d$-hidden in $W''$.
	Then $\{H',H''\}$ is $d$-scattered.
\end{observation}
\begin{proof}
	Any path $Q$ in $G$ connecting a vertex of $H'$ with a vertex of $H''$ must necessarily intersect both $C_{W'}$ and $C_{W''}$. Since $H'$ is $d$-hidden in $W'$, the prefix of $Q$ until the first intersection with $C_{W'}$ must have length at least $d$. Similarly, the suffix of $Q$ from the last intersection with $C_{W''}$ also has length at least $d$. As $C_{W'}$ and $C_{W''}$ are disjoint, the length of $Q$ is at least $2d+1>d$. 
\end{proof}

Also, we state a useful criterion for arguing that a subgraph is hidden.

\begin{observation}\label{obs:nest-criterion}
	Let $W'$ be a segment of $W$ and $H$ be subgraph of $G$ that is entirely contained in $\Delta_{W'}$.
	Suppose there exists a nest $\Cc$ of order $d$, say with inner disk $\Delta^\insf$ and outer disk $\Delta^\outsf$, such that $\Delta^\outsf\subseteq \Delta_{W'}$ and $H\subseteq \inte(\Delta^\insf)$.
	Then $H$ is $d$-hidden in $W'$.
\end{observation}
\begin{proof}
    Note that any path $Q$ connecting $H$ with $C_{W'}$ must intersect each of the cycles of $\Cc$. As these $d$ cycles are pairwise disjoint and disjoint from $H$, it follows that $Q$ has length at least~$d$.
\end{proof}

It follows that each vertex of $F$ (seen as a trivial subgraph of $G$) is $d$-hidden in a vortex segment of $W.$

\begin{observation}\label{obs:special_vertices_hidden}
	Let $h\leq t$ be a non-negative integer.
	Then, every vertex $u \in F$ is $h$-hidden in some vortex segment $W'$ of $W.$
	In particular, every vertex of~~$\Ball^{h}_{G}(u)$ is drawn in the disk bounded by the $h$-th innermost cycle of the nest of $W'$ that contains the inner vortex disk of $W'$ in $\Zz.$
\end{observation}
\begin{proof}
	By assumption, $u$ is drawn in the interior of the inner vortex disk of some vortex segment $W'$ of $W.$
	Therefore, the fact that $u$ is $h$-hidden in $W'$ follows directly by applying the criterior of \cref{obs:nest-criterion} with the nest of $W'$ as input.
	The second part of the claim then follows.
\end{proof}

Finally, we introduce a useful definition for finding relevant infrastructure within the base annulus of a walloid.
A {\em{$d$-scattered shaft}} of order $s\in \N_{\geq 2}$ in $W$ is a nest $\Cc^\star$ consisting of cycles of the base annulus of $W$ that is $d$-scattered and does not contain any of the $d$ outermost cycles of $W$.
We require that the order of the cycles in $\Cc^\star$ matches the inner-to-outer order of the cycles of the base annulus of $W$, so that the inner disk of $\Cc^\star$ contains the exceptional disk $\Delta^\exsf_W$.

We note that in a large walloid we can always find a large shaft that avoids the vicinity of $F$.

\begin{observation}\label{obs:shaft}
	Suppose $s\in \N_{\geq 2}$ is such that $r \geq s(d+1)$.
	Then $W$ contains a $d$-scattered shaft $\Cc^\star$ of order $s$ that avoids $\Ball^{t}_{G}(F).$
\end{observation}
\begin{proof}
	Enumerating the cycles of the base annulus of $W$ as $C_1,\ldots,C_r$ from the innermost to the outermost, it suffices to take
	\[\Cc^\star\coloneqq (C_1,C_{(d+1)+1},C_{2(d+1)+1},\ldots,C_{(s-1)(d+1)+1}).\]
	Indeed, the nest $\Cc^\star$ defined in this way is $d$-scattered, because any path in $G$ connecting any two distinct cycles from $\Cc^\star$ must intersect, at an internal vertex, $d$ other cycles of $C_1,\ldots,C_r$, and therefore must have length larger than $d$.
	Note also that since $r\geq s(d+1)$, $\Cc^\star$ does not contain any of the cycles $C_{r-d+1},\ldots,C_r$.
	That $\Cc^\star$ avoids $\Ball^t_G(F)$ follows immediately from  \cref{obs:special_vertices_hidden}.
\end{proof}

\subsection{Scattered connectors}

We start by introducing a general object, a scattered connector, and arguing that a large scattered connector always contains a large scattered $S$-$T$-linkage. This observation will be later used repeatedly when analyzing the infrastructure provided by \cref{structure_relative_terminals}. In essence, we will find different scattered connectors depending on the possible alignments of the terminals.

Let $k\in \N_{\geq 1}$. A \emph{$d$-scattered $S$-$T$-connector} of \emph{order} $k$ in $G$ is a $d$-scattered confined railed nest $(\Cc,\Pp)$ of order $(k+2,2k)$ satisfying the following property: $\Pp$ can be partitioned into two sets $\Pp_S,\Pp_T$ consisting of $k$ paths each, each path in $\Pp_S$ has an endpoint in $\Delta^\insf\cap S$ and each path in $\Pp_T$ has an endpoint in $\Delta^\insf\cap T$; here, $\Delta^\insf$ is the inner disk of~$\Cc$. (In other words, $(\Cc,\Pp_S)$ is $S$-rooted and $(\Cc,\Pp_T)$ is $T$-rooted.)

\begin{lemma}\label{scattered_connector}
	If $G$ contains a $d$-scattered $S$-$T$-connector $(\Cc, \Pp)$ of order $k$, then $G$ contains a $d$-scattered $S$-$T$-linkage $\Qq$ of order $k$ such that $$\bigcup \Qq \subseteq \bigcup (\Cc \cup \Pp).$$
\end{lemma}
\begin{proof}
	Let $(\Cc,\Pp)$ be a $d$-scattered $S$-$T$-connector of order $k$ in $G$, with a suitable partition $\Pp_S,\Pp_T$ of $\Pp$. 
	Let $\Cc=(C_0,C_1,\ldots,C_{k+1})$. Recall that the cycle $C_0$ bounds the inner disk $\Delta^\insf$ of $\Cc$. Recalling that $(\Cc,\Pp)$ is confined, let us enumerate the paths of $\Pp$ as $P_1,\ldots,P_{2k}$ according to the order of their endpoints lying on $C_{k+1}$, see \cref{fig:connector_to_linkage_naming}.
	We will follow the convention that $P_0=P_{2k}$, $P_{2k+1}=P_1$, and $P_{2k+2}=P_2$. Also, we let $S'$ and $T'$ be the endpoints of the paths of $\Pp_S$ and $\Pp_T$ that lie in $\Delta^\insf$, respectively. Thus, $S'\subseteq S$ and~$T'\subseteq T$.

	Recall that since $(\Cc,\Pp)$ is a railed nest, $\Pp$ is orthogonal to $\Cc$, hence $C_i\cap P_j$ is a path for every $i\in [0,k+1]$ and $j\in [2k]$. We observe that contracting each of these paths to a single vertex breaks neither the assumption that $(\Cc,\Pp)$ is a confined railed nest, nor the assumption that both $\Cc$ and $\Pp$ are $d$-scattered. Therefore, from now on we assume that for every $i\in [0,k+1]$ and $j\in [2k]$, $C_i\cap P_j$ consists of a single vertex $z_{i,j}$, see \cref{fig:connector_to_linkage_contraction}.
    We also define
	\[
		Z\coloneqq \{z_{i,j}\colon i\in [0,k+1], j\in [2k]\}\cup S'\cup T'
		\qquad\textrm{and}\qquad
		H\coloneqq \bigcup_{i\in [k]} C_i\;\cup \bigcup_{j\in [2k]} P_j.
	\]
	Let a {\em{segment}} be a path of positive length in $H$ whose both endpoints belong to $Z$ and whose internal vertices do not belong to $Z$. Note that there are no segments that are subpaths of $C_0$ or $C_{k+1}$, since these two cycles are {\em{not}} included in $H$. Note also that $H$ is the union of all the segments.
	
	\begin{figure}[h]
		\centering
		\begin{subfigure}[t]{0.36\textwidth}
			\centering
			\includegraphics[page=1,scale=1.1]{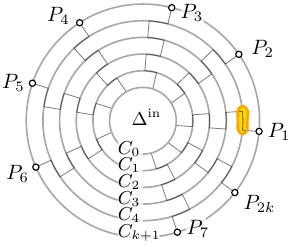}
			\caption{$d$-scattered $S$-$T$-connector of order $4$.}\label{fig:connector_to_linkage_naming}
		\end{subfigure}
		\begin{subfigure}[t]{0.36\textwidth}
			\centering
			\includegraphics[page=2,scale=1.1]{connector_to_linkage.pdf}
			\caption{$S$-$T$-connector after contraction}\label{fig:connector_to_linkage_contraction}
		\end{subfigure}
		\begin{subfigure}[t]{0.26\textwidth}
			\centering
			\includegraphics[page=4,scale=1.1]{connector_to_linkage.pdf}
			\caption{Non-incident segments are $d$-scattered}\label{fig:connector_to_linkage_segment_distances}
		\end{subfigure}
		\caption{
			Obtaining a $d$-scattered $S$-$T$-linkage from a given $d$-scattered $S$-$T$-connector.
		}\label{fig:connector_to_linkage}
	\end{figure}

	Call two segments {\em{incident}} if they share an endpoint. We observe that non-incident segments are actually far from each other.
	
	\begin{claim}\label{cl:frame}
		For any two non-incident segments $A,A'$, we have that $\{A,A'\}$ is $d$-scattered in $G$.
	\end{claim}
	\begin{claimproof}
		Suppose first that one of these segments is a subpath of some cycle of $\Cc$; say $A$ is the segment on $C_i$ connecting $z_{i,j}$ and $z_{i,j+1}$, for some $i\in [k]$ and $j\in [2k]$, see \cref{fig:connector_to_linkage_segment_distances}.
		Let $D$ be the cycle obtained by concatenating:
		\begin{itemize}
			\item the subpath of $P_{j-1}$ between $z_{i-1,j-1}$ and $z_{i+1,j-1}$;
			\item the subpath of $C_{i+1}$ between $z_{i+1,j-1}$ and $z_{i+1,j+2}$ containing $z_{i+1,j}$ and $z_{i+1,j+1}$;
			\item the subpath of $P_{j+2}$ between $z_{i+1,j+2}$ and $z_{i-1,j+2}$; and
			\item the subpath of $C_{i-1}$ between $z_{i-1,j+2}$ and $z_{i-1,j-1}$ containing $z_{i-1,j+1}$ and $z_{i-1,j}$.
		\end{itemize}
		We observe that any path in $G$ that connects $A$ with $D$ must have length more than~$d$. This is because such a path either connects $A\subseteq C_i$ with $C_{i-1}$ or $C_{i+1}$ --- and $\{C_{i-1},C_i,C_{i+1}\}\subseteq \Cc$ is $d$-scattered --- or intersects both $P_j$ and $P_{j-1}$, or both $P_{j+1}$ and $P_{j+2}$ --- and $\{P_{j-1},P_j,P_{j+1},P_{j+2}\}\subseteq \Pp$ is $d$-scattered. It now remains to note that since $A'$ is not incident to $A$, any path in $G$ connecting $A$ with $A'$ must necessarily intersect $D$.
		
		We are left with the case when both $A$ and $A'$ are subpaths of some paths from $\Pp$. Since $\Pp$ is $d$-scattered, we are immediately done unless $A$ and $A'$ are subpaths of the same path from $\Pp$, say $P_j$ for some $j\in [2k]$. Since $A$ and $A'$ are not incident, there is an index $i\in [k-1]$ such that each of the cycles $C_i$ and $C_{i+1}$ separates $A$ and $A'$. Then every path in $G$ connecting $A$ with $A'$ must intersect both $C_i$ and~$C_{i+1}$, and hence have length more than $d$ due to $\Cc$ being $d$-scattered.
	\end{claimproof}
	  
	 Next, we note that there is a large $S'$-$T'$-linkage  in $H$. This follows from the (classic) Menger's~Theorem.
	 
	 \begin{claim}\label{cl:easy-linkage}
	 	There is an $S'$-$T'$-linkage of order $k$ in $H$.
	 \end{claim}
 	\begin{claimproof}
 		Suppose otherwise.
 		By Menger's Theorem, we find a set $X\subseteq V(H)$ with $|X|<k$ that intersects every $S'$-$T'$-path in $H$. Since $|\Pp_S|=|\Pp_T|=k$, there is a path $P_S\in \Pp_S$ that is disjoint from $X$, a path $P_T\in \Pp_T$ that is disjoint from $X$, and a cycle $C\in \{C_i \mid i\in [k]\}$ that is disjoint from $X$. Then $P_S\cup C\cup P_T$ contains an $S'$-$T'$-path disjoint from $X$, a contradiction.
 	\end{claimproof}
	 
	Let $\Qq$ be the linkage provided by \cref{cl:easy-linkage}. Since $S'\cup T'\subseteq Z$ and $H$ is the union of all the segments, every path $Q\in \Qq$ is the concatenation of a collection of segments. And since the paths of $\Qq$ are pairwise disjoint, segments contained in different paths of $\Qq$ are not incident. It now follows from \cref{cl:frame} that $\Qq$ is $d$-scattered.
\end{proof}

\subsection{Scattered paths via represented terminals}

In this section, we discuss two cases when a large scattered $S$-$T$-linkage can be exposed immediately within the infrastructure provided by \cref{structure_relative_terminals}. This happens when the walloid either represents $\rend_S\cap \rend_T$, or represents both $\rend_S\setminus \rend_T$ and $\rend_T\setminus \rend_S$.

\paragraph{$\rend_{S} \cap \rend_{T}$ is represented.} We first examine the easy case  when the walloid provided by \cref{structure_relative_terminals} represents~$\rend_{S} \cap \rend_{T}$.

\begin{lemma}\label{lemma:union_represented}
	Suppose that $r, t \geq d + 1.$
	Suppose further that $W$ $k$-represents $\rend_{S} \cap \rend_{T}$, for some $k\in \N$.
	Then, $G$ contains a $d$-scattered $S$-$T$-linkage of order $k$ that avoids $\Ball^{d}_{G}(F).$
\end{lemma}
\begin{proof}
	Let $W_{1}, \ldots, W_{k}$ be consecutive flap segments of $W$ such that for each $i \in [k]$, $W_{i}$ hosts a cell $\Delta_{i} \in \rend_{S} \cap \rend_{T}$.
	Let $\Delta^{\holesf}_{i}\coloneqq \Delta^\holesf_{W_i}$ be the pocket of $W_i$ and let $\Pp_{i}$ be the confined $(\Delta_{i}, \Delta^{\holesf}_{i})$-radial linkage in $G$ witnessing that $W_i$ hosts $\Delta_i$.
	By combining
	\begin{itemize}
	\item the $S$-$\per(\Delta_{i})$ and $T$-$\per(\Delta_{i})$-paths in $G \cap \Delta_{i}$, existing by the definition of $\rend_{S}$ and~$\rend_{T}$,
	\item paths of the linkage $\Pp_i$, and
	\item a subpath of the pocket cycle $C^\holesf_{W_i}=\bd(\Delta^\holesf_i)$,	
	\end{itemize}
	we may construct an $S$-$T$-path $Q_i$ entirely contained in $G\cap \Delta^{\holesf}_{i}$.
	Observe also that due to $r>d$ and $t>d$, $Q_i$ is $d$-hidden in $W_i$. To see this, it suffices to apply the criterion from \cref{obs:nest-criterion}: a suitable nest of order $d$ can be constructed using the rainbow of $W_i$, the $d$ leftmost and the $d$ rightmost columns of the base of~$W_i$, and the $d$ bottommost rows of the base of $W_i$.
	Therefore, by \cref{obs:union-hidden} we conclude that $\{Q_i\colon i\in [k]\}$ is a $d$-scattered $S$-$T$-linkage of order $k$ in $G$.
	Since, each path $Q_{i}$ is $d$-hidden in a segment of $W$ that is not a vortex segment, \cref{obs:special_vertices_hidden} also implies that $\Qq$ avoids $\Ball^{d}_{G}(F)$.
\end{proof}

\paragraph{Both $\rend_{S} \setminus \rend_{T}$ and $\rend_{T} \setminus \rend_{S}$ are represented.}

Next, we examine the case when \cref{structure_relative_terminals} gives us a walloid that represents both $\rend_{S} \setminus \rend_{T}$ and $\rend_{T} \setminus \rend_{S}$. The main construction is captured in the following statement, which we extract separately for the purpose of future reusage.

\begin{lemma}\label{lem_s_represented}
	Suppose that $r, t \geq d + 1.$
	Suppose further that $W$ $k$-represents $\rend_{S} \setminus \rend_{T}$ for some $k\in \N$, and let $\Cc^\star$ be a $d$-scattered shaft in $W$ of order $s\in \N_{\geq 2}$. Then there is a~linkage $\Pp$ in $G$ such that
	\begin{itemize}
		\item $(\Cc^\star,\Pp)$ is an $S$-rooted $d$-scattered confined railed nest;
		\item every path of $\Pp$ is $d$-hidden in a different segment of $W$ that represents $\rend_S\setminus \rend_T$; and
		\item $\Pp$ avoids $\Ball^{d}_{G}(F)$.
	\end{itemize}
\end{lemma}
\begin{proof}
	Let $W_{1}, \ldots, W_{k}$ be $k$ consecutive flap segments of $W$ such that for each $i \in [k]$, $W_{i}$ hosts a cell $\Delta_{i}\in \rend_S\setminus \rend_T$. For brevity, by $\Delta^\holesf_{i}\coloneqq \Delta^\holesf_{W_{i}}$ we denote the pocket of $W_{i}$. Let $\Pp_{i}$ be the confined $(\Delta_{i},\Delta_{i}^\holesf)$-radial linkage witnessing that $W_{i}$ hosts $\Delta_{i}$.
	
	We define $\Pp = \{ P_{1}, \ldots, P_{k} \}$ as follows.
	Let $C$ be the outermost cycle of $\Cc^\star$.
	For each $i \in [k]$, we define $P_{i}$ as the $S$-$V(C)$-path obtained as follows:
	\begin{itemize}
		\item Choose an $S$-$\per(\Delta_{i})$-path in $G \cap \Delta_{i}$, existing by the assumption that $\Delta_{i}\in \rend_S$, and
		\item extend this path to a top boundary vertex of $W_{i}$ via a path of $\Pp_{i}$, and then all the way through $W_i$ via the corresponding vertical path of the base of $W_{i}$ until it hits $C$ for the first time.
	\end{itemize}
	By construction it is clear that $(\Cc^\star,\Pp)$ is an $S$-rooted confined railed nest.

	Let us argue that $P_i$ is $d$-hidden in $W_i$. To see this, we may apply the criterion of \cref{obs:nest-criterion}, observing that a suitable nest of order $d$ can be constructed using the rainbow of $W_i$, the $d$ leftmost and the $d$ rightmost columns of the base of $W_i$, and the $d$ bottommost rows of the base of $W_i$. Note here that by the definition of a shaft, $C$ is not among the $d$ outermost cycles of the base annulus of $W$, hence $P_i$ does not intersect the $d$ bottommost rows of the base of $W_i$.
	
	It now follows from \cref{obs:union-hidden} that $\Pp$ is $d$-scattered.
	Moreover, since $P_{i}$ is $d$-hidden in $W_{i},$ and $W_{i}$ is not a vortex segment, \cref{obs:special_vertices_hidden}  also implies that $P_{i}$ avoids $\Ball^{d}_{G}(F)$.
	As $\Cc^\star$ is $d$-scattered by definition, $(\Cc^\star,\Pp)$ is an $S$-rooted $d$-scattered railed nest that satisfies all the required properties.
\end{proof}

We now use \cref{lem_s_represented} twice to complete this case.

\begin{lemma}\label{lem_st_both_represented}
	Let $k$ be a non-negative integer and suppose $r \geq (k + 2)(d + 1)$ and $t \geq d + 1.$
	Suppose further that $W$ $k$-represents both $\rend_{S} \setminus \rend_{T}$ and $\rend_{T} \setminus \rend_{S}.$
	Then $G$ contains a $d$-scattered $S$-$T$-linkage of order $k$ that avoids $\Ball^{d}_{G}(F).$
\end{lemma}
\begin{proof}
	By \cref{obs:shaft}, we may find in $W$ a $d$-scattered shaft $\Cc^\star,$ which, since $t \geq d,$ consists of cycles disjoint from $\Ball^{d}_{G}(F).$
	Next, by applying \cref{lem_s_represented} twice --- once in the original form and once with the roles of $S$ and $T$ switched --- we find $d$-scattered linkages $\Pp_S$ and $\Pp_T$ such that $(\Cc^\star,\Pp_S)$ and $(\Cc^\star,\Pp_T)$ are $d$-scattered confined nests of order $(k+2,k)$ that are $S$- and $T$-rooted, respectively. \cref{lem_s_represented} also asserts that each path of $\Pp_S$ is disjoint from $\Ball^{d}_{G}(F)$ and $d$-hidden in a segment that represents $\rend_S\setminus \rend_T,$ while each path of $\Pp_T$ is also disjoint from $\Ball^{d}_{G}(F)$ and $d$-hidden in a segment that represents $\rend_T\setminus \rend_S$.
	Therefore, by \cref{obs:union-hidden} we conclude that $\Pp_S\cup \Pp_T$ is $d$-scattered, and hence $(\Cc^\star,\Pp_S\cup \Pp_T)$ is a $d$-scattered connector of order $k$. 

	Now, by definition of $\Cc^\star$ and the guarantees of \cref{lem_s_represented}, imply that the graph $H \coloneqq \bigcup (\Cc^\star \cup \Pp_{S} \cup \Pp_{T})$ is disjoint from $\Ball^{d}_{G}(F).$	
	It now remains to apply \cref{scattered_connector}.
\end{proof}

This concludes the cases where we find a scattered $S$-$T$-linkage directly from representation.

\subsection{Scattered paths in vortex segments}

Next, we examine under what circumstances we are able to extract a scattered $S$-$T$-linkage by partially finding it within the vortex~segments. We first need to formalize in what way a vortex segment may provide a partial linkage.

\paragraph{Consistency and partial linkages.}

Suppose for some $s \in \N_{\geq \max(2,d+1)},$ $G$ contains a railed nest $(\Cc, \Pp)$ of order $(s, s),$ sandwiched by $(\FDisk^{\insf}, \FDisk^{\outsf})$, where $\FDisk^{\insf} \subseteq \FDisk^{\outsf}$ is a pair of $\rend$-aligned disks. 
Let $\Cc = (C_1, \ldots, C_s)$ and enumerate the paths of $\Pp$ as $P_1, \ldots,P_s$ in the order of their endpoints on $C_s$.
For $i \in [s],$ let $\Delta_i$ be the disk bounded by $C_i$ that contains $\FDisk^\insf.$
Thus, $\FDisk^\insf\subseteq \Delta_1\subseteq \ldots \subseteq \Delta_s\subseteq \FDisk^\outsf$ and $\bd(\FDisk^\insf)\cap \bd(\Delta_1)=\bd(\FDisk^\outsf)\cap \bd(\Delta_s)=\emptyset$.

Let
\[\Af \coloneqq \FDisk^\outsf \setminus \inte(\Delta_{s - d}).\]
Note that $\Af$ is an annulus in $\Sigma$ that contains the $d+1$ outermost cycles of $\Cc$.
We say that a $(\FDisk^{\insf}, \FDisk^{\outsf})$-radial linkage $\Qq$ is \emph{$d$-consistent} with $(\Cc, \Pp)$ \emph{relative} to $(\FDisk^{\insf}, \FDisk^{\outsf})$ if \[\Qq \cap \Af = \bigcup_{i \in I} P'_{i},\]
where $I \subseteq [d + 1, s - (d + 1)]$ is a $d$-scattered set of indices such that $|I| = |\Qq|$ and $P'_{i}$ is the connected component of $P_{i} \cap \Af$ that is a $V(C_{s - d})$-$\per(\FDisk^\outsf)$-path, which due to orthogonality is uniquely defined.

Additionally, for a set of vertices $X\subseteq V(G)$, we say that a linkage $\Qq$ in $G$ is a \emph{$(d,(\Cc,\Pp),X)$-partial linkage} in $G$ \emph{relative} to $(\FDisk^{\insf}, \FDisk^{\outsf})$ if
\begin{itemize}
	\item $\Qq$ is an $X$-rooted confined $(\FDisk^{\insf}, \FDisk^{\outsf})$-radial linkage in $G$;
	\item $\Qq$ is $d$-scattered in $G \cap \FDisk^{\outsf};$ and
	\item $\Qq$ is $d$-consistent with $(\Cc,\Pp)$ relative to $(\FDisk^{\insf}, \FDisk^{\outsf}).$
\end{itemize}
Note that we require $d$-scatteredness only in the subgraph $G \cap \FDisk^\outsf$, and not in the whole graph. 

The following lemma shows how a partial linkage found within a vortex segment can be prolonged to a scattered railed nest that constitutes ``half'' of a scattered connector.

\begin{lemma}\label{scattered_railed_nest_from_vortex}
	Suppose $r, t \geq 2d + 1$ and $d' \leq d$ is a non-negative integer.
	Suppose further that $\Cc^\star$ is a $d$-scattered shaft in $W$ of order $s\in \N_{\geq 2}.$
	Let $W'$ be a vortex segment of $W,$ $(\Cc,\Pp)$ be the nest of $W',$ and $\FDisk^{\insf},\FDisk^{\outsf}$ be the inner and the outer vortex disk of $W'$ in $\Zz$, respectively.

	Suppose that for some $X \subseteq V(G),$  there is a $(d,(\Cc,\Pp),X)$-partial linkage $\Qq$ of order $p\in \N_{\geq 1}$ relative to $(\FDisk^{\insf}, \FDisk^{\outsf})$ that avoids $\Ball^{d'}_{G}(F).$
	Then in $G$ there is a $d$-scattered linkage $\Rr$ of order $p$ such that
	\begin{itemize}
		\item $(\Cc^\star,\Rr)$ is an $X$-rooted $d$-scattered confined railed nest of order $(s, p)$;
		\item every path of $\Rr$ is $d$-hidden in $W'$; and
		\item $\Rr$ avoids $\Ball^{d'}_{G}(F).$
	\end{itemize}
\end{lemma}
\begin{proof}
	Let $W_1$ be the $(r\times t)$-wall segment used to construct~$W'$.
	(That is, $W_1$ is the base of $W'.$)
	Also, let $\Cc=(C_1,\ldots,C_{t})$, so that $C_{1}$ bounds the inner disk of $\Cc$ whose interior contains $\FDisk^{\insf}$, and $C_{t}$ bounds the outer disk of $\Cc$ which is contained in the interior of $\FDisk^{\outsf},$ and let, for each $i \in [t],$ $\Delta_{i}$ denote the disk bounded by $C_i$ that contains $\FDisk^\insf.$
	Finally, let us enumerate $\Pp$ as $P_1,\ldots,P_t$ so that $P_i$ has an endpoint at the $i$-th top boundary vertex of $W_1$.
	
	As in the definition of a partial linkage, we define the annulus $\Af\coloneqq \FDisk^\outsf\setminus \inte(\Delta_{t-d}).$
	Thus, $\Af$ contains the $d+1$ outermost cycles of $\Cc$.
	Since $\Qq$ is a $(d,(\Cc,\Pp),X)$-partial linkage of order $p$, there is a $d$-scattered set $I\subseteq [d+1, t-(d+1)]$ such that $|I|=p$ and $\Qq$ can be enumerated as $\{Q_i\colon i\in I\}$ so that
	\[Q_i\cap \Af = P'_i,\qquad\textrm{for all }i\in I,\]
	where $P'_{i}$ denotes the connected component of $P_{i} \cap \Af$ that is a $V(C_{t - d})$-$\per(\FDisk^\outsf)$-path.
	Note here that by orthogonality, $P'_i$ is a uniquely defined subpath of $P_i.$
	Let $x_{i}$ denote the endpoint of $P'_{i}$ in $\per(\FDisk^\outsf).$
	
	Now, for every $i\in I$ we extend $Q_i$ to a path $R_i$, first using the subpath of $P_{i}$ connecting $x_{i}$ to its endpoint --- the $i$-th top boundary vertex of $W_{1},$ and then along the $i$-th column of $W_1$ until the first intersection with the outermost cycle of $\Cc^\star$; call it $C$. Let $M_i$ be the subpath of $R_i$ used for the extension; that is, $R_i$ is the concatenation of $Q_i$ and $M_i$.
	We define
	\[\Rr\coloneqq \{R_i\colon i\in I\}.\]
	
	Our goal now is to verify that the linkage $\Rr$ satisfies all the required properties.
	That $(\Cc^\star,\Rr)$ is a confined railed nest of order $(s,p)$ is clear from the construction.
	Also, since $\Qq$ is a $(p,(\Cc,\Pp),X)$-partial linkage, every path in $\Qq$ has an endpoint in $X \cap \FDisk^\insf$.
	So the same can be also said about every path in~$\Rr$, hence $(\Cc^\star,\Rr)$ is $X$-rooted.
	We are left with verifying that $\Rr$ is $d$-scattered and that every path of $\Rr$ is $d$-hidden in $W'$ and avoids $\Ball^{d'}_{G}(F).$
	We do this in the claims that follow.
	
	\begin{claim}
		For all distinct $i,j\in I$, we have $\dist_G(R_i,R_j)>d$.
	\end{claim}
	\begin{claimproof}
		Recall that $R_i$ is the concatenation of $Q_i$ and $M_i$. Further, $Q_i$ can be written as the concatenation of its prefix $Q_i'$ from the endpoint in $\FDisk^\insf$ until the first intersection with the cycle $C_{t+1-d}$, and its suffix $P_i'$. We decompose $R_j$ as the concatenation of $Q_j'$, $P_j'$, and $M_j$ analogously. We now argue that each of the subpaths $Q_i',P_i',M_i$ is at distance more than $d$ from each of the subpaths $Q_j',P_j',M_j$.
		
		Let us first argue that $\dist_G(Q_i',R_j)>d$. For contradiction, suppose $A$ is a path in $G$ of length at most $d$ that has one endpoint on $Q_i'$ and the other on $R_j$. Note that $Q_i'\subseteq \Delta_{t+1-d}$ by construction, hence due to the existence of the cycles $C_{t+1-d},\ldots,C_{t+1}$, $A$ must be entirely contained in $\FDisk^\outsf$. It follows that $A$ is actually a path of length at most $d$ connecting $Q_i'$ with $Q_j$ within $G\cap \FDisk^\outsf$. However, this is a contradiction with the assumption that $\Qq$ is $d$-scattered within $G\cap \FDisk^\outsf$ (following from $\Qq$ being a $(p,(\Cc,\Pp),X)$-partial~linkage sandwiched by $(\FDisk^{\insf}, \FDisk^{\outsf})$).
		
		A symmetric argument shows that $\dist_G(R_i,Q_j')>d$.
		
		We are left with proving that $\dist_G(P_i'\cup M_i,P_j'\cup M_j)>d$. Recall that since $I$ is $d$-scattered, we have $|i-j|>d$. Therefore, we may use
		\begin{itemize}
			\item the columns of $W_1$ with indices in $[i-d,i-1]\cup [i+1,i+d]$;
			\item the columns of $W_2$ with indices in $[i-d,i-1]\cup [i+1,i+d]$;
			\item the paths connecting the bottom boundary vertices of $W_2$ with the top boundary vertices of $W_1$, with indices in $[i-d,i-1]\cup [i+1,i+d]$;
			\item the cycles $C_{t+1-2d},\ldots,C_{t+1-d-1}$; and
			\item the $d$ bottommost rows of $W_1$;
		\end{itemize}
		to construct a nest $\wh{\Cc}$ of order $d$ such that the interior of the inner disk of $\wh{\Cc}$ contains $P_i'\cup M_i$, and the exterior of the outer disk of $\wh{\Cc}$ contains $P_j'\cup M_j$. Note here that as $\Cc^\star$ is a $d$-scattered shaft, none of the $d$ bottommost rows of $W_1$ intersects $M_i$.
		Since any path in $G$ connecting $P_i'\cup M_i$ and $P_j'\cup M_j$ must intersect each of the cycles of $\wh{\Cc}$ at an internal vertex, it follows that $\dist_G(P_i'\cup M_i,P_j'\cup M_j)>d$.
	\end{claimproof}
	
	\begin{claim}
		For every $i\in I$, the path $R_i$ is $d$-hidden in $W'$ and avoids $\Ball^{d'}_{G}(F).$
	\end{claim}
	\begin{claimproof}
		To see that every $R_{i}$ is $d$-hidden, we apply the criterion of \cref{obs:nest-criterion} while observing that a suitable nest of order $d$ can be constructed using 
		\begin{itemize}
			\item the $d$ leftmost and the $d$ rightmost columns of $W_1$;
			\item the $d$ leftmost and the $d$ rightmost columns of $W_2$;
			\item the $d$ outermost cycles of $\Cc$; and
			\item the $d$ bottommost rows of $W_1$.
		\end{itemize}
		Now, \cref{obs:special_vertices_hidden} implies that every $R_{i}$ avoids $\Ball^{d'}_{G}(u),$ for every $u \in F$ that is $d'$-hidden in a different vortex segment of $W$ than $W'.$
		If $u\in F$ is $d'$-hidden in $W',$ recall that $R_{i}$ is the concatenation of $Q_{i}$ and $M_{i},$ and that by assumption, $Q_{i}$ avoids $\Ball^{d'}_{G}(u).$
		So we only have to verify that $M_{i}$ also avoids $\Ball^{d'}_{G}(u).$
		For this, observe that by definition, $M_{i}$ is disjoint from $G \cap \inte(\Theta^\outsf)$ and that $\Delta_{d'} \subseteq \inte(\Theta^\outsf).$
		By \cref{obs:special_vertices_hidden} we have $\Ball^{d'}_{G}(u) \subseteq V(G \cap \Delta_{d'}),$ and therefore we conclude.
	\end{claimproof}
	
	Thus, the constructed linkage $\Rr$ satisfies all the required properties.
\end{proof}

With the previous lemma at hand we may now examine how we can combine partial linkages found in two distinct vortex segments, or between a vortex segment and terminals represented in flap segments, in order to find the desired $d$-scattered $S$-$T$-linkage.

\begin{lemma}\label{lem_s_or_t_represented_and_vortex}
	Suppose $k \geq 2$ is an integer and suppose $r \geq (k+2)(d+1)$ and $t \geq 2d + 1.$
	Assume that one of the following conditions holds:
	\begin{enumerate}
		\item For some vortex segment $W'$ of $W$, there is a $(d,(\Cc,\Pp),S\cup T)$-partial linkage $\Qq$ of order $2k,$ relative to $(\FDisk^{\insf}, \FDisk^{\outsf}),$ such that $k$ paths in $\Qq$ have an endpoint in $S\cap \FDisk^{\insf}$ and $k$ paths in $\Qq$ have an endpoint in $T \cap \FDisk^{\insf}$. Here, $(\Cc,\Pp)$ is the nest of $W',$ and $\FDisk^\insf,\FDisk^\outsf$ are the inner and the outer vortex disk of $W'$ in $\Zz$, respectively.\label{case:P2}
		\item There exist two distinct vortex segments $W_S$ and $W_T$ of $W$ such that there is a $(d,(\Cc_S,\Pp_S),S)$-partial linkage $\Qq_S$ of order $k,$ relative to $(\FDisk^{\insf}_{S}, \FDisk^{\outsf}_{S}),$ and a $(d,(\Cc_T,\Pp_T),T)$-partial linkage $\Qq_T$ of order $k,$ relative to $(\FDisk^{\insf}_{T}, \FDisk^{\outsf}_{T}).$ Here, for $A\in \{S,T\}$, $(\Cc_A,\Pp_A)$ is the nest of $W_A,$ and $\FDisk^\insf_{A},\FDisk^\outsf_{A}$ are the inner and the outer vortex disk of $W_{A}$ in $\Zz$, respectively.\label{case:PP}
		\item For some $\{A,B\}=\{S,T\}$, $W$ $k$-represents $\rend_{A}\setminus \rend_B$ and there exists a vortex segment $W_B$ of $W$ such that there is a $(d,(\Cc_B,\Pp_B),Y)$-partial linkage $\Qq_B$ of order $k,$ relative to $(\FDisk^{\insf}_{B}, \FDisk^{\outsf}_{B}).$ Here, $(\Cc_B,\Pp_B)$ is the nest of $W_B,$ and $\FDisk^\insf_{B},\FDisk^\outsf_{B}$ are the inner and the outer vortex disk of $W_{B}$ in $\Zz$, respectively.\label{case:RP}
	\end{enumerate}
	Then, $G$ contains a $d$-scattered $S$-$T$-linkage $\Qq$ of order $k$.
	Suppose further that $d' \leq d$ is a non-negative integer such that the partial linkages above avoid $\Ball^{d'}_{G}(F).$
	Then $\Qq$ also avoids $\Ball^{d'}_{G}(F).$
\end{lemma}
\begin{proof}
	By \cref{obs:shaft}, we may find a $d$-scattered shaft $\Cc^\star$ in $W$ that avoids $\Ball^{d'}_{G}(F).$
	
	We now analyze each of the cases separately.
	In each case, we show how to find a $d$-scattered $S$-$T$-connector of order $k$ where both its set of cycles and its set of paths avoid $\Ball^{d'}_{G}(F).$
	Then the conclusion follows by applying \cref{scattered_connector}.
	
	\medskip
	
	\noindent\textbf{Case i)} In this case, $W$, $\Cc^\star$, and $\Qq$ satisfy all requirements necessary to apply \cref{scattered_railed_nest_from_vortex}, with $X$ being the endpoints of the paths of $\Qq$ in $\FDisk^\insf$.
	This application yields a confined $d$-scattered railed nest $(\Cc^\star, \Rr)$ in $G$ of order $(k+2, 2k)$ such that $\Rr$ contains $k$ paths with an endpoint in $S \cap \FDisk^\insf$ and $k$ paths with an endpoint in $T \cap \FDisk^\insf,$ all avoiding $\Ball^{d'}_{G}(F).$
	This means that $(\Cc^\star, \Rr)$ is in fact a $d$-scattered $S$-$T$-connector of order $k$ disjoint from $\Ball^{d'}_{G}(F).$
	
	\medskip
	
	\noindent\textbf{Case ii)}
	Similarly to the previous case, we may apply \cref{scattered_railed_nest_from_vortex}
	to $W_S$, $\Cc^\star$, $\Qq_S$, and $X=S$, and to $W_T$, $\Cc^\star$, $\Qq_T$, and $X=T$.
	These two applications yield two $d$-scattered confined railed nests $(\Cc^\star, \Rr_{S})$ and $(\Cc^\star, \Rr_{T})$, both of order $(k + 2, k)$ and $S$- and $T$-rooted, respectively, such that $\Rr_S$ is $d$-hidden in $W_S,$ $\Rr_T$ is $d$-hidden in $W_T,$ and $\Rr_{S} \cup \Rr_{T}$ avoids $\Ball^{d'}_{G}(F).$
	Since $W_S$ and $W_T$ are distinct vortex segments, by \cref{obs:union-hidden} we conclude that $\Rr_S\cup \Rr_T$ is $d$-scattered as well.
	So $(\Cc^\star,\Rr_S\cup \Rr_T)$ is a $d$-scattered $S$-$T$-connector of order $k$ disjoint from $\Ball^{d'}_{G}(F).$
	
	\medskip
	
	\noindent\textbf{Case iii)} By symmetry, we may assume that $X=S$ and $Y=T$. Similarly as in the previous cases, we may apply \cref{scattered_railed_nest_from_vortex}
	to $W_T$, $\Cc^\star$, $\Qq_T$, and $X=T$. This application yields a $T$-rooted $d$-scattered confined railed nest $(\Cc^\star,\Rr_T)$ of order $(k+2, k)$ such that $\Rr_T$ is $d$-hidden in $W_T$ and avoids $\Ball^{d'}_{G}(F).$
	Next, since $W$ $k$-represents $\rend_S\setminus \rend_T$, we may apply \cref{lem_s_represented} to find an $S$-rooted $d$-scattered confined railed nest $(\Cc^\star,\Rr_S)$ of order $(k+2,k)$ such that every path of $\Rr_S$ is $d$-hidden in a flap segment that represents $\rend_S\setminus \rend_T$ and avoids $\Ball^{d'}_{G}(F)$ as well.
	From \cref{obs:union-hidden} we infer that in fact, $\Rr_S\cup \Rr_T$ is also $d$-scattered.
	So $(\Cc^\star,\Rr_S\cup \Rr_T)$ is a $d$-scattered connector of order $k$ where both $\Cc^\star$ and $\Rr_{S} \cup \Rr_{T}$ avoid $\Ball^{d'}_{G}(F)$ as desired.
\end{proof}

\section{Harvesting}

In this section we prove a ``local'' duality statement that applies to individual vortex segments. Roughly speaking, we show that within a vortex segment we can either construct a partial scattered linkage, say $X$-rooted for $X\in \{S,T\}$, or find a small number of balls of small radius that intersect any path that starts in a vertex of $X$ and escapes the vortex. Since the process of finding consecutive paths of a partial linkage resembles ``harvesting'' the paths, we call the corresponding statement the Harvesting Lemma.

For the remainder of this section we fix a planar graph $G$ with a vertex subset $F \subseteq V(G),$ sets of terminals $S,T\subseteq V(G)$, a rendition $\rend$ of $G$ in the sphere $\Sigma$, and distance parameter $d\in \N$.

\subsection{The Harvesting Lemma}

Let us start with some intuition.
Recall that thanks to \cref{structure_relative_terminals}, we may assume to be working with a large walloid $W \subseteq G$, of bounded breadth and depth.
Each vortex segment $W'$ of $W$ is equipped with its inner vortex disk $\FDisk^{\insf}$ and its outer vortex disk $\FDisk^{\outsf},$ and that $W$ has bounded depth means that there is a linear decomposition of bounded adhesion of the graph $G\cap \FDisk^\outsf$. Moreover, this decomposition is ``rooted'' at the vertices lying the boundary of $\FDisk^\outsf$. The harvesting lemma applies to such a situation abstractly, by providing a duality statement for any $\rend$-aligned disk $\Delta$ of bounded depth.

Besides harvesting $S$- or $T$-rooted paths that escape the vortex, we will also harvest $S$-$T$-paths that lie entirely in the vortex but interact non-trivially with the vortex's linear decomposition.
To facilitate this, we need one more definition. Let $\Delta$ be a $\rend$-aligned disk in $\Sigma$, and
let $\Bb = \langle B_{1}, \ldots, B_{n} \rangle$ be a linear decomposition of $G \cap \Delta$; we follow the convention that $B_{0} = B_{n + 1} = \emptyset.$
We say that a path $P \subseteq G$ is \emph{$\Bb$-dormant} if there exists $i \in [n]$ such that $V(P) \subseteq B_{i} \setminus (B_{i - 1} \cup B_{i + 1}).$

We are now ready to state and prove the Harvesting Lemma.

\begin{lemma}[Harvesting Lemma]\label{lem:harvesting_lemma} Let $k,w\in \N$. Let $\Delta$ be a $\rend$-aligned disk of depth at most $w\in \N;$ denote $H\coloneqq G\cap \Delta$.
	Then, there is a set $A\subseteq V(H)$ satisfying the following properties:
	\begin{itemize}
		\item we have $|A|\leq 2k(2w+1);$ and
		\item for each $X\in \{S,T\}$, if $H-(\Ball^{\lfloor \nicefrac{d}{2}\rfloor}_H(A) \cup \Ball^{\lfloor \nicefrac{d}{2} \rfloor}_{G}(F))$ contains an $X$-$\per(\Delta)$-path, then in fact there exists a $X$-$\per(\Delta)$-linkage $\Pp_{X}$ of order $k$ that is $d$-scattered in $H$ and avoids $\Ball^{\lfloor \nicefrac{d}{2} \rfloor}_{G}(F)$. Moreover, if this holds for both $S$ and $T$, then even $\Pp_S\cup \Pp_T$ is $d$-scattered in $H$.
	\end{itemize}
Further, if the graph $H-(\Ball^{\lfloor \nicefrac{d}{2}\rfloor}_H(A) \cup \Ball^{\lfloor \nicefrac{d}{2} \rfloor}_{G}(F))$ does not contain any $S$-$\per(\Delta)$-path, or it does not contain any $T$-$\per(\Delta)$-path, then one of the following conditions~holds:
\begin{itemize}
	\item there is a $d$-scattered $S$-$T$-linkage of order $k$ in $G$ that avoids $\Ball^{d}_{G}(F)$; or
	\item there is a set $A'\subseteq V(H)$ with $|A'|\leq 2kw$ such that every $S$-$T$-path in $$H-(\Ball^{d}_H(A) \cup \Ball^{\lfloor \nicefrac{d}{2}\rfloor}_H(A') \cup \Ball^{d}_{G}(F))$$ is $\Bb$-dormant.
\end{itemize}
\end{lemma}
\begin{proof}
	Denote $d'\coloneqq \lfloor \nicefrac{d}{2}\rfloor$ and $\Omega\coloneqq \per(\Delta)$ for brevity. Since $\Delta$ has depth $w$, we may enumerate the vertices of $\Omega$ as $x_{1}, \ldots, x_{n}$ in the order of encountering them when traversing $\bd(\Delta)$ in the clockwise direction, so that there is a linear decomposition $\Bb = \langle B_{1}, \ldots, B_{n} \rangle$ of $H$ of adhesion at most $w$ satisfying $x_i\in B_i$ for each $i\in [n]$. 
	
	For all $i, j \in [0,n]$ with $i<j,$ we define the graph \[H[i, j] \coloneqq H\left[\bigcup_{\ell \in [i+1,j]} B_{\ell}\right].\]
	Also, we denote
	\[W_i\coloneqq B_i\cap B_{i+1},\qquad \textrm{for }i\in [0,n],\]
	where $B_0=B_{n+1}=\emptyset$ by convention.
	Note that we have $|W_i|\leq w$ for each $i\in [0,n]$.
	
	We now apply an iterative procedure that greedily harvests $S$-$\Omega$-paths and $T$-$\Omega$-paths along the linear decomposition $\Bb$. The procedure will construct the following objects:
	\begin{itemize}
		\item A sequence of indices $1\leq p_1< p_2<\ldots<p_m\leq n$. We follow the convention that $p_0=0$.
		\item For each $i\in [m]$, a symbol $\Xi_i\in \{S,T\}$.
		\item For each $i\in [m]$, a $\Xi_i$-$\Omega$-path $P_i$ contained in the graph $$H[p_{i-1},p_i]\setminus \big( \Ball^{d'}_H(W_{p_{i-1}}\cup W_{p_i}) \cup \Ball^{d'}_{G}(F) \big).$$
	\end{itemize}
	The procedure iterates through the consecutive indices $i=1,2,3,\ldots$. Suppose $p_1<p_2<\ldots<p_{i-1}$ have already been constructed. The next step, of constructing $p_i,\Xi_i,P_i$, is executed as follows:
	
	\medskip
	\noindent\textbf{Step 1, choosing the type of the next path:} If among $\Xi_1,\ldots,\Xi_{i-1}$ there are already $k$ symbols $S$ and $k$ symbols $T$, then finish the procedure by putting $m\coloneqq i-1$. If there are $k$ symbols $S$ but fewer than $k$ symbols $T$, set $X\coloneqq T$. Similarly, if there are $k$ symbols $T$ but fewer than $k$ symbols $S$, set $X\coloneqq S$. Finally, if among $\Xi_1,\ldots,\Xi_{i-1}$ there are fewer than $k$ symbols $S$ and fewer than $k$ symbols $T$, set $X\coloneqq S\cup T$.
	
	\medskip
	\noindent\textbf{Step 2, stopping condition:} If the graph $H[p_{i-1},n]\setminus (\Ball^{d'}_H(W_{p_{i-1}}) \cup \Ball^{d'}_{G}(F))$ contains no $X$-$\Omega$-path, then finish the procedure by setting $m\coloneqq i-1$.
	
	\medskip
	\noindent\textbf{Step 3, executing the harvest:} Otherwise, we let $p_i$ be the smallest index in $[p_{i-1}+1,n]$ satisfying the following: the graph $H[p_{i-1},p_i]\setminus (\Ball^{d'}_H(W_{p_{i-1}}\cup W_{p_i}) \cup \Ball^{d'}_{G}(F))$ contains an $X$-$\Omega$-path $P_i$. We set $\Xi_i\coloneqq S$ if $P_i$ is an $S$-$\per(\Delta)$-path, and $\Xi_i\coloneqq T$ if $P_i$ is a $T$-$\Omega$-path.
	
	\medskip
	We now analyze the outcome of the procedure. Let us denote: 
	\begin{align*}
		&I_S\coloneqq \{i\in [m]\mid \Xi_i=S\},& &\quad I_T\coloneqq \{i\in [m]\mid \Xi_i=T\},\\ &\Pp_S\coloneqq \{P_i\colon i\in I_S\},& &\quad \Pp_T\coloneqq \{P_i\colon i\in I_T\},
	\end{align*}
	and
	\begin{align*}
		&A\coloneqq \bigcup_{i=1}^m W_{p_i-1}\cup W_{p_i}\cup \{x_{p_i-1}\},\\
		&H_i\coloneqq H[p_{i-1},p_i]\setminus (\Ball^{d'}_H(W_{p_{i-1}}\cup W_{p_{i}}) \cup \Ball^{d'}_{G}(F)), \textrm{ for each }i\in [m].
	\end{align*}
	Note that by construction, $P_i\subseteq H_i$ for each $i\in [m]$.
	Let us first verify that $\Pp\coloneqq \Pp_S\cup \Pp_T$ is a $d$-scattered linkage.
	
	\begin{claim}\label{cl:harvest-scattered}
		$\Pp$ is $d$-scattered in $H$.
	\end{claim}
	\begin{claimproof}
		Consider any pair of paths $P_i,P_j$, $i,j\in [m]$, $i<j$. Let $R$ be any path in $H$ with one endpoint, say~$u$, on $P_i$, and the second, say $v$, on $P_j$. We need to prove that $R$ has length more than $d$. Note that we have $P_i\subseteq H_i$ and $P_j\subseteq H_j$. As $\Bb$ is a linear decomposition of $H$, there must be a vertex $u'\in V(R)\cap W_{p_{i}}$ and a vertex $v'\in V(R)\cap W_{p_{j-1}}$ so that $u,u',v',v$ appear in this order on $R$ (possibly $u'=v'$). Since $H_i$ is disjoint with $\Ball^{d'}_H(W_{p_i})$, the prefix of $R$ from $u$ to $u'$ has length more than~$d'$. Similarly, the suffix of $R$ from $v'$ to $v$ also has length more than $d'$. So the length of $R$ is at least $2d'+2$, which is larger than $d$.
	\end{claimproof}

	Next, it is clear from the construction that $|I_S|\leq k$ and $|I_T|\leq k$. So we have $m\leq 2k$, and hence
	\[|A|\leq 2k(2w+1).\]
	Finally, we verify that if we did not manage to harvest $k$ paths of a particular type, then the union of radius-$d'$ balls around the vertices of $A$ and $F$ actually hits all the paths of this type.
	
	\begin{claim}\label{cl:no-harvest-hitting}
		Suppose $|\Pp_S|<k$. Then $\Ball^{d'}_H(A) \cup \Ball^{d'}_G(F)$ intersects every $S$-$\Omega$-path in $H$.
	\end{claim}
	\begin{claimproof}
		Suppose for contradiction that in the graph $H-(\Ball^{d'}_H(A) \cup \Ball^{d'}_G(F))$ there is an $S$-$\Omega$-path, say~$Q$. Observe that since $\bigcup_{i\in [m]} W_i$ is contained in $A$, $\Bb$ is a linear decomposition of $H$, and $Q$ is connected, $Q$ must be entirely contained in one of the following graphs:
		\[H_1,H_2,\ldots,H_m,H[p_m,n]\setminus (\Ball^{d'}_H(W_{p_m}) \cup \Ball^{d'}_G(F)).\]
		Since $|I_S|=|\Pp_S|<k$, the procedure did not stop in \textbf{Step 1} due to finding $k$ $S$-$\Omega$-paths and $k$ $T$-$\Omega$-paths, hence it stopped in \textbf{Step 2}. Moreover, as $|I_S|<k$, in this step we have $X=S$ or $X=S\cup T$. This implies that the graph $H[p_m,n]\setminus (\Ball^{d'}_H(W_{p_m}) \cup \Ball^{d'}_G(F))$ does not contain any $S$-$\Omega$-path. Hence $Q\subseteq H_i$ for some $i\in [m]$.
		
		By the minimality of the choice of $p_i$ in \textbf{Step 2}, it cannot happen that $Q$ is entirely contained in the graph $H[p_{i-1},p_i-1]\setminus (\Ball^{d'}_H(W_{p_{i-1}}\cup W_{p_{i}-1}) \cup \Ball^{d'}_G(F))$. Since $W_{p_{i}-1}\subseteq A$ and $Q$ is disjoint with $A$, $Q$ must be entirely contained in the graph $H[B_{p_i}]\setminus (\Ball^{d'}_H(W_{p_{i}-1}\cup W_{p_{i}}) \cup \Ball^{d'}_G(F))$. However, since $\Bb$ is a linear decomposition of $H$, the only vertex of $\Omega$ that may possibly belong to $B_{p_i}\setminus (W_{p_{i}-1}\cup W_{p_{i}})$ is~$x_{p_i}$. So $x_{p_i}$ must be the endpoint of $Q$ in $\Omega$, but $x_{p_i}\in A$; a contradiction.
	\end{claimproof}
	
	It now follows from \cref{cl:harvest-scattered}, \cref{cl:no-harvest-hitting} applied to $S$ and $T$, and \cref{cl:no-harvest-hitting} with the roles of $S$ and $T$ swapped, that $A$ satisfies all the required properties. Indeed, if $\Ball^{d'}_H(A)\cup \Ball^{d'}_G(F)$ does not intersect every $S$-$\Omega$-path, then by \cref{cl:no-harvest-hitting} we must have $|\Pp_S|=k$; and similarly for $T$-$\Omega$-paths.
	
	\medskip
	
	We are left with proving the last part of the lemma statement. By symmetry, assume that $\Ball^{d'}_H(A) \cup \Ball^{d'}_G(F)$ intersects all $S$-$\Omega$-paths in $H$. Call an $S$-$T$-path in $H$ {\em{invading}} if it is not $\Bb$-dormant and is disjoint with $\Ball^{d}_H(A) \cup \Ball^{d}_G(F)$. We first harvest invading paths.
	
	\begin{claim}\label{cl:harvest-invading}
		One of the following conditions holds:
		\begin{itemize}
			\item There is a linkage $\Qq$ of order $k$ that is $d$-scattered in $H$ and consists of invading paths.
			\item There is a set $A'\subseteq V(H)$ with $|A'|\leq 2kw$ such that $\Ball^{d'}_H(A')$ intersects every invading path in $H$.
		\end{itemize}
	\end{claim}
	\begin{claimproof}
		We perform a similar harvesting procedure as before, just this time we harvest invading paths. That is, we iteratively construct indices $0=q_0<q_1<q_2<\ldots$ according to the following rule: Once $q_0,\ldots,q_{i-1}$ are constructed, we pick $q_i$ as the smallest index in $[q_{i-1}+1,n]$ such that the graph $H[q_{i-1},q_{i}]\setminus \Ball^{d'}_H(W_{q_{i-1}}\cup W_{q_{i}})$ contains an invading path, say $Q_i$; or we finish the procedure if $H[q_{i-1},q_n]\setminus \Ball^{d'}_H(W_{q_{i-1}})$ does not contain any invading path. If the procedure managed to perform $k$ iterations and thus construct $k$ invading paths $Q_1,\ldots,Q_k$, then the same reasoning as in the proof of \cref{cl:harvest-scattered} shows that the linkage $\Qq\coloneqq \{Q_1,\ldots,Q_k\}$ is $d$-scattered in $H$, and thus constitutes a valid first outcome. On the other hand, if the procedure stopped after constructing $q_1,\ldots,q_{m'}$ for some $m'<k$, then the same reasoning as in the proof of \cref{cl:no-harvest-hitting} shows that $A'\coloneqq \bigcup_{i=1}^{m'} W_{q_i-1}\cup W_{q_i}$  has the property that $\Ball^{d'}_{H}(A')$ intersects every invading path. As $|A'|\leq 2kw$ due to $m'<k$, the set $A'$ constitutes a valid second~outcome.
	\end{claimproof}
	
	Now, if applying \cref{cl:harvest-invading} yields the second outcome, then by the definition of an invading path, every $S$-$T$-path that is left in $H\setminus (\Ball^{d}_H(A)\cup \Ball^{d'}_H(A')\cup \Ball^{d}_G(F))$ must be $\Bb$-dormant; so $A'$ satisfies the property postulated in the lemma statement. We are left with proving that if we obtain the first outcome --- an $S$-$T$-linkage $\Qq$ that is $d$-scattered in $H$ and consists of $k$ invading paths --- then in fact $\Qq$ is also $d$-scattered~in~$G$.
	
	Towards a contradiction, suppose in $G$  there is a path $R$ of length at most $d$ that connects two distinct paths $Q,Q'\in \Qq$. Let $u$ be the endpoint of $R$ on $Q$ and $v$ be the endpoint of $R$ on $Q'$. Since $\Qq$ is $d$-scattered in $H$, it cannot be that $R$ is entirely contained in $H$. Let then $u',v'$ be the first and the last intersection of $R$ with $\Omega$ when traversing $R$ in the direction from $u$ to $v$. Thus, $u,u',v',v$ appear in this order on $R$. Since the length of $R$ is at most $d$, either the prefix of $R$ from $u$ to $u'$ or the suffix of $R$ from $v'$ to $v$ has length at most $d'$; by symmetry, assume the former and call this prefix $R[u,u']$. Note that $Q\cup R[u,u']$ contains an $S$-$\Omega$-path entirely contained in $H$. Since $\Ball^{d'}_H(A) \cup \Ball^{d'}_G(F)$ must intersect this $S$-$\Omega$-path, and $R[u,u']$ has length at most $d'$, we conclude that either $\dist_H(Q,A)\leq d'+d'\leq d$ or that $\dist_{G}(Q,F) \leq d' + d' \leq d$. This is a contradiction with $Q$ being invading.
\end{proof}

\subsection{The Combing Lemma}

Recall that one of the outcomes of the harvesting lemma (\cref{lem:harvesting_lemma}) to the outer vortex disk $\FDisk^{\outsf}$ for a vortex segment $W,$ is a $d$-scattered $S$- and/or $T$-rooted linkage $\Pp$ in $G \cap \FDisk^{\outsf}.$
Our ultimate goal is to utilize $\Pp$ in combination with the tools developed in \cref{sec:finding}, particularly \cref{lem_s_or_t_represented_and_vortex} to find a $d$-scattered $S$-$T$-linkage in $G.$

However, in its current state, $\Pp$ cannot be used towards this as harvesting alone provides no control over the endpoints of $\Pp$ in $\per(\FDisk^{\outsf}),$ nor towards the consistency of $\Pp$ with the nest of $W,$ say $(\Cc, \Rr).$
For this, we need a tool that allows us to reroute the paths in $\Pp$ in order to obtain a $(d, (\Cc, \Rr, S \cup T))$-partial linkage $\Pp'$ in $G \cap \FDisk^{\outsf}.$
This will be done by utilizing the bidimensional infrastructure provided by the railed nest $(\Cc, \Rr)$ in order to ``comb'' the paths, through the desired paths of $\Rr.$

For this purpose, we import the linkage combing lemma from the work of Golovach, Stamoulis, and Thilikos~\cite{GolovachST23Combing}.
To state this result, we first present some additional definitions.

\medskip
We say that two linkages $\Ll$ and $\Rr$ in a planar graph $G$ are \emph{equivalent} if they have the same pairs of endpoints: $$\big\{ \{ s, t \} \mid \text{$\Ll$ contains an $s$-$t$-path} \big\} = \big\{ \{ s, t \} \mid \text{$\Rr$ contains an $s$-$t$-path} \big\}.$$
Further, let $p\in \N_{\geq 3}$ be odd and let $s\in [p]$ be odd as well. Suppose $(\Cc,\Pp)$ is a fully confined $(\Delta^\insf, \Delta^\outsf)$-railed nest of order $(p,q)$ in $G$, for some disk $\Delta^\insf\subseteq \Delta^\outsf$ and $q\in \N$. Enumerate the cycles of $\Cc$ as $C_1,\ldots,C_p$ from the innermost to the outermost, and the paths of $\Pp$ as $P_1,\ldots,P_q$ according to the order of their endpoints on $\bd(\Delta^\outsf)$. For a subset of indices $I \subseteq [q],$ we shall say that a linkage $\Ll$ in $G$ is \emph{$(s, I)$-combed} in $(\Cc, \Pp)$ if $$\bigcup \Ll \cap \Af_s \subseteq \bigcup_{i \in I} P_{i},$$
where $\Af_s$ is the annulus with boundary $C_{m+1+t}\cup C_{m+1-t}$, where $m,t$ are integers such that $p=2m+1$ and $s=2t+1$.

What follows is an abbreviated version of \cite[Corollary 2]{GolovachST23Combing} for planar graphs.

\begin{proposition}[Corollary 2, \cite{GolovachST23Combing}]\label{linkage_combing} There exist functions $f_{\ref{linkage_combing}}, g_{\ref{linkage_combing}} \colon \N^{2} \to \N$ such that the following holds. Suppose $d,k\in \N$, $s \in \N_{\geq 1}$ is odd, and $H$ is a planar graph drawn on the sphere $\Sigma$. Suppose further that 
	\begin{itemize}
		\item $\Delta^{\insf}$ and $\Delta^{\outsf}$ are two disks in $\Sigma$ such that $\Delta^{\insf} \subseteq \Delta^{\outsf}$;
		\item $(\Cc, \Pp)$ is a fully confined $(\Delta^{\insf}, \Delta^{\outsf})$-railed nest in $H$ of order $(p, q),$ where $p = f_{\ref{linkage_combing}}(d, k) + s$ is odd and $q \geq \frac{2d + 5}{2} \cdot g_{\ref{linkage_combing}}(d, k);$
		\item $\Ll$ is a $d$-scattered linkage in $H$ of order $k$ such that for every path $P\in \Ll$, both endpoints of $P$ are drawn in $\Upsilon\coloneqq \inte(\Delta^\insf)\cup \exte(\Delta^\outsf)$; and
		\item $I \subseteq [q]$ is a subset of indices such that $|I| > g_{\ref{linkage_combing}}(d, k) \cdot (d + 1).$
	\end{itemize}
	Then $H$ contains a $d$-scattered linkage $\Ll'$ equivalent to $\Ll$ such that \[\bigcup \Ll' \cap \Upsilon \subseteq \bigcup \Ll \cap \Upsilon\] and $\Ll'$ is $(s, I)$-combed in $(\Cc, \Pp).$
	Moreover, \[f_{\ref{linkage_combing}}(d, k) = \Oh((g_{\ref{linkage_combing}}(d, k))^{2})\qquad \textrm{and} \qquad g_{\ref{linkage_combing}}(d, k) = d \cdot 2^{\Oh(k)}.\]
\end{proposition}

\begin{figure}[ht]
    \centering
	\begin{subfigure}{0.3\textwidth}
		\centering
        \includegraphics[page=1,scale=1.2]{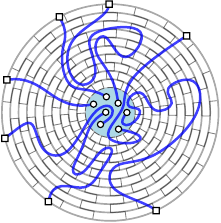}
		\caption{The $X$-rooted $d$-scattered confined radial linkage $\Pp$.}\label{subfig:combing_initial}
	\end{subfigure}
	~~
	\begin{subfigure}{0.3\textwidth}
		\centering
        \includegraphics[page=2,scale=1.2]{combing.pdf}
		\caption{The linkage $\Pp'$ combed in $(\Cc', \Rr')$.}\label{subfig:combing_combed}
	\end{subfigure}
	~~
	\begin{subfigure}{0.3\textwidth}
		\centering
        \includegraphics[page=3,scale=1.2]{combing.pdf}
		\caption{The linkage $\Qq$ combed and extended through the rails of $(\Cc, \Rr)$.}\label{subfig:combing_final}
	\end{subfigure}
    \caption{An illustration of how we obtain the $(d, (\Cc, \Rr), X)$-partial linkage $\Qq$ in the proof of \cref{combing_lemma}.}
    \label{fig:combing}
\end{figure}

Everything is set in place for the proof of the Combing Lemma.
Recall that we have fixed in the context: a planar graph $G$, its $\Sigma$-rendition $\rend$, and the distance parameter $d\in \N$.

\begin{lemma}[Combing Lemma]\label{combing_lemma} There exist functions $f_{\ref {combing_lemma}}, g_{\ref{combing_lemma}} \colon \N^{2} \to \N$ such that the following holds.
	Let $k\in \N$.
	Suppose $(\Cc, \Rr)$ is  a railed nest in $G$ of order $(p, q),$ where $p = f_{\ref{combing_lemma}}(d, k) \geq d$ and $q \geq g_{\ref{combing_lemma}}(d, k),$ and $(\Cc,\Rr)$ is sandwiched by $(\FDisk^{\insf}, \FDisk^{\outsf})$ for a pair of $\rend$-aligned disks $\FDisk^{\insf} \subseteq \FDisk^{\outsf}.$

	Suppose further that for some set $X\subseteq V(G)$, there exists a confined $(\FDisk^{\insf}, \FDisk^{\outsf})$-radial linkage $\Pp$ of order $k$ that is $X$-rooted and $d$-scattered in $G\cap \FDisk^\outsf$.

	Then in $G$ there is a $(d, (\Cc, \Rr), X)$-partial linkage $\Qq$ of order $k$ relative to $(\FDisk^\insf, \FDisk^\outsf).$ Moreover, we have $$\bigcup \Qq \cap \Delta_{d} \subseteq \bigcup \Pp \cap \Delta_{d},$$
	where $\Delta_d$ is the disk bounded by the $d$th innermost cycle of $\Cc$ that contains $\FDisk^\insf$.
	Finally, \[f_{\ref{combing_lemma}}(d, k) = d^{2} \cdot 2^{\Oh(k)}\qquad\textrm{and}\qquad g_{\ref{combing_lemma}}(d, k) = d^{3} \cdot 2^{\Oh(k)}.\]
\end{lemma}
\begin{proof} We begin by defining the functions $f_{\ref{combing_lemma}}$ and $g_{\ref{combing_lemma}}.$
	Let
	\begin{align*}
		f_{\ref{combing_lemma}}(d, k) \ &\coloneqq \ d + f_{\ref{linkage_combing}}(d, k) + 2d + 1\\
		g_{\ref{combing_lemma}}(d, k) \ &\coloneqq \ \nicefrac{2d + 5}{2} \cdot g_{\ref{linkage_combing}}(d, k) \cdot (d + 1) + 2d+2.
	\end{align*}
	The claimed asymptotic bounds on $f_{\ref{combing_lemma}},g_{\ref{combing_lemma}}$ then follow directly from the bounds of \cref{linkage_combing}.
	
	Let us start by providing some intuition.
	As the linkage we are looking for need only be $d$-scattered within $H \coloneqq G \cap \FDisk^{\outsf},$ our proof will only focus on $H.$
	To find the desired linkage, we apply \cref{linkage_combing} to $\Pp,$ and subsequently we alter the resulting linkage, to obtain a linkage that is $d$-consistent with $(\Cc, \Rr)$ relative to $(\FDisk^{\insf}, \FDisk^{\outsf}).$
	What we have to show is that the altered linkage remains $d$-scattered.
	We now dive into the details.
	
	Let $C_{1}, \ldots, C_{p}$ be the cycles of $\Cc$ ordered from innermost to outermost, so that $\Delta_d$ is bounded by $C_d$. Let $\Cc' \coloneqq (C_{d + 1}, \ldots, C_{p}).$
	Thus $C_{d + 1}$ bounds a disk $\Delta^{\insf}$ that contains $\FDisk^{\insf}$ and avoids $C_{p},$ and $C_{p}$ bounds a disk $\Delta^{\outsf}$ that contains $\Delta^{\insf}$ and is contained in $\FDisk^\outsf$.
	We let $\Af \coloneqq \Delta^\outsf\cap (\Sigma\setminus \inte(\Delta^\insf)),$ i.e, $\Af$ is the annulus with boundary $C_{d + 1} \cup C_p.$
	Also let $\Upsilon \coloneqq \Sigma \setminus \Af.$
	
	We define \[\Rr' \coloneqq \{ R \cap \Af \colon R \in \Rr \}.\]
	Since $\Rr$ is orthogonal to $\Cc,$ $R \cap \Af$ is connected for every $R \in \Rr,$ and therefore $\Rr'$ is a linkage in $H$ of order $q$ as well.
	In particular, we may now observe that $(\Cc', \Rr')$ is a fully confined $(\Delta^\insf, \Delta^\outsf)$-railed nest in $H$ of order $(p', q)$, where $p'\coloneqq p-d$.
	By the assumptions on $p$ and $q,$ this nest satisfies the prerequisites of \cref{linkage_combing} for $s=2d+1$.
	Moreover, since $\Pp$ is a $(\FDisk^{\insf}, \FDisk^{\outsf})$-radial linkage, we have that both endpoints of each path in $\Pp$ are drawn in $\Upsilon,$ as desired for the application of \cref{linkage_combing}.
	Let $(R_{1}, \ldots, R_{q})$ be an enumeration of $\Rr$ that respects the ordering of their endpoints along $\bd(\FDisk^{\outsf}).$
	We may now apply \cref{linkage_combing} to the graph $H,$ the nest $(\Cc',\Rr')$, and the linkage $\Pp,$ with $s=2d+1$ and $I \subseteq [d+1, q - (d + 1)]$ chosen maximally so that it is $d$-scattered.
	Observe that $q$ is assumed to be large enough so that $|I| > g_{\ref{linkage_combing}}(k,d) \cdot (d + 1),$ as required.
	Let $\Pp'$ be the outcome of this application.
	
	Noting that $p'$ is odd due to $f_{\ref{linkage_combing}}(d,k)$ being always even, we let $m$ be the integer such that $p' = 2m + 1.$
	Consider the nest $\Cc^\star = (C_{m+1}, \ldots, C_{m+d+1}),$ and let $\Delta^{-}$ and $\Delta^+$ be the inner and the outer disk of this nest, respectively. Note that the nest $\Cc^\star$ is entirely contained in the annulus $\Af_{2d+1}$, as defined in the definition of a combed linkage for the nest $(\Cc',\Rr')$.
	Also, consider the annuli $\Af' \coloneqq \Delta^{+} \setminus \mathsf{int}(\Delta^{-}),$ $\Af'' \coloneqq \FDisk^{\outsf} \setminus \inte(\Delta^{-}),$ $\Af''' \coloneqq \FDisk^{\outsf} \setminus \inte(\Delta^{+}),$ and note that $\Af' \cup \Af''' = \Af''.$

	Fix a path $P \in \Pp'.$
	Observe that since $\Pp'$ is $(2d+1,I)$-combed in $(\Cc',\Rr')$, there is a unique connected component of $P \cap \Delta^{-}$ that contains the endpoint of $P$ drawn in $\FDisk^{\insf}.$
	Let $P^\star$ be this component.
	Also, since $\Rr$ is orthogonal to $\Cc,$ for every path $R \in \Rr,$ there is a unique connected component of $R \cap \Af''$ that is a $\per(\Delta^{-})$-$\per(\FDisk^{\outsf})$-path.
	Let $R^{\star}$ be this component.

	Since $\Pp'$ is $(2d + 1, I)$-combed in $(\Cc', \Rr')$, for every $P\in \Pp'$ there is a unique path $R_P \in \Rr$ such that $P^\star \cap R_P^\star$ is non-empty and it is in fact a (possibly length-$0$) subpath of $C_{m + 1}.$
	We define \[\Qq \coloneqq \{ P^\star \cup R^\star_{P} \mid P \in \Pp' \}\] and claim that $\Qq$ is the desired $(d, (\Cc, \Rr), X)$-partial linkage relative to $(\FDisk^{\insf}, \FDisk^{\outsf})$.
	That $\Qq$ is a confined $(\FDisk^{\insf}, \FDisk^{\outsf})$-radial linkage in $H$ of order $k$ follows directly by construction.
	Moreover, since $I$ has been chosen to be $d$-scattered, it also follows that $\Qq$ is $d$-consistent with $(\Cc, \Rr)$ relative to $(\FDisk^{\insf}, \FDisk^{\outsf}).$ Let us verify the compatibility with the disk $\Delta_d$.

	\begin{claim} We have that $$\bigcup \Qq \cap \Delta_{d} \subseteq \bigcup \Pp \cap \Delta_{d}.$$
	\end{claim}
	\begin{claimproof} By the guarantees of \cref{combing_lemma}, we have that $\bigcup \Pp' \cap \Upsilon \subseteq \bigcup \Pp \cap \Upsilon.$
		Moreover, by the definition of $\Qq,$ we have that $\bigcup \Qq \cap \Delta^{-} \subseteq \bigcup \Pp' \cap \Delta^{-}.$
		Then, since $\Delta_{d} \subseteq \Upsilon$ and $\Delta_{d} \subseteq \Delta^{-},$ we conclude that $\bigcup \Qq \cap \Delta_{d} \subseteq \bigcup \Pp \cap \Delta_{d}$ as desired.
	\end{claimproof}

	We now see each path $Q = P^\star \cup R^\star_{P} \in \Qq$ as the concatenation of three paths $Q^{1},$~$Q^{2},$ and $Q^{3}$, defined as follows:
	$Q^{1} = P^\star,$ $Q^2\coloneqq R^\star_{P} \cap \Af'$, and $Q^{3} = R^\star_{P} \cap \Af'''$. Note that $Q^1\cap Q^2$ is a subpath of $C_{m+1}$, and $Q^2\cap Q^3$ is a subpath of $C_{m+d+1}$.  

	It remains to show that $\Qq$ is $d$-scattered in $H$.

	\begin{claim} For any two distinct paths $Q_{1}, Q_{2} \in \Qq,$ we have $\dist_{H}(Q_{1}, Q_{2}) > d.$
	\end{claim}
	\begin{claimproof} Consider any $u$-$v$-path $Z$ in $H$, for any choice of $u \in V(Q_{1})$ and $v \in V(Q_{2}).$
		What we have to show is that the length of $Z$ is larger than $d.$
		We distinguish three cases up to symmetry.

	\medskip\noindent
	\textbf{Case 1:} Assume that $u \in V(Q^{1}_{1} \cup Q^{2}_{1})$ and that $v \in V(Q^{1}_{2} \cup Q^{2}_{2}).$
	In this case, by definition, both $u$ and $v$ are vertices that come from distinct paths of $\Pp'.$
	Since $\Pp'$ is $d$-scattered in $H,$ our claim follows.
	
	\medskip\noindent
	\textbf{Case 2:} Assume that $u \in V(Q^{1}_{1})$ and that $v \in V(Q^{3}_{2}) \setminus V(Q^{2}_{2}).$
	In this case $u$ is drawn in $\Delta^{-}$ while $v$ is drawn in $\Af''' \setminus \bd(\Delta^{+}).$
	Moreover, every cycle of $\Cc^\star$ is drawn in $\Af'.$
	This implies that all the cycles of $\Cc^\star$, except possibly $C_{m+1}$, separate $u$ from $v.$
	Since $\Cc^\star$ is a nest of order $d+1$, and $Z$ necessarily intersects every cycle of $\Cc^\star\setminus \{C_{m+1}\}$ at an internal vertex, the length of $Z$ is larger than $d.$
	
	\medskip\noindent
	\textbf{Case 3:} Assume that $u \in V(Q_{1}) \setminus V(Q^{1}_{1})$ and that $v \in V(Q^{3}_{2}) \setminus V(Q^{2}_{2}).$
	First of all, observe that similarly to the previous case, since $v$ is drawn in $\Af''' \setminus \bd(\Delta^{+}),$ the existence of $\Cc^\star$ shows that the length of $Z$ is larger than $d$ whenever $Z$ contains a vertex drawn in $\Delta^{-}.$
	Hence we may assume $Z$ to be fully drawn within $\Af''.$
	Now, if $i \in [q]$ is the index in the enumeration of $\Rr$ so that $Q_{1} \setminus Q^{1}_{1} \subseteq R_{i},$ since $I$ is chosen to be $d$-scattered, the paths $R_{i - 1 - d}, \ldots, R_{i - 1}$ and $R_{i + 1} \ldots, R_{i + 1 + d}$ are all disjoint from $Q^{3}_{2},$ and in fact must separate $u$ from $v$ in $G \cap \Af''.$
	Since in both directions there are at least $d$ paths which $Z$ must intersect at an internal vertex, we conclude that the length of $Z$ is larger than $d.$
	\end{claimproof}

	This concludes the proof.
\end{proof}

\section{The Local Structure Theorem}\label{sec:local}

In this section we prove the central decomposition result, describing the structure in the graph relative to a large well-linked set.

\begin{theorem}[Local Structure Theorem]\label{local_structure_2} There exist functions $\link_{\ref{local_structure_2}}, \sep_{\ref{local_structure_2}}, \cov_{\ref{local_structure_2}} \colon \N^{3} \to \N$ such that for all integers $d,k \geq 0$ and $q \geq 1,$ the following holds.
If $G$ is a planar graph with a vertex subset $F \subseteq V(G)$ of size less than $q$, a well-linked set $X\subseteq V(G)$ of size $\link_{\ref{local_structure_2}}(d, k, q)$ and $S, T \subseteq V(G)$, then one of the following conditions holds:
	\begin{enumerate}
		\item $G$ contains a $d$-scattered $S$-$T$-linkage of order $k$ that avoids $\Ball^{\lfloor \nicefrac{d}{2} \rfloor}_{G}(F)$; or
		\item there exists a non-empty set $\{ (A_{i}, B_{i}) \colon i \in [n] \}$ of separations of $G,$ for some $n \in \N_{\geq 1},$ such that
		\begin{itemize}
			\item for every $i \in [n],$ $|A_{i} \cap B_{i}| \leq \sep_{\ref{local_structure_2}}(d, k, q)$;
			\item for every $i \in [n],$ $B_{i}$ is the $X$-majority side of $(A_{i}, B_{i});$ and
			\item for all distinct $i,j \in [n],$ we have $(A_{i} \setminus B_{i}) \cap (A_{j} \setminus B_{j}) = \emptyset.$
		\end{itemize}
		Further, assuming that $$A \coloneqq \bigcup_{i \in [n]} A_{i} \quad\text{and}\quad B \coloneqq V(G) \setminus \bigcup_{i \in [n]} (A_{i} \setminus B_{i}),$$ there exists a set $Y \subseteq A \cap B$ of size at most $\cov_{\ref{local_structure_2}}(d, k, q)$ such that every $S$-$T$-path in $G$ that avoids $\Ball^{d}_{G}(F)$ and intersects $B$ also intersects $\Ball^{d}_{G[A]}(Y).$
	\end{enumerate}
Moreover, assuming that $k' \coloneqq \max\{ k, q \}$ it holds that
	$$\link_{\ref{local_structure_2}}(d, k, q) \in d^{\Oh(1)} \cdot 2^{\Oh(k' \log k')} \quad\text{and}\quad \sep_{\ref{local_structure_2}}(d, k, q), \cov_{\ref{local_structure_2}}(d, k, q) \in d^{3} \cdot 2^{\Oh(k' \log k')}.$$
\end{theorem}

Before we proceed with the proof, let us discuss the functions $\link_{\ref{local_structure_2}}, \sep_{\ref{local_structure_2}},$ and $\cov_{\ref{local_structure_2}}.$

First we introduce an auxiliary function $$\mathsf{t}(d, k) \coloneqq \max\{ f_{\ref{combing_lemma}}(d, 2k), g_{\ref{combing_lemma}}(d, 2k), 2d + 1 \} \in d^{3} \cdot 2^{\Oh(k)},$$ which determines the order required for the railed nests of the vortex segments of the walloid obtained from \cref{structure_relative_terminals}.
This choice ensures that we have sufficient infrastructure to apply the Combing Lemma (\cref{combing_lemma}) as well as \cref{lemma:union_represented,lem_st_both_represented,lem_s_or_t_represented_and_vortex} later in the proof.

Also, to force the vertices of $F$ to be confined within the inner vortex disks of the vortex segments we will apply \cref{structure_relative_terminals} with the representation parameter being $k' \coloneqq \max\{ k, q \}.$

This requirement influences the adhesion of the linear decompositions admitted by the outer vortex disk of each vortex segment, which is set to $$\depth_{\ref{structure_relative_terminals}}(\mathsf{t}(d, k), k') \in d^{3}  \cdot 2^{\Oh(k' \log k')}.$$

Consequently, the order of separation required to separate each bag of the linear decomposition from the remainder of the graph is $$\sep_{\ref{local_structure_2}}(d, k, q) \coloneqq 2 \cdot \depth_{\ref{structure_relative_terminals}}(\mathsf{t}(d, k), k') + 1 \in d^{3}  \cdot 2^{\Oh(k' \log k')}.$$

Next, we define a second auxiliary function $$\mathsf{r}(d, k, q) \coloneqq \max\{ \sep_{\ref{local_structure_2}}(d, k, q) + 1, (k + 2)(d + 1) \} \in d^{3} \cdot 2^{\Oh(k' \log k')},$$
which determines both the number of cycles in the base annulus of the resulting walloid and the order of its wall segment.
This parameter must be sufficiently large to allow the application of \cref{lem_s_or_t_represented_and_vortex}.
In addition, it must ensure that every separation induced by a bag of the linear decomposition of a vortex segment is oriented by the walloid away from the bag, which corresponds to the $X$-majority side of that separation.

For these conditions to hold, the size of $X$ must satisfy $$\link_{\ref{local_structure_2}}(d, k, q) \coloneqq c \cdot (3 \cdot \mathsf{annulus}_{\ref{structure_relative_terminals}}(\mathsf{r}(d, k), \mathsf{t}(d, k), k'))^{20} \in d^{\Oh(1)} \cdot 2^{\Oh(k' \log k')},$$ where $c$ is the constant from \cref{thm_algogrid}.

Finally, the size of the hitting set obtained follows from the bounds in \cref{lem:harvesting_lemma}, yielding $$\cov_{\ref{local_structure_2}}(d, k, q) \coloneqq 2k (3 \cdot \depth_{\ref{structure_relative_terminals}}(\mathsf{t}(d, k), k') + 1) \cdot \mathsf{breadth}_{\ref{structure_relative_terminals}}(k') \in d^{3} \cdot 2^{\Oh(k' \log k')},$$ where the depth of the walloid is determined by the requirements described above.

The claimed asymptotic bounds follow directly from the bounds on the corresponding functions.

\begin{proof}[Proof of \cref{local_structure_2}] First, assume that $G$ is not connected.
	By the definition of well-linkedness it follows that there is a unique connected component $C$ of $G$ such that $V(C)$ is the $X$-majority side of the separation $(V(G) \setminus V(C), V(C))$ of $G,$ which is of order $0.$
	Clearly $(V(G) \setminus V(C), V(C))$ satisfies all required assumptions independent of the set of separations of $C$ we may choose in the future.
	So we may freely add $(V(G) \setminus V(C), V(C))$ to the constructed family of separations and assume hereafter without loss of generality that $G$ is connected.

	Let $t \coloneqq \mathsf{annulus}_{\ref{structure_relative_terminals}}(\mathsf{r}(d, k), \mathsf{t}(d, k), k').$
	Our proof starts with an application of \cref{thm_algogrid} with $X$ and $t.$
	As a result, we obtain a $(3t \times 3t)$-wall $W^0$ in $G$ controlled by $X.$
	From $W^{0},$ we obtain a $(t \times t)$-annulus wall $W^{1}$ in $G$ as follows.
	We define the set $\Cc$ of the $t$ base cycles of $W^{1}$ by iteratively peeling off $t$ layers of $W^0$ starting with its perimeter.
	Notice that $\bigcup \Cc$ intersects all but $t$ many vertical paths of $W^0.$
	By trimming these vertical paths to have their endpoints on the inner and outer cycle of $\Cc$ we obtain the desired $t$ radial paths that together with $\Cc$ make up $W^{1}.$
	Clearly, $W^{1}$ is controlled by $W^0$ and hence also controlled by~$X.$

	With $W^1$ in hand, we may invoke \cref{structure_relative_terminals}, and obtain a $\Sigma$-rendition $\rend$ of $G$ and a $\rend$-well-grounded $(\mathsf{r}(d, k, c), \mathsf{t}(d, k), a, b)$-walloid $W \subseteq G$ for some $a \in \N$ and $b \in [\max \{1, \mathsf{breadth}_{\ref{structure_relative_terminals}}(k') \}],$ along with a sandwich sequence $\langle\FDisk^{\insf}_{j},\FDisk^{\outsf}_{j}\colon j\in [b]\rangle$ such that the following conditions hold:
	\begin{enumerate}
		\item $W$ is controlled by $W^1$ (and therefore by $X$ as well);
		\item $W$ has depth at most $\depth_{\ref{structure_relative_terminals}}(\mathsf{t}(d, k), k')$ with respect to $\rend;$
		\item $F \subseteq \bigcup_{j \in [b]} \inte(\Theta^\insf_j)$; and
		\item for each $\mathfrak{R} \in \{ \rend_{S} \cap \rend_{T}, \rend_{S} \setminus \rend_{T}, \rend_{T} \setminus \rend_{S} \},$ either $$\bigcup \mathfrak{R} \subseteq \bigcup_{j \in [b]} \mathsf{int}(\FDisk^{\mathsf{in}}_{j})$$ or $W$ $k$-represents $\mathfrak{R}$ (since $k' \geq k$).
	\end{enumerate}

	Let $w \coloneqq \depth_{\ref{structure_relative_terminals}}(\mathsf{t}(d, k), k').$
	For each $j \in [b],$ let $W_{j}$ denote the $j$-th vortex segment of $W,$ so that $\FDisk^\insf_j,\FDisk^\outsf_j$ are the inner and the outer vortex disk of $W_j$, respectively. Let $H_{j} \coloneqq G \cap \FDisk^{\outsf}_{j},$ $(\Cc_{j}, \Pp_{j})$ be the nest of $W_{j},$ and $\Delta_{j}$ be the disk bounded by the $d$-th innermost cycle of $\Cc_{j}.$
	Also, let $\langle x^{j}_{1}, \ldots, x^{j}_{n_{j}} \rangle$ be the enumeration of the vertices in $\per(\FDisk^{\outsf}_{j})$ in the order of encountering them when traversing $\bd(\FDisk^{\outsf}_{j})$ in the clockwise direction, with the first vertex $x^{j}_{1}$ chosen arbitrarily from $\per(\FDisk^{\outsf}_{j})$, and $\Bb_{j} = \langle B^{j}_{1}, \ldots, B^{j}_{n_{j}} \rangle$ be a linear decomposition of $H_{j}$ of adhesion at most $w$ such that $x^{j}_{i} \in B^{j}_{i}$ for all $i \in [n_{j}]$. Such a decomposition  exists because $W$ has depth at most $w.$
	We also make the convention that $B^{j}_{0} = B^{j}_{n_{j} + 1} = \emptyset.$

	Let $d' \coloneqq \lfloor \nicefrac{d}{2} \rfloor.$
	Suppose further that $G$ contains no $d$-scattered $S$-$T$-linkage of order $k$ that avoids $\Ball^{d'}_{G}(F)$.
	In what follows we show that we may always conclude with the second outcome of our claim.

	\begin{claim}\label{claim:local_1} If for some $Z \in \{ S, T \},$ $W$ does not $k$-represent $\rend_{Z},$ then $$Z \subseteq \bigcup_{j \in [b]} \inte(\FDisk^{\insf}_{j}).$$
		Moreover, for at least one $Z \in \{ S, T \},$ $W$ does not $k$-represent $\rend_{Z}.$
	\end{claim}
	\begin{claimproof} 	By definition of the sets $\rend_{S}$ and $\rend_{T}$ and our assumption that $G$ is connected, every vertex of $S$ and $T$ is drawn in a cell of $\rend_{S}$ and $\rend_{T}$ respectively.
	Therefore, condition iv) of the application of \cref{structure_relative_terminals} implies the first part of the claim.
		
	Assume towards contradiction that the second part of the claim is false for both $\rend_{S}$ and $\rend_{T}.$
	Then, condition iv) of the application of \cref{structure_relative_terminals} above, implies that either $W$ $k$-represents $\rend_{S} \cap \rend_{T},$ or $W$ $k$-represents both $\rend_{S} \setminus \rend_{T}$ and $\rend_{T} \setminus \rend_{S}.$
	We treat these two cases separately.

	\medskip\noindent
	\textbf{$W$ $k$-represents $\rend_{S} \cap \rend_{T}$:} By definition $\mathsf{t}(d, k) \geq d + 1.$
	Hence, by applying \cref{lemma:union_represented} to $W,$ we infer that there exists a $d$-scattered $S$-$T$-linkage of order $k$ in $G$ that avoids $\Ball^{d'}_{G}(F)$; this contradicts our assumptions.

	\medskip\noindent
	\textbf{$W$ $k$-represents $\rend_{S} \setminus \rend_{T}$ and $\rend_{T} \setminus \rend_{S}$:} Similarly, $\mathsf{t}(d, k) \geq d + 1$ and $\mathsf{r}(d, k, c) \geq (k + 2)(d + 1)$ by definition.
	Hence, by applying \cref{lem_st_both_represented} to $W,$ we infer that there exists a $d$-scattered $S$-$T$-linkage of order $k$ in $G$ that avoids $\Ball^{d'}_{G}(F)$; this contradicts our assumptions.
	\end{claimproof}

	\begin{claim}\label{claim:local_2} For every $j \in [b],$ there exists a set $I_{j} \subseteq V(H_{j})$ with the following properties:
		\begin{itemize}
			\item we have $|I_{j}| \leq 2k(3w + 1)$;
			\item there exists $Z \in \{ S, T \}$ such that $$Z \subseteq \bigcup_{j \in [b]} \inte(\FDisk^{\insf}_{j}),$$ and for all $j \in [b],$ the graph $H_{j} - (\Ball^{d}_{H_{j}}(I_{j}) \cup \Ball^{d}_{G}(F))$ contains no $Z$-$\per(\FDisk^{\outsf}_{j})$-path; and
			\item every $S$-$T$-path in $H_{j} \setminus (\Ball^{d}_{H_{j}}(I_{j}) \cup \Ball^{d}_{G}(F))$ is $\Bb_{j}$-dormant.
		\end{itemize}
	\end{claim}
	\begin{claimproof} Recall that for every $j \in [b],$ $\FDisk^{\outsf}_{j}$ is a $\rend$-aligned disk of depth at most $w.$
		Hence we may apply \cref{lem:harvesting_lemma} to each $\FDisk^{\outsf}_{j},$ $j \in [b],$ with $k$ and $w,$ and with distance parameter $d,$ in order to obtain a set $I^{1}_{j} \subseteq V(H_{j})$ of size at most $2k(2w + 1)$ such that the following two conditions are satisfied. Let $$H'_{j} \coloneqq H_{j} \setminus (\Ball^{d'}_{H_{j}}(I^{1}_{j}) \cup \Ball^{d'}_{G}(F)).$$
		Then:
		\begin{enumerate}
			\item\label{c:EP} For each $X \in \{ S, T \},$ if the graph $H'_{j}$ contains an $X$-$\per(\FDisk^{\outsf}_{j})$-path, then in fact there exists an $X$-$\per(\FDisk^{\outsf}_{j})$-linkage $\Pp^{j}_{X}$ of order $k$ that is $d$-scattered in $H_{j}$ and avoids $\Ball^{d'}_{G}(F).$
			Moreover, if this holds for both $S$ and $T,$ then even $\Pp_{S}^{j} \cup \Pp_{T}^{j}$ is $d$-scattered in $H_{j}.$
			\item\label{c:noTwo} Further, if the graph $H'_{j}$ does not contain any $S$-$\per(\FDisk^{\outsf}_{j})$-path or any $T$-$\per(\FDisk^{\outsf}_{j})$-path, since we assume that $G$ contains no $d$-scattered $S$-$T$-linkage of order $k$ that avoids $\Ball^{d'}_{G}(F)$, then there is a set $I^{2}_{j} \subseteq V(H_{j})$ with $|I^{2}_{j}| \leq 2kw$ such that every $S$-$T$-path in $$H_{j} \setminus \big((\Ball^{d}_{H_{j}}(I^{1}_{j}) \cup \Ball^{d'}_{H_{j}}(I^{2}_{j})) \cup \Ball^{d}_{G}(F)\big)$$ is $\Bb_{j}$-dormant.
		\end{enumerate}

		Whenever for some $j \in [b]$ and $X \in \{ S, T\},$ $H'_{j}$ contains no $X$-$\per(\FDisk^{\outsf}_{j})$-path, we follow the convention that $\Pp^{j}_{X} \coloneqq \emptyset.$
		We distinguish cases.

		\medskip\noindent
		\textbf{$W$ $k$-represents $\rend_{S} \setminus \rend_{T}$:} By \cref{claim:local_1} it must be that $W$ does not $k$-represent $\rend_{T}$ and as a result, $T \subseteq \bigcup_{j \in [b]} \inte(\FDisk^{\insf}_{j}).$ Under these assumptions we show that for all $j \in [b],$ the graph $H'_{j}$ contains no $T$-$\per(\FDisk^{\outsf}_{j})$-path.

		Assume towards contradiction that for some $j \in [b],$ $\Pp^{j}_{T} \neq \emptyset.$
		Since $T \cap V(H_{j}) \subseteq \inte(\FDisk^{\insf}_{j}),$ condition~\ref{c:EP} above guarantees that $\Pp^{j}_{T}$ is a $T$-rooted $(\FDisk^{\insf}, \FDisk^{\outsf})$-radial linkage of order $k$ that is $d$-scattered in $H_{j}$ and avoids $\Ball^{d'}_{G}(F).$
		By the definition of vortex segments, $(\Cc_{j}, \Pp_{j})$ is a railed nest in $G$ sandwiched by $(\FDisk^{\insf}_{j}, \FDisk^{\outsf}_{j}).$
		Also, by assumption, the order of $(\Cc_{j}, \Pp_{j})$ is $(\mathsf{t}(d, k), \mathsf{t}(d, k)),$ which suffices to call \cref{combing_lemma} with $(\Cc_{j}, \Pp_{j}),$ $T,$ $\Pp^{j}_{T},$ and distance parameter $d.$
		This application yields as a result a $(d, (\Cc_{j}, \Pp_{j}), T)$-partial linkage $\Rr_{j}$ of order~$k$ relative to $(\FDisk^{\insf}_{j}, \FDisk^{\outsf}_{j})$ such that $$\bigcup \Rr_{j} \cap \Delta_j \subseteq \bigcup \Pp^j_T \cap \Delta_j.$$
		Moreover, by \cref{obs:special_vertices_hidden}, we have that for every vertex $u \in F,$ there is an index $j_{u} \in [b],$ such that $\Ball^{d}_{G}(u) \subseteq \Delta_{j_{u}}.$
		Now, if $j_{u} = j,$ since the paths of $\Rr$ agree with the paths of $\Pp^j_T$ within $\Delta_{j}$ and $\Pp^j_T$ avoids $\Ball^{d'}_{G}(u),$ then so does $\Rr_{j}.$
		In case $j_{u} \neq j,$ the same follows easily, since by definition, $\Rr_{j}$ is disjoint from $G \cap \Theta^\outsf_{j_{u}}$ (and therefore from $G \cap \Delta_{j_{u}}$ as well).

		As a final step, since $\mathsf{t}(d, k) \geq 2d + 1$ and $\mathsf{r}(d, k) \geq (k + 2)(d + 1),$ our setting satisfies case \ref{case:RP} of \cref{lem_s_or_t_represented_and_vortex}, and therefore we may conclude with a $d$-scattered $S$-$T$-linkage of order $k$ in $G$ that avoids $\Ball^{d'}_{G}(F),$ which of course contradicts our assumptions.
		Hence, we conclude that for every $j \in [b],$ the graph $H'_{j}$ contains no $T$-$\per(\FDisk^{\outsf}_{j})$-path, which in addition to the fact that $T \subseteq \bigcup_{j \in [b]} \inte(\FDisk^{\insf}_{j})$ satisfies the second part of our claim.

		Moreover, by condition \ref{c:noTwo} above we also get that every $S$-$T$-path in $H_{j} \setminus \big((\Ball^{d}_{H_{j}}(I^{1}_{j}) \cup \Ball^{d'}_{H_{j}}(I^{2}_{j})) \cup \Ball^{d}_{G}(F)\big)$ is $\Bb_{j}$-dormant and we may conclude the proof of the claim with the set $I_{j} \coloneqq I^{1}_{j} \cup I^{2}_{j}.$

		\medskip\noindent
		\textbf{$W$ $k$-represents $\rend_{T} \setminus \rend_{S}$:} This case is symmetric to the first one and we argue as before by swapping the roles of $S$ and $T.$

		\medskip\noindent
		\textbf{$W$ $k$-represents neither $\rend_{S}\setminus \rend_{T}$ nor $\rend_{T}\setminus \rend_S$:} In this case \cref{claim:local_1} guarantees that $S \cup T \subseteq \bigcup_{j \in [b]} \inte(\FDisk^{\insf}_{j}).$
		Now, for every $j \in [b]$ such that $\Qq_{j} \coloneqq \Pp^{j}_{S} \cup \Pp^{j}_{T} \neq \emptyset,$ condition \ref{c:EP} above implies that $\Qq_{j}$ is a $(S \cup T)$-rooted $(\FDisk^{\insf}, \FDisk^{\outsf})$-radial linkage of order $k$ or $2k$ that is $d$-scattered in $H_{j}.$
		As before, we may apply \cref{combing_lemma} with $(\Cc_{j}, \Pp_{j}),$ $S \cup T,$ $\Qq_{j}.$
		This application yields as a result a $(d, (\Cc_{j}, \Pp_{j}), S \cup T)$-partial linkage $\Rr_{j}$ of order $|\Qq_{j}| \in \{ k, 2k \}$ relative to $(\FDisk^{\insf}, \FDisk^{\outsf}).$
		Moreover, as in the case above, we get that the paths of $\Rr$ agree with the paths of $\Qq_{j}$ within $\Delta_j,$ and since $\Qq_{j}$ avoids $\Ball^{d'}_{G}(F),$ then so does $\Rr_{j}.$

		Now there are two subcases to consider.
		If there exists $j \in [b]$ such that both $\Pp^{j}_{S} \neq \emptyset$ and $\Pp^{j}_{T} \neq \emptyset,$ since the paths of $\Rr$ agree with the paths of $\Qq_{j}$ within $\Delta_j,$ we have that the endpoints of $\Qq_{j}$ in $\FDisk^{\insf}$ are precisely the endpoints of $\Rr_{j}$ in $\FDisk^{\insf}$.
		Therefore, $k$ paths in $\Rr_{j}$ have an endpoint in $S \cap \FDisk^{\insf}_{j}$ and $k$ paths in $\Rr_{j}$ have an endpoint in $T \cap \FDisk^{\insf}_{j}.$
		Then, since $\mathsf{t}(d, k) \geq 2d + 1$ and $\mathsf{r}(d, k) \geq (k + 2)(d + 1),$ our setting satisfies case \ref{case:P2} of \cref{lem_s_or_t_represented_and_vortex}, and hence we may conclude with a $d$-scattered $S$-$T$-linkage of order $k$ in $G$ that avoids $\Ball^{d'}_{G}(F),$ which once more contradicts our assumptions.
		Therefore, we may assume that for every $j \in [b],$ either $\Pp^{j}_{S} = \emptyset$ or $\Pp^{j}_{T} = \emptyset.$

		Now, assume that there exist two distinct indices $j, j' \in [b]$ such that $\Pp^{j}_{S} \neq \emptyset$ and $\Pp^{j}_{T} \neq \emptyset.$
		In this case, case \ref{case:PP} of \cref{lem_s_or_t_represented_and_vortex} applies, and we may once more conclude with a $d$-scattered $S$-$T$-linkage of order $k$ that avoids $\Ball^{d'}_{G}(F),$ and obtain a contradiction.

		This implies that there exists $Z \in \{ S, T \}$ such that for every $j \in [b],$ $\Pp^{j}_{Z} = \emptyset,$ which by definition, implies that the graph $H'_{j}$ contains no $Z$-$\per(\FDisk^{\outsf}_{j})$-path.
		Additionally recall that $Z \subseteq \bigcup_{j \in [b]} \inte(\FDisk^{\insf}_{j})$ as desired.

		In this case, by condition \ref{c:noTwo} above, we also get that every $S$-$T$-path in $H_{j} \setminus (\Ball^{d}_{H_{j}}(I^{1}_{j}) \cup \Ball^{d'}_{H_{j}}(I^{2}_{j}) \cup \Ball^{d}_{G}(F))$ is $\Bb_{j}$-dormant and we may conclude the proof of the claim with the set $I_{j} \coloneqq I^{1}_{j} \cup I^{2}_{j}.$
	\end{claimproof}

	Let $Y \coloneqq \bigcup_{j \in [b]} I_{j}$ and $H \coloneqq \bigcup_{j \in [b]} H_{j}.$

	\begin{claim}\label{claim:local_3} Every $S$-$T$-path in $G - (\Ball_{H}^{d}(Y) \cup \Ball_{G}^{d}(F))$ is $\Bb_{j}$-dormant for some $j \in [b].$
	\end{claim}
	\begin{claimproof} Consider any $S$-$T$-path $P$ in $G - (\Ball_{H}^{d}(Y) \cup \Ball^{d}_{G}(F)).$
		By \cref{claim:local_2} we have that there exists $Z \in \{ S, T \}$ such that $Z \subseteq \bigcup_{j \in [b]} \inte(\FDisk^{\insf}_{j}).$
		This clearly implies that the endpoint of $P$ in $Z$ belongs to one of $H_{j},$ $j \in [b].$
		Now one of the following two assertions must be true.
		Either $P$ is an $S$-$T$-path in $H_{j}$, or $P \cap (G - V(H_j)) \neq \emptyset,$ which implies that there exists a subpath $P'$ of $P$ that is a $Z$-$\per(\FDisk^{\outsf}_{j})$-path in $H_{j}.$
		However, observe that by \cref{claim:local_2}, the latter cannot be the case as then $P'$ must either intersect $\Ball^{d}_{G}(F)$ or $\Ball^{d}_{H_{j}}(I_{j}) \subseteq \Ball^{d}_{H}(Y).$
		Therefore, we conclude that $P$ is $\Bb_{j}$-dormant as desired.
	\end{claimproof}

	We are now in position to define the desired set of separations.
	For every $j \in [b]$ and every~$i \in [n_{j}],$ let \[Y^{j}_{i} \coloneqq (B^{j}_{i - 1} \cap B^{j}_{i}) \cup (B^{j}_{i} \cap B^{j}_{i + 1}) \cup \{ x^{j}_{i} \}.\]
	By the definition of a linear decomposition, it follows that $(B^{j}_{i}, V(G) \setminus (B^{j}_{i} \setminus Y^{j}_{i}))$ defines a separation of $G$ with $B^{j}_{i} \cap (V(G) \setminus (B^{j}_{i} \setminus Y^{j}_{i})) = Y^{j}_{i}.$
	We claim that $\{ (B^{j}_{i}, V(G) \setminus (B^{j}_{i} \setminus Y^{j}_{i})) \colon j \in [b], i \in [n_{j}] \}$ is the desired set of separations.
	We proceed to verify each of the desired properties in the order they are listed in the statement of the theorem.
	\begin{enumerate}
		\item Since $\Bb_{j}$ has adhesion at most $w,$ it follows that $|Y^{j}_{i}| \leq 2w + 1 \leq \sep_{\ref{local_structure_2}}(d, k)$, as desired.
		\item Let $W'$ be the base annulus of $W$ which is controlled by $X,$ since $W$ is controlled by $X.$ By the definition of a walloid, it follows that $W'$ is disjoint from $H.$ Since the order of $W'$ is larger than $\sep_{\ref{local_structure_2}}(d, k),$ the $W'$-majority side of $(B^{j}_{i}, V(G) \setminus (B^{j}_{i} \setminus Y^{j}_{i}))$ is well-defined and by our previous observation it is the side $V(G) \setminus (B^{j}_{i} \setminus Y^{j}_{i}).$ Since $W'$ is controlled by $X,$ it follows that $V(G) \setminus (B^{j}_{i} \setminus Y^{j}_{i})$ is the $X$-majority side of $(B^{j}_{i}, V(G) \setminus (B^{j}_{i} \setminus Y^{j}_{i}))$, as desired.
		\item For any two distinct $i, j \in [b],$ we have $H_{i} \cap H_{j} = \emptyset$ because the disks $\FDisk^\outsf_i$ and $\FDisk^\outsf_j$ are disjoint. Therefore, it suffices to check that for all $j \in [b]$ and two distinct $i, i' \in [n_{j}],$ we have $(B^{j}_{i} \setminus Y^{j}_{i}) \cap (B^{j}_{i'} \setminus Y^{j}_{i'}) = \emptyset.$ This follows directly by the definition of a linear decomposition.
	\end{enumerate}

	Finally, let $B \coloneqq V(G) \setminus (V(H) \setminus \bigcup_{j \in [b], i \in [n_{j}]} Y^{j}_{i}).$
	We also have to argue that every $S$-$T$-path in $G$ that avoids $\Ball^{d}_{G}(F)$ and intersects $B$ also intersects $\Ball^{d}_{H}(Y).$
	This follows directly from \cref{claim:local_3} and the definition of $B.$
	With this our proof is complete.
\end{proof}

\section{The Global Structure Theorem}\label{sec:global}

In this section we use the Local Structure Theorem (\cref{local_structure_2}) to construct a global tree decomposition of the graph whose bags exhibit the described local structure. This part of the reasoning is standard and conceptually dates back to the classic work of Robertson and Seymour in the Graph Minors series.

For the sake of modularity, we choose to capture the output of our Global Structure Theorem in the following definition.
Suppose $G$ is a graph.
A \emph{guarded tree-decomposition} of $G$ is a triple $(T,\beta,\gamma)$ such that
\begin{itemize}
	\item $(T,\beta)$ is a rooted tree-decomposition of $G$;
	\item $\gamma$ is a mapping that assigns each node $t\in V(T)$ its \emph{guard} $\gamma(t)$ so that $\gamma(t)\subseteq \beta(t)$ and if $t$ is not the root of $T$, then $\gamma(t)\supseteq \beta(t)\cap \beta(t')$, where $t'$ is the parent of $t$ in $T$.
\end{itemize}
The \emph{adhesion} of $(T,\beta,\gamma)$ is the adhesion of $(T,\beta)$, and the \emph{vigilance} of $(T,\beta,\gamma)$ is the maximum size of a guard, that is, $\max_{t\in V(T)} |\gamma(t)|$.
Now, if $d\in \N$ and $\Ff$ is a family of connected subgraphs of $G$, then we shall say that $(T,\beta,\gamma)$ \emph{distance-$d$ guards} $\Ff$ if the following condition holds: for every $H\in \Ff$ and node $t\in V(T)$, if $H$ intersects $\beta(t)$, then $H$ also intersects $\Ball_G^d(\gamma(t))$.
We will typically use this definition when $\Ff$ is the family of all $S$-$T$-paths in $G$, for some terminal sets $S,T\subseteq V(G)$; in this case we will simply say that $(T,\beta,\gamma)$ distance-$d$ guards all $S$-$T$-paths in $G$.

With this definition in place, our Global Structure Theorem reads as follows.

\begin{theorem}[Global Structure Theorem]\label{global_structure} There exist functions $\vig_{\ref{global_structure}},$ $\adh_{\ref{global_structure}} \colon \N^{3} \to \N$ such that for integers $d,k \geq 0$ and $q \leq 1$ the following holds.
	Suppose $G$ is a planar graph with a vertex subset $F \subseteq V(G)$ of size strictly less than $q$ and suppose $S, T \subseteq V(G).$
	Then, either
	\begin{itemize}
		\item $G$ contains a $d$-scattered $S$-$T$-linkage of order $k$ that avoids $\Ball^{\lfloor \nicefrac{d}{2} \rfloor}_{G}(F)$; or
		\item $G$ admits a guarded tree-decomposition $(T, \beta, \gamma)$ of adhesion at most $\adh_{\ref{global_structure}}(d, k, q)$ and vigilance at most $\vig_{\ref{global_structure}}(d, k, q)$ that distance-$d$ guards all $S$-$T$-paths in $G$ that avoid $\Ball^{d}_{G}(F).$
	\end{itemize}
	Moreover, it holds that
	$$\adh_{\ref{global_structure}}(d, k, q), \vig_{\ref{global_structure}}(d, k, q) \in d^{\Oh(1)} \cdot 2^{\Oh(\max\{ k, q \} \log \max \{ k, q \})}.$$
\end{theorem}

To facilitate induction, we instead prove the following slightly stronger result, which clearly implies \cref{global_structure} by considering the separation $(A,B)=(\emptyset,V(G))$.
The asymptotic bound on $\adh_{\ref{global_structure}}$ and $\vig_{\ref{global_structure}}$ follows from the definition of $\link_{\ref{global_structure_induction}}$ below.

	\begin{theorem}\label{global_structure_induction} There exists a function $\link_{\ref{global_structure_induction}} \colon \N^{3} \to \N$ such that for all integers $d,k \geq 0$ and $q \geq 1$ the following holds.
		Suppose $G$ is a planar graph with a vertex subset $F$ of size strictly less than $q,$ suppose $S, T \subseteq V(G),$ and suppose $(A, B)$ is a separation of $G$ of order at most $3 \cdot \link_{\ref{global_structure_induction}}(d, k, q) + 1.$
		Then, either
		\begin{itemize}
			\item $G$ contains a $d$-scattered $S$-$T$-linkage of order $k$ that avoids $\Ball^{\lfloor \nicefrac{d}{2} \rfloor}_{G}(F)$; or
			\item $G[B]$ admits a guarded tree-decomposition $(T,\beta,\gamma)$ of adhesion at most $6 \cdot \link_{\ref{global_structure_induction}}(d, k, q)- 1$ and vigilance at most $6 \cdot \link_{\ref{global_structure_induction}}(d, k, q) - 1$ such that if $r$ is the root of $T$, then $A \cap B\subseteq \gamma(r)$, and for every node $t$ of~$T$, every $S$-$T$-path in $G$ that avoids $\Ball^{d}_{G}(F)$ and intersects $\beta(t)$ also intersects $\Ball_G^d(\gamma(t)).$
		\end{itemize}
	\end{theorem}
	\begin{proof} We first define the function $\link_{\ref{global_structure_induction}}$ as follows:
		\begin{align*}
			\link_{\ref{global_structure_induction}}(d, k, q) \coloneqq~&\max\{ \mathsf{link}_{\ref{local_structure_2}}(d, k, q), \sep_{\ref{local_structure_2}}(d, k, q), \cov_{\ref{local_structure_2}}(d, k, q) \}\\
			\in~&d^{\Oh(1)} \cdot 2^{\Oh(\max \{ k, q \} \log \max \{ k, q \})}.
		\end{align*}

		Let us assume that $G$ contains no $d$-scattered $S$-$T$-linkage of order $k$ that avoids $\Ball^{\lfloor \nicefrac{d}{2} \rfloor}_{G}(F).$
		We proceed by induction on $|B \setminus A|$ and start by discussing two principal cases.

		\medskip\noindent
		\textbf{Principal case 1:} Assume that $|B| < 3 \cdot \link_{\ref{local_structure_2}}(d, k, q) + 1.$
		Then, we can define $T$ to consist of a single root node $r$ and set $\gamma(r)\coloneqq \beta(r) \coloneqq B$.
		Therefore, from now on we may  assume that $|B| \geq 3 \cdot \link_{\ref{local_structure_2}}(d, k, q) + 1.$

		\medskip\noindent
		\textbf{Principal case 2:} Assume that $|A \cap B| < 3 \cdot \link_{\ref{local_structure_2}}(d, k, q) + 1.$
		In this case we can choose any vertex $u \in B \setminus A$ (which exists since we assume $B \geq 3 \cdot \link_{\ref{local_structure_2}}(d, k, q) + 1$) and set $A' \coloneqq A \cup \{ u \},$ thereby achieving $|B \setminus A'| < |B \setminus A|.$
		Clearly $(A', B)$ is a separation of $G$ and hence we conclude by induction.
		Therefore, from now on we may also assume that $|A \cap B| = 3 \cdot \link_{\ref{local_structure_2}}(d, k, q) + 1.$

		\medskip
		Now we are in one of two cases.
		Either $A \cap B$ is well-linked in $G,$ or otherwise, there exists a separation of $G$ witnessing that this is not the case.
		We treat these two cases separately.

		\smallskip\noindent
		\textbf{Case 1:} There exists a separation $(A', B')$ of $G$ witnessing that $A \cap B$ is not well-linked in $G.$

		By the definition of well-linkedness, we have that
		\begin{equation}\label{eq:antiwelllinked}
			|(A \cap B) \cap A'| > |A' \cap B'|\quad \textrm{and} \quad|(A \cap B) \cap B'| > |A' \cap B'|.
		\end{equation}
		We get
		\begin{align*}
			&|(A' \cap B') \cup (A \cap B \cap A')| & \text{partitions into $A' \cap B'$ and $A \cap B \cap A' \setminus B'$}\\
			=~~&|A' \cap B'| + |A \cap B \cap (A' \setminus B')| & \text{$A \cap B \cap (A' \setminus B')$ and $A \cap B \cap B'$ partition $A \cap B$} \\
			=~~&|A' \cap B'| + |A \cap B| - |A \cap B \cap B'| & \text{due to \eqref{eq:antiwelllinked}} \\
			<~~&|A \cap B|.
		\end{align*}
		Taking the above and applying it again with $A'$ and $B'$ swapped we get
		\begin{equation}\label{eq:splitsidesaresmall}
			|(A' \cap B') \cup (A \cap B \cap A')| < |A \cap B| \quad\text{and}\quad |(A' \cap B') \cup (A \cap B \cap B')| < |A \cap B|.
		\end{equation}
		Based on this fact, in the claim that follows we argue that both corner separations $(A \cup B', B \cap A')$ and $(A \cup A', B \cap B')$ have order less than $|A \cap B|.$

		\begin{claim} $|(A \cup B') \cap (B \cap A')| < |A \cap B|$ and $|(A \cup A') \cap (B \cap B')| < |A \cap B|.$
		\end{claim}
		\begin{claimproof}
			We observe that
			\[
				(A' \cap B') \cup (A \cap B \cap A') \supseteq (B \cap A' \cap B') \cup (A \cap B \cap A') = (B \cap A') \cap (A \cup B'),
			\]
			and by substituting in \eqref{eq:splitsidesaresmall} we get $|(B \cap A') \cap (A \cup B')| < |A \cap B|$.
			Similarly, we have
			\[
				(A' \cap B') \cup (A \cap B \cap B') \supseteq (B \cap A' \cap B') \cup (A \cap B \cap B') = (A' \cup A) \cap (B \cap B'),
			\]
			and by substituting in \eqref{eq:splitsidesaresmall} we get $|(A \cup A') \cap (B \cap B')| < |A \cap B|$.
		\end{claimproof}

		Now $G$ along with $(A \cup B') \cap (B \cap B')$ and $G$ along with $(A \cup A') \cap (B \cap B')$ satisfy one of the two principal cases.
		Therefore by induction, we obtain a guarded tree-decomposition $(T_{1}, \beta_{1},\gamma_1)$ for $G[B \cap A']$ and a guarded tree-decomposition $(T_{2}, \beta_{2},\gamma_2)$ for $G[B \cap B'],$ both satisfying the second outcome of the theorem.
		We may now obtain a suitable guarded tree-decomposition $(T, \beta,\gamma)$ of $G[B]$ by taking the union of $T_{1}$ and $T_{2},$ introducing a new root node $r,$ making the roots of $T_1$ and $T_2$ children of $r$, setting $\beta(t) \coloneqq \beta_{1}(t)$ and $\gamma(t)\coloneqq \gamma_1(t)$ for every $t \in V(T_{1})$ and $\beta(t) \coloneqq \beta_{2}(t)$, $\gamma(t)\coloneqq \gamma_2(t)$ for every $t \in V(T_{2}),$ and finally setting $\gamma(r)\coloneqq \beta(r) \coloneqq (A \cap B) \cup (A' \cap B' \cap B).$
		Notice that $(T,\beta,\gamma)$ is indeed a guarded tree-decomposition of~$G[B].$
		In particular, by the definition of well-linkedness we have that $|A' \cap B'| \leq |A \cap B| - 2.$
		Therefore, $|\gamma(r)|=|\beta(r)| \leq |A \cap B| + |A \cap B'| \leq 6 \cdot \link_{\ref{local_structure_2}}(d, k, q) - 1,$ as claimed.

		\smallskip\noindent
		\textbf{Case 2:} $A \cap B$ is well-linked in $G.$

		Since $|A \cap B| > \mathsf{link}_{\ref{local_structure_2}}(d, k),$ we can apply \cref{local_structure_2} with $G,$ $F,$ and $A \cap B$ and obtain as a result a non-empty set $\{ (A_{i}, B_{i}) \mid i \in [n] \}$ of separations of $G,$ for some $n \in \N_{\geq 1},$ such that
		\begin{itemize}
			\item for every $i \in [n],$ $|A_{i} \cap B_{i}| \leq \sep_{\ref{local_structure_2}}(d, k, q) \leq \link_{\ref{global_structure_induction}}(d, k, q)$;
			\item for every $i \in [n],$ $B_{i}$ is the $(A \cap B)$-majority side of $(A_{i}, B_{i})$; and
			\item for all distinct $i,j \in [n],$ we have $(A_{i} \setminus B_{i}) \cap (A_{j} \setminus B_{j}) = \emptyset.$
		\end{itemize}
		Further, assuming that $$X \coloneqq \bigcup_{i \in [n]} A_{i} \quad\text{and}\quad Y \coloneqq V(G) \setminus \bigcup_{i \in [n]} (A_{i} \setminus B_{i}),$$ there exists a set $Z \subseteq X \cap Y$ of size at most $\cov_{\ref{local_structure_2}}(d, k, q) \leq \link_{\ref{global_structure_induction}}(d, k, q)$ such that every $S$-$T$-path in $G$ that avoids $\Ball^{d}_{G}(F)$ and intersects $Y$ also intersects $\Ball^{d}_{G[X]}(Z).$
		
		Our goal here is only to decompose $G[B].$
		Therefore, we first replace each separation $(A_{i}, B_{i})$ with the corner separation of $(A, B)$ and $(A_{i}, B_{i})$ whose ``small'' side (with respect to the well-linked set $A \cap B$) is fully contained in $B.$
		For every $i \in [n],$ we define \[A'_{i} \coloneqq A_{i} \cap B\quad \textrm{and}\quad B'_{i} \coloneqq B_{i} \cup A.\]
		Clearly $(A'_{i}, B'_{i})$ is a separation of $G.$
		In the claim below we argue that the order of $(A'_{i}, B'_{i})$ stays small.

		\begin{claim}\label{claim:global_1} For every $i \in [n],$ we have $|A'_{i} \cap B'_{i}| \leq 2 \cdot \link_{\ref{global_structure_induction}}(d, k, q).$
		\end{claim}
		\begin{claimproof} Since $B$ is the $(A \cap B)$-majority side of $(A_{i}, B_{i})$, it follows that \[|A_{i} \cap (A \cap B)| \leq |A_{i} \cap B_{i}| \leq \link_{\ref{global_structure_induction}}(d, k, q).\]
			Therefore we have that $$|(A_{i} \cap (A \cap B)) \cup (A_{i} \cap B_{i})| \leq 2 \cdot \link_{\ref{global_structure_induction}}(d, k, q).$$
			Now we have the following inclusions:
			\begin{align*}
				(A_{i} \cap (A \cap B)) \cup (A_{i} \cap B_{i}) \supseteq (A_{i} \cap (A \cap B)) \cup (A_{i} \cap B_{i} \cap B) &=\\
				(A_{i} \cap B) \cap (B_{i} \cup A) &= A'_{i}\cap B'_{i}.
			\end{align*}
			Therefore, we deduce that $|A'_{i} \cap B'_{i}| \leq 2 \cdot \link_{\ref{global_structure_induction}}(d, k, q).$
		\end{claimproof}

		Let $$X' \coloneqq \bigcup_{i \in [n]} A'_{i}, \quad Y' \coloneqq B \setminus \bigcup_{i \in [n]} (A'_{i} \setminus B'_{i}), \quad\textrm{and}\quad Z' \coloneqq (Z \cap X') \cup (A \cap B).$$
		We proceed to distinguish two cases.

		\medskip\noindent
		\textbf{$X'$ is empty:} In this case, we claim that we may conclude with a trivial guarded tree-decomposition $(T, \beta,\gamma)$ consisting of a single root node $r$ with bag $\beta(r) \coloneqq B$ and guard $\gamma(r)\coloneqq A \cap B$, which is of size $3 \cdot \link_{\ref{global_structure_induction}}(d, k, q) + 1$ and therefore within the claimed bounds.

		To see that this is the case, observe that by the definition of $X',$ if $X' = \emptyset,$ then for every $i \in [n]$ we have $A_{i} \subseteq A \setminus B.$
		However, this means that $Z \subseteq A \setminus B$ as well.

		Now, if there is an $S$-$T$-path, say $P,$ that avoids $\Ball^{d}_{G}(F)$ and intersects $B,$ then since $B \subseteq Y,$ it must be that $P$ is at distance at most $d$ from some vertex $z \in Z,$ which is a vertex of $A \setminus B.$
		Since $A \cap B$ separates $z$ from any vertex of~$B,$ it must be that $\Ball^{d}_{G}(A \cap B)$ also intersects $P,$ as desired.

		\medskip\noindent
		\textbf{$X'$ is non-empty:} In this case, without loss of generality we may assume that for every $i \in [n],$ $A'_{i} \neq \emptyset,$ as otherwise we may simply discard any separations with $A'_i=\emptyset$ as there is nothing to decompose further.

		Now, \cref{claim:global_1} implies that $G$ along with each $(A'_{i}, B'_{i})$ satisfy one of the two principal cases.
		Therefore by induction, we obtain a guarded tree-decomposition $(T_{i}, \beta_{i},\gamma_i)$ of $G[A'_{i}] = G[A_{i} \cap B]$ satisfying the second outcome of the theorem.
		We may now obtain a suitable guarded tree-decomposition $(T, \beta,\gamma)$ of $G[B]$ by taking the union of all $T_{i},$ introducing a new root node $r,$ making $r$ the parent of the root of each~$T_i$, setting $\beta(t) \coloneqq \beta_{i}(t)$ and $\gamma(t)\coloneqq \gamma_i(t)$ for every $t \in V(T_{i}),$ and finally setting $\beta(r) \coloneqq (A \cap B) \cup Y'$ and~$\gamma(r)\coloneqq Z'$.

		It follows directly from the construction that $(T, \beta,\gamma)$ is indeed a guarded tree-decomposition of $G[B]$ and $\gamma(r)\supseteq A\cap B$ as required.
		In particular, \cref{claim:global_1}, the fact that $|A \cap B| = 3 \cdot \link_{\ref{global_structure_induction}}(d, k) + 1,$ and the inductive assumptions on each $(T_{i}, \beta_{i},\gamma_i),$ imply that the adhesion of $(T, \beta,\gamma)$ is at most $6 \cdot \link_{\ref{global_structure_induction}}(d, k, q) - 1,$ which satisfies the claimed bound. Also, since $|Z'| \leq |Z| + |A \cap B| \leq 4 \cdot \link_{\ref{global_structure_induction}}(d, k, q) + 1,$ the vigilance of $(T, \beta,\gamma)$ is within the claimed bound.

		What remains to show is that $\gamma(r)=Z'$ satisfies the desired condition that $\Ball^d_G(\gamma(r))$ intersects all $S$-$T$-paths that avoid $\Ball^{d}_{G}(F)$ and intersect $\beta(r)$. We verify this in the claim below.

		\begin{claim} There is no $S$-$T$-path in $G$ that avoids $\Ball^{d}_{G}(Z),$ that would avoid $\Ball^{d}_{G}(Z')$ and intersect $\beta(r).$
		\end{claim}
		\begin{claimproof} Assume towards contradiction that there exists such an $S$-$T$-path in $G,$ say $P.$
			Since $P$ intersects $\beta(r)$ and $A \cap B \subseteq \gamma(r),$ it follows that $P$ is a path in $G[B]$ that avoids $\Ball^{d}_{G}(A \cap B).$
			However, we know that $\Ball^{d}_{G}(Z)$ must intersect $P$ and since every vertex of $Z$ that is a vertex of $B$ actually belongs to $\gamma(r),$ it must be that $\Ball^{d}_{G}(u)$ intersects $P,$ for some vertex $u \in Z \setminus B.$
			But this is clearly a contradiction since $A \cap B$ separates $u$ from $P$ in $G,$ and $\Ball^{d}_{G}(A \cap B)$ does not intersect $P.$
		\end{claimproof}

		With this our proof concludes.
	\end{proof}

\section{Weak Coarse Menger's Conjecture in planar graphs}\label{sec:planar}

We now have all ingredients necessary to prove the Weak Coarse Menger's Conjecture in planar graphs.
The proof will be a standard greedy bottom-up procedure employed on the tree-decomposition provided by \cref{global_structure}. In fact, we prove a more general statement that applies to any family of connected subgraphs.

\begin{theorem}\label{thm:tree-EP}
	Suppose $G$ is a graph, $\Hh$ is a family of connected subgraphs of $G$, and $(T,\beta,\gamma)$ is a guarded tree-decomposition of $G$ of vigilance $\ell$ that distance-$d$ guards $\Hh$, for some $\ell,d\in \N$. Then for any $k\in \N$, there exists one of the following objects:
	\begin{itemize}
		\item a subfamily $\Hh'\subseteq \Hh$ of size $k$ that is $d$-scattered in $G$; or
		\item a set $X\subseteq V(G)$ of size at most $\ell(k-1)$ such that $\Ball_G^d(X)$ intersects every member of $\Hh$. 
	\end{itemize}
\end{theorem}

In our case, we will take $\Hh$ to be the family of all $S$-$T$-paths in $G$, for given terminal sets $S,T\subseteq V(G)$.
Towards the proof of \cref{thm:tree-EP}, we argue the following technical statement.

\begin{lemma}\label{lemma:packing_hitting_dec} Let $d\in \N$.
    Suppose $G$ is a graph, $\Hh$ is a family of connected subgraphs of $G$, and $(T, \beta,\gamma)$ is a guarded tree-decomposition of $G$ that distance-$d$ guards $\Hh$.
    Then, there exist the following two~objects:
    \begin{itemize}
        \item a subfamily $\Hh'\subseteq \Hh$ that is $d$-scattered in $G$, and
        \item a set $N \subseteq V(T)$ with $|N| = |\Hh'|$ such that every member of $\Hh$ intersects the set
        $$\bigcup_{t \in N} \Ball^{d}_{G}(\gamma(t)).$$
    \end{itemize}
\end{lemma}
\begin{proof} 
	Let $r$ be the root of $T$.
	For every non-root node $t$ of $T$, say with parent $t'$, by $T^{\abvsf}_{t}$ and $T^{\belsf}_{t}$ we denote the components of $T - tt'$ that contains $t'$ and $t$, respectively.
    We also make the convention that $T^{\abvsf}_{r} = \emptyset$ and $T^{\belsf}_{r} = T$.
    Moreover, for every subtree $S$ of $T$ we define $$\beta(S) \coloneqq \bigcup_{t \in V(S)} \beta(t).$$
    Our goal is to prove the following stronger claim by induction over the size of $T^{\belsf}_{t}.$

    \medskip\noindent
    \textbf{Inductive claim:} Given a vertex subset $A \subseteq V(G)$ and a node $t \in V(T),$ there exist the following objects:
    \begin{enumerate}
        \item a subfamily $\Hh_t\subseteq \Hh$ that is $d$-scattered in $G$ and such that $H\subseteq G[\beta(T^{\belsf}_{t})] - A$ for each $H\in \Hh_t$;~and
        \item a set $N_{t} \subseteq V(T^{\belsf}_{t})$ with $|N_{t}| = |\Hh_{t}|$ such that 
        \begin{enumerate} 
        	\item for every $H\in \Hh_t$, there exists $t' \in N_{t}$ satisfying $V(H) \subseteq \beta(T^{\belsf}_{t'})$, and 
        	\item every $H\in \Hh$ that is entirely contained in $G[\beta(T^{\belsf}_{t})]\setminus A$ intersects $\bigcup_{t' \in N_{t}} \Ball^{d}_{G}(\gamma(t')).$
    	\end{enumerate}
	\end{enumerate}

    \medskip\noindent
    \textbf{Base:} Assume that $t \in V(T)$ is a leaf node of $T.$
    In case that $G[\beta(t)] - A$ does not contain any subgraph of $\Hh$, there is nothing to show and we may conclude with $\Hh_{t} = N_{t} = \emptyset.$
    Otherwise, set $\Hh_{t} \coloneqq \{ H \},$ where $H$ is an arbitrarily chosen subgraph of $\Hh$ contained in $G[\beta(t)] - A,$ and set $N_{t} \coloneqq \{ t \}.$
    By the assumptions on $(T, \beta,\gamma)$ the base case follows.

    \medskip\noindent
    \textbf{Inductive step:} For the inductive step we may assume that $t \in V(T)$ is a non-leaf node of $T.$
    Suppose that $(t_{1}, \ldots, t_{n})$ is an enumeration of the children of $t$.

    We first describe a procedure that for every $i \in [n]$, constructs a subfamily $\Hh_{i}\subseteq \Hh$ and a set $N_{i} \subseteq V(T^{\belsf}_{t_{i}})$ that satisfy the inductive claim for $t_i$ and some carefully chosen $A_{i}$.
    The construction of the sets $A_{i},$ $i \in [n],$ will guarantee that the union of all the families $\Hh_{i},$ $i \in [n],$ remains $d$-scattered in $G.$

    Suppose for some index $i\in [n],$ we have already constructed $A_{j}, \Hh_{j},$ and $N_{j},$ for all $j\in [i-1]$
    The next step is executed as follows. First, we set
    \[A_i\coloneqq A\cup \bigcup_{j\in [i-1]}\bigcup_{t'\in N_j} \Ball_G^d(\gamma(t')).\]
    Next, noting that $T^{\belsf}_{t_{i}} \subsetneq T^{\belsf}_{t},$ we apply the inductive claim to $t_{i}$ and $A_{i},$ thus obtaining a subfamily $\Hh_i\subseteq \Hh$ and a set $N_i\subseteq V(T^{\belsf}_{t_{i}})$ such that
    \begin{itemize}
    	\item every member of $\Hh_i$ is entirely contained in $G[\beta(T^{\belsf}_{t'})]\setminus A_i$ for some $t'\in N_i$; and
    	\item $|N_i|=|\Hh_i|$; and
    	\item every member of $\Hh$ that is entirely contained in $G[\beta(T^{\belsf}_{t_i})]\setminus A_i$ intersects $\bigcup_{t'\in N_i} \Ball^d_G(\gamma(t'))$.
    \end{itemize}
	By applying this procedure for consecutive $i=1,2,\ldots,n$ we eventually construct sets $A_i$, $\Hh_i$, and $N_i$ for all $i\in [n]$.
    Finally, we set $$\Hh'_{t} \coloneqq \bigcup_{i \in [n]} \Hh_{i},\quad N'_{t} \coloneqq \bigcup_{i \in [n]} N_{i}, \quad\text{and}\quad A_{n+1}\coloneqq A\cup \bigcup_{t'\in N_t'} \Ball_G^d(\gamma(t')).$$
    Let us verify the properties of this construction.

    \begin{claim}\label{claim:aux_1} For any two indices $i,j \in [n],$ with $i<j$, if $H_{i} \in \Hh_{i}$ and $H_{j} \in \Hh_{j},$ then every path in $G$ connecting a vertex of $H_{i}$ and a vertex of $H_{j}$ has length more than $d.$
        Moreover, $$|\Hh'_{t}| = \sum_{i \in [n]} |\Hh_{i}| \quad\text{and}\quad |N'_{t}| = |\Hh'_{t}|.$$
    \end{claim}
    \begin{claimproof} To see this, let $t' \in N_{t_{i}}$ be the node of $T_{t_{i}}$ such that $H_{i} \subseteq G[\beta(T^{\belsf}_{t'})]$, which is guaranteed to exist by induction.
    Now, by the definition of a tree-decomposition, we have that $(\beta(T^{\belsf}_{t'}), \beta(T^{\abvsf}_{t'}))$ is a separation of $G$ with $\beta(T^{\belsf}_{t'}) \cap \beta(T^{\abvsf}_{t'}) = \beta(t') \cap \beta(t''),$ where $t''$ is the parent of $t'$.
    Moreover, by the construction we have $H_{j} \subseteq G[\beta(T^{\abvsf}_{t'})] - \Ball^{d}_{G}(\beta(t') \cap \beta(t'')).$
    This immediately implies that $H_{i}$ and $H_{j}$ are at distance more than $d$ in $G.$
	
    Since all the members of $\Hh$ collected in distinct families $\Hh_i$ are pairwise at distance more than $d$ from each other, in particular the families $\Hh_i$ are pairwise disjoint. Therefore, $|\Hh'_{t}| = \sum_{i \in [n]} |\Hh_{i}|.$
    As sets $N_{t_i}$ are pairwise disjoint by construction, it also follows that $|N'_{t}| = |\Hh'_{t}|.$
    \end{claimproof}

    The fact that for every subgraph $H \in \Hh'_{t}$ there exists $t' \in N'_{t}$ such that $H \subseteq G[\beta(T^{\belsf}_{t'})]$ clearly follows from the construction.
    Now there are two cases to distinguish.

    \medskip\noindent
    \textbf{Case 1: $G[\beta(T^\belsf_{t})] - A_{n + 1}$ contains no member of $\Hh$.} In this case, we may conclude with $\Hh_{t} \coloneqq \Hh'_{t}$ and $N_{t} \coloneqq N'_{t}.$
    This follows directly from \cref{claim:aux_1}.

    \medskip\noindent
    \textbf{Case 2: $G[\beta(T^\belsf_{t})] - A_{n + 1}$ contains some member of $\Hh$.} In this case, we claim that we may conclude with $\Hh_{t} \coloneqq \Hh'_{t} \cup \{ H \},$ where $H$ is an arbitrarily chosen member of $\Hh$ that is contained in $G[\beta(T^{\belsf}_{t})] - A_{n + 1},$ and $N_{t} \coloneqq N'_{t} \cup \{ t \}.$

    We first show that the graph $$G[\beta(T^{\belsf}_{t})] - (A_{n + 1} \cup \Ball^{d}_{G}(\gamma(t)))$$ does not contain any subgraph $H\in\Hh$.
    By the assumption that $(T, \beta,\gamma)$ distance-$d$ guards $\Hh$, it suffices to show that any such subgraph $H$ would need to intersect $\beta(t).$
    To see this, observe that if $H$ did not intersect $\beta(t)$, then by the definition of a tree-decomposition and since $H$ is connected, there must exist an index $i \in [n]$ such that $H \subseteq G[\beta(T^{\belsf}_{t_{i}})].$
    However, by the induction assumption, we know that $G[\beta(T^{\belsf}_{t_{i}})] - A_{i+1}$ contains no member of $\Hh$.
    This is clearly a contradiction.

    Finally, in addition to \cref{claim:aux_1}, we have to argue that for every subgraph $J \in \Hh_{i},$ for some $i \in [n],$ the distance between any vertex of $J$ and any vertex of $H$ is more than $d.$
    This follows by an argument similar to that used in the proof of \cref{claim:aux_1}.
    By construction, we know that there exists $t' \in N_{i}$ such that $J \subseteq G[\beta(T^{\belsf}_{t'})].$
    The definition of a tree-decomposition implies that $(\beta(T^{\belsf}_{t'}), \beta(T^{\abvsf}_{t'}))$ is a separation of $G$ with $\beta(T^{\belsf}_{t'}) \cap \beta(T^{\abvsf}_{t'}) = \beta(t') \cap \beta(t''),$ where $t''$ is the parent of $t'$.
    Since $H$ is connected and disjoint from $\Ball^{d}_{G}(\beta(t') \cap \beta(t''))$ (this follows from $\beta(t') \cap \beta(t'')\subseteq \gamma(t')$ and $t'\in N'_t$), and therefore in particular from $\beta(t') \cap \beta(t'')$ as well, it must be that $H \subseteq G[\beta(T^{\abvsf}_{t'})] - \Ball^{d}_{G}(\beta(t') \cap \beta(t'')),$ since as in the previous paragraph, it cannot be that $H \subseteq \beta(T^{\belsf}_{t'}).$
    This readily implies that $H$ and $J$ are at distance more than $d.$

    \medskip
    Our proof now concludes by calling upon the inductive claim with the root node $r$ and $A$ being the empty set.
\end{proof}

With the technical statement established, \cref{thm:tree-EP} follows as an easy consequence.

\begin{proof}[Proof of \cref{thm:tree-EP}]
	We apply \cref{lemma:packing_hitting_dec} and thus obtain a $d$-scattered subfamily $\Hh'\subseteq\Hh$ and a set $N\subseteq V(T)$ with $|N|=|\Hh'|$ such that every member of $\Hh$ intersects $\bigcup_{t\in N}\Ball_G^d(\gamma(t))$. If $|\Hh'|\geq k$, then any subfamily of $\Hh'$ of size $k$ constitutes a valid first outcome. Otherwise we have $|\Hh'|\leq k-1$, and therefore $X\coloneqq \bigcup_{t\in N} \gamma(t)$ has size at most $(k-1)\ell$. Hence $X$ constitutes a valid second outcome.
\end{proof}

With \cref{thm:tree-EP} in place, we may now complete the proof of the planar case of \cref{thm:main}.

\begin{theorem}\label{menger_planar}
  There exists a function $f_{\ref{menger_planar}}\colon \N^2\to \N$ such that for all $d, k\in \N$ the following holds.
Let $G$ be a planar graph and let $S,T\subseteq V(G)$ be vertex subsets.
Suppose that one cannot find $k$ $S$-$T$-paths in $G$ that are pairwise at distance more than $d$ apart.
Then there is a vertex subset $X\subseteq V(G)$ with $|X|\leq f_{\ref{menger_planar}}(k,d)$ such that every $S$-$T$-path in $G$ is at distance at most $d$ from some vertex of $X$.

Moreover, we have that $f_{\ref{menger_planar}}(d, k)\in d^{\Oh(1)}\cdot 2^{\Oh(k\log k)}$.
\end{theorem}
\begin{proof}
	It suffices to apply \cref{global_structure} with $F = \emptyset$ and feed the obtained guarded tree-decomposition to \cref{thm:tree-EP}, with $\Hh$ being the family of all $S$-$T$-paths in $G$. Note that by the assumption, none of these applications may result in exposing a $d$-scattered $S$-$T$-linkage of order $k$.
\end{proof}

\section{Weak Coarse Menger's Conjecture in surface-embeddable graphs}\label{sec:genus}

In this section we lift the approach presented in the previous sections to graphs embeddable on surfaces more complicated than the sphere. 

\subsection{Preliminaries on surface-embedded graphs}

We first need to establish the relevant terminology and basic results about embeddings of graphs in surfaces. We mostly follow the presentation from the classic book of Mohar and Thomassen~\cite{MoharT01Graphs}, to which we refer an interested reader for more details.

\paragraph{Surfaces.} In this paper, by a \emph{surface} we mean an arc-connected and compact $2$-dimensional manifold without a boundary. In a \emph{boundaried surface} we additionally allow the manifold to have a boundary consisting of a finite number of arc-connected components called \emph{cuffs}, each homeomorphic to the circle $S^1$. The \emph{closing} of a boundaried surface $\Sigma$ is the surface $\Sigma'$ obtained by gluing an open disk to every cuff of $\Sigma$.

Given a pair of integers $(\sfh, \sfc) \in \N \times \N,$ we define $\Sigma^{(\sfh, \sfc)}$ to be the surface obtained from the sphere by adding $\sfh$ handles and $\sfc$ crosscaps.
If $\sfc = 0$ the surface $\Sigma^{(\sfh, \sfc)}$ is \emph{orientable}, and it is \emph{non-orientable} otherwise. It is well-known that every surface without a boundary is homeomorphic to $\Sigma^{(\sfh,\sfc)}$, for some $(\sfh, \sfc) \in \N \times \N$. 
We define the \emph{genus} of a surface $\Sigma$ as $2\sfh + \sfc,$ where $\Sigma^{(\sfh, \sfc)}$ is the surface to which $\Sigma$ is homeomorphic. The genus of a boundaried surface is defined as the genus of its closing.

\paragraph{Drawing a graph on a surface.}
A \emph{drawing} on a surface $\Sigma$ is a triple $\Gamma = (U, V, E)$ such that
\begin{itemize}
    \item $V$ and $E$ are finite;
    \item $V \subseteq U \subseteq \Sigma$ and $e\subseteq U\subseteq \Sigma$ for each $e\in E$;
    \item $V \cup \bigcup E = U$ and $V \cap \bigcup E = \emptyset$;
    \item for every $e \in E,$ we have that $e = h([0, 1]) \setminus \{ h(0), h(1) \},$ where $h \colon [0, 1] \to U$ is a homeomorphism from the unit interval $[0,1]\subseteq \mathbb{R}$ onto its image satisfying $h(0), h(1) \in V$; and
    \item if $e, e' \in E$ are distinct, then $|e \cap e'|$ is finite.
\end{itemize}
We call the set $V,$ which we often denote by $V(\Gamma),$ the \emph{points} of $\Gamma$, and the set $E,$ which we often denote by $E(\Gamma),$ the \emph{arcs} of $\Gamma.$
If $G$ is a graph such that $V(\Gamma)$ and $E(\Gamma)$ naturally correspond to $V(G)$ and $E(G)$ respectively, we say that $\Gamma$ is a \emph{drawing} of $G$ on $\Sigma$.
In case no two arcs of $E(\Gamma)$ have a common point, we say that $\Gamma$ is an \emph{embedding} of $G$ on $\Sigma$.
In this case, the connected components of $\Sigma \setminus U$ are the \emph{faces} of $\Gamma.$

\paragraph{Representativity.}

A (simple) curve on a surface $\Sigma$ is a homeomorphic image of the circle $S^1$ in $\Sigma$. A curve is called \emph{contractible} if it is continuously deformable to a single point within $\Sigma$, and \emph{non-contractible} otherwise. Note that, if $\gamma$ is a non-contractible curve on $\Sigma,$ then $\Sigma \setminus \gamma$ consists of one or two connected components, each being a surface with a boundary whose genus is strictly smaller than the genus of $\Sigma.$

Given an embedding $\Gamma=(U,V,E)$ of a graph $G$ on a surface $\Sigma,$ a \emph{noose} is a curve on $\Sigma$ that meets $U$ only at the points of $V$. The \emph{representativity} of $\Gamma$ is the minimum number of points of $V$ that lie on any non-contractible noose.
Note that when $\Sigma$ is a sphere, then there are no non-contractible nooses and the representativity of every embedding is infinite.

\paragraph{Highly representative embeddings.}

We are going to need the following two results regardings surface embeddings of graphs with large representativity.

\medskip
The first result, due to Seymour and Thomas~\cite{SeymourT1996Uniqueness} (see also~\cite{Thomassen1990Embeddings,Mohar95Uniqueness,RobertsonV1990Embeddings}), says that highly representative embeddings of $3$-connected graphs are unique.

\begin{proposition}[\cite{SeymourT1996Uniqueness}]\label{prop:representativity_implies_uniqueness}
    For all non-negative integers $g$ and $r$ the following holds.
    Suppose $G$ is a $3$-connected graph and $\Gamma$ is an embedding of $G$ on a surface $\Sigma$ of genus $g$ with representativity $r.$
    If $$r \geq \frac{100 \log g}{\log \log g} \quad \text{(or $r \geq 100$ if $g \leq 2$)},$$ then $G$ cannot be embedded on a surface of genus smaller than $g$ and every embedding of $G$ on $\Sigma$ is homeomorphic~to~$\Gamma.$
\end{proposition}

The second result from the work of Gorsky, Seweryn, and Wiederrecht~\cite{GorskySW2025Polynomial} is about finding minors in graphs with surface embeddings of large representativity.

\begin{proposition}[{\cite[Theorem 15.11]{GorskySW2025Polynomial}}]\label{prop:representativity_gives_minor}
	There is a polynomial function $\repr_{\ref{prop:representativity_gives_minor}} \colon \N\times \N\to \N$ such that for all $g,t\in \N$ the following holds.
    Every graph that has an embedding with representativity at least $\repr_{\ref{prop:representativity_gives_minor}}(g, t)$ on a surface $\Sigma$ of genus at most $g$ contains, as a minor, every graph on at most $t$ vertices that can be embedded~on~$\Sigma.$
\end{proposition}

\subsection{The Representativity Lemma}

We also need a tool that will allow us to reduce our proof to embeddings of large representativity, at the cost of introducing some apex vertices.

\begin{lemma}[Representativity Lemma]\label{lemma:representativity} Let $g,r\in \N$ with $r\geq 1$ and
    suppose $G$ is a connected graph embedded on a surface of genus $g.$
    Then there exists a set of vertices $A \subseteq V(G)$ with $|A| \leq (2g - 1)(r - 1)$ (or $A=\emptyset$, if $g=0$) and a partition $\Pp$ of $V(G) \setminus A$ into non-empty vertex subsets (or $\Pp = \emptyset$ if $V(G) = A$) such that
    \begin{itemize}
        \item $\bigcup_{H \in \Pp} G[H \cup A] = G$; and
        \item for every $H \in \Pp,$ there exists a graph $H^\star$ and a vertex subset $A^\star \subseteq V(H^\star)$ with $|A^\star| \leq 2g$ satisfying the following two conditions:
        \begin{enumerate}
            \item $H^\star - A^\star = G[H]$ and $N_{H^\star}(A^\star) = N_{G}(A) \cap H$; and
            \item $H^\star$ has an embedding on a surface of genus at most $g$ with representativity at least $r.$
        \end{enumerate}
    \end{itemize}
\end{lemma}
\begin{proof} We prove the following stronger claim by induction on the genus $g$ of the surface on which $G$ is embedded.

    \smallskip
    \noindent\textbf{Inductive claim:} Let $G$ be a graph, $A \subseteq V(G)$ be a vertex subset, and let $H \coloneqq G - A.$
    Suppose there exists a connected graph $H^\star$ and a vertex subset $A^\star \subseteq V(H^\star)$ satisfying the following two conditions:
    \begin{itemize}
        \item $H^\star - A^\star = H$ and $N_{H^\star}(A^\star) = N_{G}(A)$; and
        \item $H^\star$ has an embedding $\Gamma_{H^\star}$ on a surface of genus $g.$
    \end{itemize}
    Then, there exists $X \subseteq V(H)$ with $|X| \leq (2g - 1)(r - 1)$ (or $X = \emptyset,$ if $g = 0$) and a partition $\Pp$ of $V(H) \setminus X$ into non-empty vertex subsets (or $\Pp = \emptyset$ if $V(H) = X$) such that
    \begin{itemize}
        \item $\bigcup_{K \in \Pp} H[K \cup X] = H$; and
        \item for every $K \in \Pp,$ there exists a connected graph $K^\star$ and a vertex subset $X^\star \subseteq V(K^\star)$ with $|X^\star| \leq |A^\star| + 2g$ satisfying the following two conditions:
    \begin{enumerate}
        \item $K^\star \setminus X^\star = H[K]$ and $N_{K^\star}(X^\star) = N_{G}(A \cup X) \cap K$; and
        \item $K^\star$ has an embedding on a surface of genus at most $g$ with representativity at least $r.$
    \end{enumerate}
    \end{itemize}

	\smallskip
	Observe that the lemma statement follows by applying the inductive claim to $G$ with $A = \emptyset$ and $H^\star = G.$
    Therefore, from now on we focus on proving the inductive claim.

    \smallskip
    \noindent\textbf{Base case, $g = 0$:} By assumption, if $\Sigma$ is a sphere, then any embedding on $\Sigma$ has infinite representativity and therefore we may conclude with $X = \emptyset$ and $\Pp = \{ V(H) \}$ (or $\Pp = \emptyset$ if $V(H) = \emptyset$).

    \smallskip
    \noindent\textbf{Inductive step, $g \geq 1$:} Assume the claim holds for all genera $g' < g$ and suppose that $\Sigma$ is a surface of genus $g.$
    In case the embedding $\Gamma_{H^\star}$ of $H^\star$ on $\Sigma$ already has representativity at least $r$, we may conclude with $X = \emptyset$ and $\Pp = \{ V(H) \}.$
    Therefore, we may assume otherwise, which implies that there exists a non-contractible noose $\gamma$ on $\Sigma,$ that intersects the embedding of $H^\star$ only in vertices, and in strictly less than $r$ many.
    Let $X_{0}$ be this set of vertices.

    Now, let $\widehat{\Sigma}$ be the (possibly disconnected) surface with boundary, obtained from $\Sigma$ by cutting it along~$\gamma$; i.e., $\widehat{\Sigma}$ is the compactification of $\Sigma \setminus \gamma.$
    Next, let $\widetilde{\Sigma}$ be the closing of $\Sigma$, obtained from $\widehat{\Sigma}$ by gluing a disk $\Delta_B$ to every cuff $B$ of $\widehat{\Sigma}$.
    Note that thus $\widetilde{\Sigma}$ is the union of either one or two connected surface without a boundary, depending on whether $\gamma$ was separating or not.
    In case $\widetilde{\Sigma}$ is connected, it has genus at most $g - 1$ and $\widehat{\Sigma}$ has exactly one cuff when $\gamma$ is one-sided, or exactly two cuffs when $\gamma$ is two-sided.
    In case $\widetilde{\Sigma}$ is not connected, $\gamma$ is always two-sided, and then $\widehat{\Sigma}$ consists of exactly two connected components: $\widetilde{\Sigma}_{1}$ of genus $g_{1} \geq 1$ and $\widetilde{\Sigma}_{2}$ of genus $g_{2} \geq 1,$ each with exactly one cuff, and such that $g_{1} + g_{2} = g.$

    Let $v_{1}, \ldots, v_{r'}$ ($r' < r$) be the cyclic ordering of the vertices in $X_{0},$ as traversed by $\gamma$ in an arbitrarily chosen direction.
    Note that for every vertex $v_{i},$ $i \in [r'],$ there exists a partition $E^{-}_{i}, E^{+}_{i}$ of the edges incident to $v_{i}$ so that if we define $H'$ to be the graph where each vertex $v_{i}$ is replaced by two vertices $v^{-}_{i}$ and $v^{+}_{i}$, with $v^-_i$ becoming an endpoint of the edges of $E^-_i$ and $v^+_i$ becoming an endpoint of the edges of $E^+_i$, then $H'$ has an embedding $\Gamma_{H'}$ on $\widetilde{\Sigma}$ such that the following holds.
    We distinguish two cases.

    \smallskip
    \noindent\textbf{Case 1, $\gamma$ is one-sided:} Then, the vertices $v^{-}_{1}, \ldots, v^{-}_{r'}, v^{+}_{1}, \ldots, v^{+}_{r'}$ are embedded along the boundary of $\Delta_{B}$ in this cyclic ordering, where $B$ is the unique cuff of $\widehat{\Sigma}.$
    In this case, we define the graph $H''$ by replacing the vertices $v^{-}_{1}, \ldots, v^{-}_{r'}, v^{+}_{1}, \ldots, v^{+}_{r'}$ by a single vertex $v$ with $N_{H''}(v) = N_{H'}(\{ v^{-}_{1}, \ldots, v^{-}_{r'}, v^{+}_{1}, \ldots, v^{+}_{r'} \}).$
    Clearly, $\Gamma_{H'}$ yields an embedding $\Gamma_{H''}$ of $H'$ on $\widetilde{\Sigma},$ where $v$ is drawn in the interior of $\Delta_{B}.$
    We let~$X^\star_{0} \coloneqq \{ v \}.$

    \smallskip
    \noindent\textbf{Case 2, $\gamma$ is two-sided:} Then, the vertices $v^{-}_{1}, \ldots, v^{-}_{r'}$ are embedded along the boundary of $\Delta_{B^{-}}$ in this cyclic ordering, while the vertices $v^{+}_{1}, \ldots, v^{+}_{r'}$ are embedded along the boundary of $\Delta_{B^{+}}$ in this cyclic ordering, where $B^{-}$ and $B^{+}$ are the two cuffs of $\widehat{\Sigma}.$
    In this case, we define the graph $H''$ by replacing the vertices $v^{-}_{1}, \ldots, v^{-}_{r'}, v^{+}_{1}, \ldots, v^{+}_{r'}$ by a pair of vertices $v^{-}$ and $v^{+}$ with $N_{H''}(v^{-}) = N_{H'}(\{ v^{-}_{1}, \ldots, v^{-}_{r'} \})$ and $N_{H''}(v^{+}) = N_{H'}(\{ v^{+}_{1}, \ldots, v^{+}_{r'} \}).$
    Clearly, $\Gamma_{H'}$ induces an embedding $\Gamma_{H''}$ of $H'$ on $\widetilde{\Sigma},$ where $v^{-}$ is drawn in the interior of $\Delta_{B^{-}}$ and $v^{+}$ is drawn in the interior of $\Delta_{B^{+}}$ respectively.
    We let $X^\star_{0} \coloneqq \{ v^{-}, v^{+} \}.$

    \smallskip
    Define $X'_{0} \coloneqq X_{0} \setminus A^\star$ and $A' \coloneqq A^\star \setminus X_{0}.$
    By construction we have $H'' - (A' \cup X^\star_{0}) = G - (A \cup X'_{0})$ and $N_{H''}(A' \cup X^\star_{0}) = N_{G}(A \cup X'_{0}).$
    Now, note that if $V(H'') = A' \cup X^\star_{0}$ then we may conclude with $X = X'_{0}$ and $\Pp = \emptyset.$
    Therefore, we may assume that $V(H'') \supsetneq A' \cup X^\star_{0}.$

    Let $H^\star_{1}, \ldots, H^\star_{q}$ be the connected components of $H'',$ which are not subgraphs of $H''[A' \cup X^\star_{0}],$ and for each $j \in [q],$ define $A^\star_{j} \coloneqq (A' \cup X^\star_{0}) \cap V(H^\star_{j}).$
    Now, define $\Pp'$ to be the partition of $V(H) \setminus X'_{0}$ into the following non-empty subsets: $$\Pp' \coloneqq \{ V(H^\star_j) \setminus A^\star_{j} \mid i \in [q] \}.$$
    Next, for each $j \in [q],$ define $H_{j} \coloneqq H^\star_{j} - A^\star_{j}$ and observe that, since $H'' - (A' \cup X^\star_{0}) = G - (A \cup X'_{0})$, we have that $\bigcup_{j \in [q]} H[V(H_{j}) \cup X'_{0}] = H.$

    Our goal is to apply the inductive claim on $G_{j} \coloneqq G[V(H_{j}) \cup (A \cup X'_{0})]$ together with the set $A \cup X'_{0}$, and $H^\star_{j}$ together with the set $A^\star_{j}.$
    To this end, we have to verify whether they satisfy its assumptions.
    First note that by construction, $N_{H^\star_{j}}(A^\star_{j}) = N_{G_{j}}(A \cup X'_{0})$, as desired.
    Moreover, for every $j \in [q],$ $\Gamma_{H''}$ induces an embedding $\Gamma_{H^\star_{j}}$ of $H^\star_{j}$ on a connected component of $\widetilde{\Sigma}$ which has genus $\hat{g}_{j} < g.$
    In fact, by the additivity of genus, we have that either $\sum_{j \in [q]} \hat{g}_{j} < g,$ in case $\gamma$ was non-separating, or that there is a partition of $[q]$ into $I_{1}$ and $I_{2}$ according to which of the two connected components of $\widetilde{\Sigma}$ each $H^\star_{j}$ is embedded on, such that $\sum_{j \in I_{1}} \hat{g}_{j} \leq g_{1}$ and $\sum_{j \in I_{2}} \hat{g}_{j} \leq g_{2}.$
    This reflects the case where $\gamma$ was separating.

    Therefore, by the induction hypothesis, for each $j \in [q]$ there is a set $X_{j} \subseteq V(H_{j})$ with $|X_{j}| \leq (2\hat{g}_{j} - 1)(r - 1)$ (or $X_{j} = \emptyset,$ if $\hat{g}_{j} = 0$) and a partition $\Pp_{j}$ of $V(H_{j}) - X_{j}$ into non-empty vertex subsets (or $\Pp_{j} = \emptyset$ if $V(H_{j}) = X_{j}$) such that
    \begin{itemize}
        \item $\bigcup_{K_{j} \in \Pp_{j}} H_{j}[K_{j} \cup X_{j}] = H_{j}$; and
        \item for every $K_{j} \in \Pp_{j},$ there exists a connected graph $K^\star_{j}$ and a vertex subset $X^\star_{j} \subseteq V(K^\star_{j})$ with $|X^\star_{j}| \leq |A^\star_{j}| + 2\hat{g}_{j}$ satisfying the following two conditions:
    \begin{itemize}
        \item $K^\star_{j} \setminus X^\star_{j} = H_{j}[K_j]$ and $N_{K^\star_{j}}(X^\star_{j}) = N_{G}(A \cup X'_{0} \cup X_{j}) \cap K_{j}$; and
        \item $K^\star_{j}$ has an embedding on a surface of genus at most $\hat{g}_{j}$ with representativity at least $r.$
    \end{itemize}
    \end{itemize}
    Set $X \coloneqq X'_{0} \cup \bigcup_{j \in [q]} X_{j}.$
    Now observe that, since $\bigcup_{j \in [q]} H[V(H_{j}) \cup X'_{0}] = H,$ if $\Pp_{j} = \emptyset$ for all $j \in [q],$ then we have that $V(H) = X$ and we may conclude with $X$ and $\Pp = \emptyset,$ subject to verifying the bound on $|X|.$
    So, we may assume that for some $j \in [p],$ $\Pp_{j} \neq \emptyset.$
    In this case, we define $\Pp$ as the refinement of $\Pp',$ obtained by taking the union of all $\Pp_{j}$ which are non-empty, and since $\bigcup_{K_{j} \in \Pp_{j}} H_{j}[K_{j} \cup X_{j}] = H_{j},$ for all $j \in [q],$ it follows that $\bigcup_{K \in \Pp} H[K \cup X] = H$ as desired.
    Therefore, subject to verifying the bound on each $|X^\star_j|$ as well, our inductive claim follows.
    
    Observe that
    \[|X|\leq |X'_0|+\sum_{j \in [q]} |X_j|\leq  (r-1)+\sum_{j \in [q], \hat{g}_j\geq 1} (2 \hat{g}_{j} - 1)(r - 1).\]
    The following claim will be useful in bounding the last summand.

    \begin{claim}\label{claim:rep_1} If $\sum_{j \in [q]} \hat{g}_{j} \leq h$, then $$\sum_{j \in [q], \hat{g}_j\geq 1} (2 \hat{g}_{j} - 1)(r - 1) \leq (2h -1)(r - 1).$$
    \end{claim}
    \begin{claimproof} $$\sum_{j \in [q], \hat{g}_{j} \geq 1} (2 \hat{g}_{j} - 1)(r - 1) \leq \bigg(2 \sum_{j \in [q], \hat{g}_{j} \geq 1} \hat{g}_{j} - 1 \bigg)(r - 1) \leq (2h - 1)(r - 1).\qedhere$$
    \end{claimproof}

    We now distinguish two cases.
    
    \smallskip
    \noindent\textbf{$\gamma$ is non-separating:} In this case, $\widetilde{\Sigma}$ is connected and from our previous observation we have that $\sum_{j \in [q]} \hat{g}_{j} \leq g - 1.$
    By \cref{claim:rep_1}, $$|X| \leq (r - 1) + (2 (g - 1) - 1)(r - 1) = (r - 1) + (2g - 3)(r - 1) = (2g - 2)(r - 1) < (2g - 1)(r - 1).$$

    \smallskip
    \noindent\textbf{$\gamma$ is separating:} In this case, $\widetilde{\Sigma}$ consists of two connected components of genus $g_{1}, g_{2} \geq 1$ respectively, such that $g_{1} + g_{2} = g.$
    Therefore, from our previous observation we have that $\sum_{j \in I_{1}} \hat{g}_{j} \leq g_{1}$ and $\sum_{j \in I_{2}} \hat{g}_{j} \leq g_{2}.$
    By applying \cref{claim:rep_1} to each side, we obtain $$|X| \leq (r - 1) + (2 g_{1} - 1)(r - 1) + (2 g_{2} - 1)(r - 1) + (r - 1)(1 + (2 g_{1} - 1) + (2 g_{2} - 1)) = (2 g - 1)(r - 1).$$

    Finally, note that for every $j \in [q],$ we have $$|X^\star_{j}| \leq |A^\star_{j}| + 2 \hat{g}_{j} \leq |A^\star| + 2 + 2(g - 1) = |A^\star| + 2g.$$
    This completes the proof.
\end{proof}

\subsection{Star replacements}

Let $G$ be a graph, $A \subseteq V(G)$ be a vertex subset, $H$ be a connected component of $G - A,$ and $H' \coloneqq G[V(H) \cup A].$
Let $H^\star$ be a graph with a vertex subset $A^\star \subseteq V(H^\star)$ satisfying the following two conditions:
\begin{itemize}
    \item $H^\star - A^\star = H$; and
    \item $N_{H^\star}(A^\star) = N_{H'}(A).$
\end{itemize}
We call the pair $(H^\star,A^\star)$ as above a \emph{replacement} of $(H',A)$. We shall refer to $A^\star$ as to the \emph{star vertices} of the replacement.

\medskip
For the remainder of this subsection, fix in the context a non-negative integer $d$, a graph $G$, and a vertex subset $A \subseteq V(G).$
Moreover, we fix a connected component $H$ of $G - A$ and let $H' \coloneqq G[V(H) \cup A].$ Also, we fix a replacement $(H^\star,A^\star)$ of $(H',A)$. Let us make a few observations about distances.

\begin{observation}\label{obs:packA} For any two vertices $u, v \in V(H)$ such that $\dist_{H^\star}(u, A^\star) > \lfloor \nicefrac{d}{2} \rfloor,$ $\dist_{H^\star}(v, A^\star) > \lfloor \nicefrac{d}{2} \rfloor,$ and $\dist_{H^\star}(u, v) > d,$ we have $$\dist_{G}(u, v) > d.$$
\end{observation}
\begin{proof} Let $P$ be a shortest $u$-$v$-path in $G.$ We prove that the length of $P$ is larger than $d$.
    There are two~cases.

    \smallskip
    \noindent\textbf{Case 1, $P$ avoids $A$:} In this case $P$ is a path in $H$ and therefore also a path in $H^\star.$
    Since $\dist_{H^\star}(u, v) > d,$ the length of $P$ must be larger than $d$.

    \smallskip
    \noindent\textbf{Case 2, $P$ intersects $A$:} Let $P_u$ be the longest prefix of $P$ that starts at $u$ and avoids $A$. Note that $P_u$ is a path in $H$ and its other endpoint belongs to $N_{H'}(A)$. Since $N_{H^\star}(A^\star)=N_{H'}(A)$, $P_u$ can be extended by one edge to a $u$-$A^\star$-path in $H^\star$. As $\dist_{H^\star}(u, A^\star) > \lfloor \nicefrac{d}{2} \rfloor$, it follows that the length of $P_u$ is at least $\lfloor \nicefrac{d}{2} \rfloor$. Similarly, if $P_v$ is the longest suffix of $P$ that ends at $v$ and avoids $A$, then the length of $P_v$ is at least $\lfloor \nicefrac{d}{2} \rfloor$. Note that $P_u$ and $P_v$ are disjoint subpaths of $P$ and there are at least two edges of $P$ --- the ones incident to the vertices of $A$ --- that belong to neither $P_u$ nor $P_v$. It follows that the length of $P$ is at least $\lfloor \nicefrac{d}{2} \rfloor + \lfloor \nicefrac{d}{2} \rfloor + 2 > d$.
\end{proof}

\begin{observation}\label{obs:hitA} For every vertex subset $X \subseteq V(H)$ and every vertex $v \in V(H)$, we have $$\dist_{H^\star}(v, X \cup A^\star) = \dist_{G}(v, X \cup A).$$
\end{observation}
\begin{proof} We first prove that $\dist_{H^\star}(v, X \cup A^\star) \leq \dist_{G}(v, X \cup A).$ Let $P$ be a shortest $v$-$(X\cup A)$-path in~$G$, say of length $\ell\coloneqq \dist_{G}(v, X \cup A)$. If $P$ avoids $A$, then $P$ is also a $u$-$X$-path in $H^\star$ that witnesses that $\dist_{H^\star}(v, X \cup A^\star)\leq \ell$. Otherwise, by minimality, the only vertex of $P$ that belongs to $A$ is the endpoint other then $v$. Since $N_{H^\star}(A^\star)=N_{H'}(A)$, we may replace this endpoint with a vertex of $A^\star$, thus obtaining a $u$-$A^\star$-path of length $\ell$ in $H^\star$. This path witnesses that $\dist_{H^\star}(v, X \cup A^\star)\leq \ell$.

We now prove that $\dist_{G}(v, X \cup A) \leq \dist_{H^\star}(v, X \cup A^\star).$ Let $P$ be a shortest $v$-$(X\cup A^\star)$-path in~$H^\star$, say of length $\ell\coloneqq \dist_{H^\star}(v, X \cup A^\star)$. Again, if $P$ avoids $A^\star$, then $P$ is also a $u$-$X$-path in $G$ that witnesses that $\dist_{G}(v, X \cup A)\leq \ell$. And otherwise, by minimality, the only vertex of $P$ that belongs to $A^\star$ is the endpoint other then $v$. Now, the same replacement argument as in the first case yields a $u$-$A$-path in $G$ that witnesses that $\dist_{G}(v, X \cup A)\leq \ell$.
\end{proof}

An immediate corollary of \cref{obs:hitA} is the following.

\begin{corollary}\label{cor:HitA} For every vertex subset $X \subseteq V(H)$ we have
    $$\Ball^{d}_{H^\star}(X \cup A^\star) \setminus A^\star = \Ball^{d}_{G}(X \cup A) \cap V(H).$$
\end{corollary}

\subsection{Classifying terminals on surfaces}

We now work towards establishing the analogue of \cref{structure_relative_terminals} for graphs embeddable in surfaces richer than the sphere. Generally speaking, we need to (i) introduce a possibility of having a bounded number of apex vertices that can be adjacent only to the vortices, and (ii) add features to the walloids that will represent the surface into which the graph is embedded.

\paragraph{Handle segments.}

Let $r, t \in \N_{\geq 4}$ and let $p = \min\{r, \lceil \nicefrac{t}{2} \rceil\} - 3.$
An \emph{elementary $(r \times t)$-handle segment} $W$ is obtained from an elementary $(r \times 4t)$-wall segment $W_{1}$ (called the \emph{base}) with top boundary vertices $$v_{1}, \ldots, v_{t}, v'_{1}, \ldots, v'_{t}, u_{1}, \ldots, u_{t}, u'_{1}, \ldots, u'_{t}$$ in left to right order by adding all the edges $$\{ v_{i} u_{t - i + 1}, v'_{i} u'_{t - i + 1} \colon i \in [t] \},$$ see~\cref{fig:handle_segment}.
An \emph{$(r \times t)$-handle segment} is a subdivision of an elementary $(r \times t)$-handle segment.

\begin{figure}[ht]
    \centering
    \includegraphics[page=2, scale=1.1]{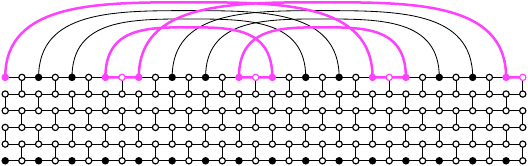}
    \caption{
        A $(6 \times 4)$-handle ssegment with its (first) exceptional path highlighted.
    }%
    \label{fig:handle_segment}
\end{figure}

Let $x,$ respectively $y,$ be the first left, respectively first right, boundary vertex of $W_{1}.$
We call the $x$-$y$-path $P$ in $W$ on the following set of edges $$\{ xv_{1}, v_{1}u_{t}, u_{t}u'_{1}, u'_{1}v'_{t}, v'_{t}u_{1}, u_{1}v_{t}, v_{t}v'_{1}, v'_{1}u'_{t}, u'_{t}y \}$$
the \emph{exceptional path} of $W.$
Moreover, notice that, assuming $r \geq 5$ and $t \geq 6,$ $W - V(P)$ contains an elementary $(r - 1, t - 2)$-handle segment $W'$ defined by taking the union of all horizontal paths and handle edges of $W$ which are disjoint from the exceptional path of $W,$ along with the vertical paths of $W$ that extend to those handle edges.

Now, for every $i \in [p],$ define $P_{i}$ to be the exceptional path of $W_{i - 1},$ where $W_{i} \coloneqq W'_{i - 1},$ concatenated with the two (possibly trivial) subpaths of the $i$-th horizontal path of $W,$ one connecting the $i$-th left boundary vertex of $W$ with the first left boundary vertex of $W',$ and the other connecting the $i$-th right boundary vertex of $W$ with the first right boundary vertex of $W'$ respectively.
We shall refer to $P_{i}$ as the \emph{$i$-th exceptional path} of $W$ (see \cref{fig:handle_segment}).

The ($i$-th) exceptional path of a non-elementary handle segment is defined in the obvious way.

\paragraph{Crosscap segments.}

Let $r, t \in \N_{\geq 4}$ and $p = \min\{r, \lceil \nicefrac{t}{2} \rceil \} - 3.$
An \emph{elementary $(r \times t)$-crosscap segment} $W$ is obtained from an elementary $(r \times 4t)$-wall segment $W_{1}$ (called the \emph{base}) with top boundary vertices $$v_{1}, \ldots, v_{2t}, u_{1}, \ldots, u_{2t}$$ in left to right order by adding all the edges $$\{ v_{i} u_{2t + i} \colon i \in [2t] \},$$ see~\cref{fig:crosscap_segment}.
An \emph{$(r \times t)$-crosscap segment} is a subdivision of an elementary $(r \times t)$-crosscap segment.

\begin{figure}[ht]
    \centering
    \includegraphics[page=1, scale=1.1]{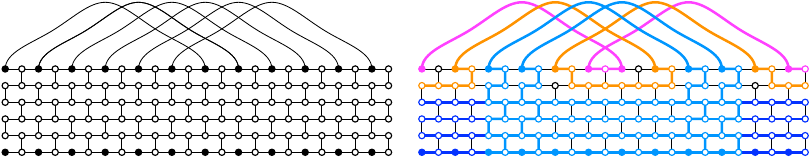}
    \caption{
        A $(6 \times 6)$-crosscap segment with its first and second exceptional path; the remaining $(4 \times 2)$-crosscap segment along with the subpaths connecting its left/right boundary vertices to the original ones; all highlighted by different colors.
    }%
    \label{fig:crosscap_segment}
\end{figure}

Let $x,$ respectively $y,$ be the first left, respectively first right, boundary vertex of $W_{1}.$
We call the $x$-$y$-path $P$ in $W$ on the following set of edges $$\{ xv_{1}, v_{1}u_{1}, u_{1}v_{2t}, v_{2t}u_{2t}, u_{2t}y \}$$
the \emph{exceptional path} of $W.$

Same as for handle segments, if $r, t \geq 5,$ $W - V(P)$ contains an elementary $(r - 1, t - 2)$-crosscap segment $W'$ defined by taking the union of all horizontal paths and crosscap edges of $W$ which are disjoint from the exceptional path of $W,$ along with the vertical paths of $W$ that extend to those crosscap edges.

Now, for every $i \in [p],$ define $P_{i}$ to be the exceptional path of $W_{i - 1},$ where $W_{i} \coloneqq W'_{i - 1},$ concatenated with the two (possibly trivial) subpaths of the $i$-th horizontal path of $W,$ one connecting the $i$-th left boundary vertex of $W$ with the first left boundary vertex of $W',$ and the other connecting the $i$-th right boundary vertex of $W$ with the first right boundary vertex of $W'$ respectively.
We shall refer to $P_{i}$ as the \emph{$i$-th exceptional path} of $W$ (see \cref{fig:crosscap_segment}).

The ($i$-th) exceptional path of a non-elementary crosscap segment is defined in the obvious way.

\paragraph{Surface walls.}

Let $r, t \in \N_{\geq 4}$ and $\sfh, \sfc \in \N.$
An \emph{elementary $(r, t, \sfh, \sfc)$-surface wall} $W$ is the concatenation of $1+\sfh + \sfc$ many segments:
\begin{itemize}
	\item one elementary $(r \times r)$-wall segment,
	\item $\sfh$ elementary $(r \times t)$-handle segments, and
	\item $\sfc$ elementary $(r \times t)$-crosscap segments.
\end{itemize} 
An \emph{$(r, t, \sfh, \sfc)$-surface wall} is a subdivision of an elementary $(r, t, \sfh, \sfc)$-surface wall.

\paragraph{$\Sigma$-walloids.}

Let $r, t \in \N_{\geq 4},$ $\sfh, \sfc, a, b \in \N,$ and $\Sigma=\Sigma^{(\sfh,\sfc)}$.
An \emph{(elementary) $(r, t, a, b)$-$\Sigma$-walloid} $W$ is defined in the same way as an (elementary) $(r, t, a, b)$-walloid, except that we replace the elementary $(r \times r)$-wall segment in the cylindrical concatenation defining $W,$ with an elementary $(r, t, \sfh, \sfc)$-surface wall.
In case $a = b = 0,$ we call $W$ an \emph{$(r, t)$-$\Sigma$-annulus wall}.
In case $r = t,$ we call $W$ a \emph{$t$-$\Sigma$-annulus wall}. As usual, the \emph{base annulus} of a $\Sigma$-walloid $W$ is the cylindrical concatenation of the bases of its segments.
Note that, by definition, $W$ is subcubic and a subdivision of a $3$-connected graph.

Also, for each $i \in [p],$ we define the \emph{$i$-th exceptional cycle} of $W$ as the cycle obtained by concatenating the following paths:
\begin{itemize}
	\item the $i$-th horizontal path of each segment of $W$ that is not a handle or a crosscap segment; and
	\item the $i$-th exceptional path of each handle and crosscap segment of $W.$
\end{itemize}

\medskip
It is known that $\Sigma^{(\sfh, \sfc)}$-annulus walls are universal graphs with respect to minors, for graphs embeddable on $\Sigma,$ for any given surface $\Sigma$ homeomorphic to $\Sigma^{(\sfh, \sfc)}.$
The following essentially follows from the results of Gavoille and Hilaire \cite{GavoilleH2023Minor}, and was first observed by Thilikos and Wiederrecht in~\cite{ThilikosW2024Excluding} with exponential bounds.
The following statement with polynomial bounds is from the work of Gorsky, Seweryn, and Wiederrecht~\cite{GorskySW2025Polynomial}.

\begin{proposition}[\cite{GorskySW2025Polynomial}]\label{prop:surface_walls_universal} There is a universal constant $a_{\ref{prop:surface_walls_universal}}$ such that for all integers $\sfh, \sfc\in \N$, every graph $H$ that embeds on $\Sigma^{(\sfh, \sfc)}$ is a minor of the $t$-$\Sigma^{(\sfh, \sfc)}$-annulus wall, where $t = a_{\ref{prop:surface_walls_universal}}\cdot  g^{4} (|H| + g)^{2}$ and $g = 2\sfh + \sfc.$
\end{proposition}

Moreover, it is fairly easy to observe that for every surface $\Sigma^{(\sfh, \sfc)},$ every large enough surface wall admits an embedding on $\Sigma^{(\sfh, \sfc)}$ of high representativity, which is moreover unique and satisfies a number of desired properties.

\begin{lemma}\label{lemma:surface_wall_repr}
    There is a universal constant $a_{\ref{lemma:surface_wall_repr}}$ such that for all $\sfh, \sfc\in \N$ and $t\in \N_{\geq 4},$ every $(a_{\ref{lemma:surface_wall_repr}}g + 2t)$-$\Sigma^{(\sfh, \sfc)}$-annulus wall $W$ has an embedding $\Gamma$ on $\Sigma^{(\sfh, \sfc)}$ such that the following conditions hold:
    \begin{itemize}
        \item $\Gamma$ has representativity $a_{\ref{lemma:surface_wall_repr}}g + 2t$;
        \item every embedding of $W$ on $\Sigma^{(\sfh, \sfc)}$ is homeomorphic to $\Gamma$; and
        \item for all $i \in [t],$ $\Gamma$ maps the $i$-th exceptional cycle $C_i$ of $W$ to a contractible curve $\gamma_i$ on $\Sigma^{(\sfh, \sfc)}.$ Moreover, there is a disk component of $\Sigma^{(\sfh, \sfc)}-\gamma_i$ that contains the drawing of the $i'$-th exceptional cycle of $W$ in~$\Gamma$, for all $i'<i$.
    \end{itemize}
\end{lemma}
\begin{proof}
	We choose $a_{\ref{lemma:surface_wall_repr}}\coloneqq 100$. For brevity, let us denote $t'\coloneqq 100g + 2t$.
	Since the statement of the lemma is preserved under subdividing edges,
	we may focus on the case when $W$ is the elementary $t'$-$\Sigma^{(\sfh, \sfc)}$-annulus wall. We also assume that $g\geq 1$, for the statement is trivial when $g=0$ (i.e., $\Sigma^{(\sfh, \sfc)}$ is the sphere) due to $W$ being $3$-connected.
	
	Let us first construct the embedding $\Gamma$. Let $W_0$ be the subgraph of $W$ obtained by removing, for every handle segment and every crosscap segment of $W$, the edges added to the top boundary of the segment in the construction of the handle/crosscap segments. Thus, $W_0$ is an annulus wall with $t'$ rows and $(1+4\sfh+4\sfc)t'$ columns. We start by taking a sphere $\Sigma_0$ and embedding $W_0$ into $\Sigma_0$ naturally; call this embedding $\Gamma_0$. Note that the inner cycle of $W_0$ bounds a face $\Delta_{\mathsf{in}}$ in $\Gamma_0$.
	
	Now, amend $\Sigma_0$ into a surface homeomorphic to $\Sigma^{(\sfh,\sfc)}$ by adding $\sfh$ handles and $\sfc$ crosscaps inside the face $\Delta_{\mathsf{in}}$. Then, for every handle segment $W'$ of $W$, use one of those handles to embed the additional edges of $W'$ that are not present in $W_0$. Similarly, for every crosscap segment $W'$ of $W$, use one of the crosscaps to embed the edges of $W'$ that are not present in $W_0$.
    Thus, we obtain an embedding $\Gamma$ of $W$ in~$\Sigma^{(\sfh,\sfc)}$.
	
	It is straightforward to see from the construction that the representativity of $\Gamma$ is $t'\geq 100g$. Therefore, from \cref{prop:representativity_implies_uniqueness} we infer that every embedding of $W$ on $\Sigma^{(\sfh,\sfc)}$ is homeomorphic to $\Gamma$. This establishes the first two assertions from the lemma statement.
	
	For the last assertion, for every $j\in [t]$ let $C_j$ be the $j$-th exceptional cycle of $W$ and $\gamma_j$ be the closed curve on $\Sigma^{(\sfh,\sfc)}$ that is the image of $C_j$ in $\Gamma$. Fix some $i\in [t]$.
	Let $D$ be the connected component of $W-V(C_i)$ that contains the last $100g$ base cycles of $W$.
	Observe that $D$ still contains a $100g$-$\Sigma^{(\sfh,\sfc)}$-annulus wall, so by applying \cref{prop:representativity_implies_uniqueness} to this wall, we infer that $D$ cannot be embedded on a surface of genus smaller than $g$. However, $D$ clearly has an embedding on the surface $\Sigma-\gamma_i$. We conclude that $\gamma_i$ must be contractible, for otherwise the genus of every component of $\Sigma-\gamma_i$ would be strictly smaller than $g$. Therefore, $\gamma_i$ is two-sided and $\Sigma-\gamma_i$ consists of two components: the component containing $D$ that is homeomorphic to $\Sigma^{(\sfh,\sfc)}$ with a cuff, and a disk $\Delta$. We finally observe that every curve $\gamma_{i'}$, $i'<i$, must lie within $\Delta$, for within the base annulus of $W$ one can find paths $P,Q$ leading from $C_{i'}$ to $C_i$ and from $D$ to~$C_i$, respectively, and these paths reach $C_i$ on the opposite sides of $C_i$ in the embedding $\Gamma$.
\end{proof}

Based on the lemma above, we prove a follow-up lemma that will be instrumental in generalizing our arguments from the planar case to surfaces.

\begin{lemma}\label{lemma:walloid_in_disk}
    There is a function $f_{\ref{lemma:walloid_in_disk}} \colon \N^{3} \to \N$ such that for all $\sfh, \sfc\in \N$ and $t \in \N_{\geq 4}$ the following holds.
    Suppose $W^\star$ is an $(f_{\ref{lemma:walloid_in_disk}}(\sfh, \sfc, t), f_{\ref{lemma:walloid_in_disk}}(\sfh, \sfc, t), a, b)$-$\Sigma^{(\sfh, \sfc)}$-walloid $W,$ for some $a, b \in \N.$
    Then, there is a $(t, t, a, b)$-walloid $W \subseteq W^\star$ such that
    \begin{itemize}
        \item the base cycles of $W$ are the first $t$ exceptional cycles of $W^\star$;
        \item every segment $W_{1}$ of $W$ is a subgraph of a segment $W_{2}$ of $W^\star$ and $W_{2} - W_{1}$ is a subgraph of the base annulus of $W^\star$; and
        \item the outermost cycle of $W$ bounds a closed disk $\Delta \subseteq \Sigma^{(\sfh, \sfc)}$ such that the drawing of $W$ induced by $\Gamma_{W^\star}$ on $\Sigma^{(\sfh, \sfc)}$ is an embedding of $W$ in $\Delta.$
    \end{itemize}
    Moreover, we have that $f_{\ref{lemma:walloid_in_disk}}(\sfh, \sfc, t) = a_{\ref{lemma:walloid_in_disk}} g + t$ for some universal constant $a_{\ref{lemma:walloid_in_disk}}\in \N$, where $g=2\sfh+\sfc$.
\end{lemma}
\begin{proof}
    Set $a_{\ref{lemma:walloid_in_disk}}\coloneqq a_{\ref{lemma:surface_wall_repr}}$. Thus $W^\star$ is a $(t',t',a,b)$-$\Sigma^{(\sfh, \sfc)}$-walloid, where $t'\coloneqq f_{\ref{lemma:walloid_in_disk}}(\sfh, \sfc, t)=a_{\ref{lemma:surface_wall_repr}}g+2t$.
    
    We construct $W$ from $W^\star$ by removing all but the first $t$ base cycles of $W^\star$ and, additionally in every handle or crosscap segment $W'$ of $W^\star$, removing all the edges that do not participate in the first $t$ exceptional paths of $W'$. It is straightforward to see that $W$ is thus a $(t, t, a, b)$-walloid, as promised. Moreover, the base cycles of $W$ are the first $t$ exceptional cycles of $W^\star$, and every segment $W_1$ of $W$ is a subgraph of a segment $W_2$ of $W^\star$ with $W_2-W_1$ being a subgraph of the base annulus of $W^\star$. This settles the first two assertions from the lemma statement. The last assertion follows immediately from the last assertion of \cref{lemma:walloid_in_disk}, applied to the $t$-th exceptional cycle of $W^\star$.
\end{proof}

\paragraph{(Tight) $\Sigma$-renditions.}

We extend the notion of $\Sigma$-renditions to graphs admitting an embedding on a arbitrary surface $\Sigma$ in the obvious way.
All previously defined notions like ground vertices, cells, aligned disks, and grounded subgraphs lift to this setting directly.
So does the notion of tight renditions.

Below we extend \cref{obs:tight_rendition} to surfaces.

\begin{observation}\label{obs:tight_rendition_surfaces} There is an algorithm that, given a graph $G$ embedded on a surface $\Sigma^{(\sfh, \sfc)}$ with a $t$-$\Sigma^{(\sfh, \sfc)}$-annulus wall, for some $\sfh, \sfc \in \N$ and $t \in \N_{\geq a_{\ref{lemma:surface_wall_repr}}g + 4},$ computes a tight $\Sigma$-rendition $\rend$ of $G$ with $W$ being $\rend$-grounded in time $\Oh(|G| + \|G\|).$
\end{observation}
\begin{proof}[Proof sketch]
	The proof proceeds exactly like the proof of \cref{obs:tight_rendition}.
    The only caveat is that we are no longer in the planar setting and therefore we cannot argue that minimal separators can be traced by contractible curves on the surface in general.
    However, since $t \geq a_{\ref{lemma:surface_wall_repr}}g + 4,$ by \cref{lemma:surface_wall_repr}, we may assume to be working with the unique embedding of $W$ that has representativity $t \geq 4.$
    Therefore, any closed curve $\gamma$ intersecting the embedding of the $\rend$-torso $T$ of $G$ only at a single cut-vertex of $T,$ has to be contractible, for one of the components of $\Sigma-\gamma$ must embed the entirety of $W$.
    Therefore, we may argue that the closed curve $J$ that bounds a disk as in the proof of \cref{obs:tight_rendition} exists and conclude in the same way.
\end{proof}

\paragraph{Representing terminals.}

We first need the statement for surface-embeddable graphs corresponding to \cref{lst_fi}. This, again, was proved by Paul, Protopapas, Thilikos, and Wiederrecht~\cite{PaulPTS2025LocalIndex}.

\begin{proposition}[{\cite[Lemma 8.1]{PaulPTS2025LocalIndex}}]\label{lst_fi_genus} There exist functions $\mathsf{annulus}_{\ref{lst_fi_genus}} \colon \N^{6} \to \N$, $\mathsf{breadth}_{\ref{lst_fi_genus}} \colon \N^{2} \to \N$, and $\depth_{\ref{lst_fi_genus}} \colon \N^{3} \to \N$ such that for all $r,t \in \N_{\geq 4},$ all $\sfh, \sfc \in \N,$ and all $\ell, q \in \N_{\geq 1},$ the following holds.
	
Suppose $G$ is a graph embeddable on a surface $\Sigma$ of genus $2\sfh + \sfc$ with a tight $\Sigma$-rendition $\rend$, $W \subseteq G$ is a $\rend$-grounded $\mathsf{annulus}_{\ref{lst_fi_genus}}(r, t, \sfh, \sfc, \ell, q)$-$\Sigma$-annulus wall, and $\chi$ is a cell-coloring of $\rend$ of capacity at most $\ell.$

Then there is a $\rend$-well-grounded $(r, t, a, b)$-$\Sigma$-walloid $W' \subseteq G$ for some $a \in \N$ and $b \leq \mathsf{breadth}_{\ref{lst_fi_genus}}(\ell, q),$ with a sandwich sequence $\langle (\FDisk^{\insf}_{j},\FDisk^{\outsf}_{j})\colon j\in [b] \rangle$, such that the following conditions hold:
	\begin{enumerate}
		\item $W'$ is controlled by $W$;
		\item $W'$ has depth at most $\depth_{\ref{lst_fi_genus}}(t, \ell, q)$ with respect to $\rend$; and
		\item For every $\alpha \in \chi(\rend)$, if $\rend_{\alpha} = \{ \Delta \in \rend \mid \chi(\Delta) = \alpha\}$, then either
		\begin{itemize}
			\item $\bigcup \rend_{\alpha} \subseteq \bigcup_{j \in [b]} \inte(\FDisk^{\insf}_{j})$; or
			\item there is a sequence of $q$ consecutive flap segments of $W'$, each hosting a cell from $\rend_{\alpha}$.
		\end{itemize}
	\end{enumerate}
	Moreover, it holds that
	\begin{align*}
		\mathsf{annulus}_{\ref{lst_fi_genus}}(r, t, \sfh, \sfc, \ell, q) \ &\in \ (r + (\ell q + 1)^{q \cdot 2^{\Oh(\ell)}} \cdot t)^{2^{\Oh(\ell)}} \cdot (\sfh + \sfc + 1)^{2^{\Oh(\ell)}}\\
		\mathsf{breadth}_{\ref{lst_fi_genus}}(\ell, q) \ &\in \ q \cdot 2^{\Oh(\ell)}\text{, and}\\
		\depth_{\ref{lst_fi_genus}}(t, \ell, q) \ &\in \ (\ell q + 1)^{q \cdot 2^{\Oh(\ell)}} \cdot t.
	\end{align*}
\end{proposition}

From this we derive the following statement corresponding to \cref{structure_relative_terminals}.

\begin{lemma}\label{structure_relative_terminals_genus} There exist functions $\mathsf{annulus}_{\ref{structure_relative_terminals_genus}} \colon \N^{5} \to \N,$ $\depth_{\ref{structure_relative_terminals_genus}} \colon \N^{2} \to \N,$ and $\mathsf{breadth}_{\ref{structure_relative_terminals_genus}} \colon \N \to \N$ such that, for all $r,t \in \N_{\geq 4},$ all $\sfh, \sfc \in \N,$ and all $q \in \N_{\geq 1}$ the following holds.

	Suppose $G$ is a graph embeddable on a surface $\Sigma$ of genus $2\sfh + \sfc$ with an $\mathsf{annulus}_{\ref{structure_relative_terminals_genus}}(r, t, \sfh, \sfc, q)$-$\Sigma$-annulus wall $W \subseteq G,$ and $S, T \subseteq V(G)$ is a pair of vertex subsets.
    Also, suppose $F \subseteq V(G)$ is a vertex subset of size strictly less than $q.$
	
	Then, there is a $\Sigma$-rendition $\rend$ of $G$ and a $\rend$-well-grounded $(r, t, a, b)$-$\Sigma$-walloid $W' \subseteq G$ for some $a \in \N$ and $b \leq \mathsf{breadth}_{\ref{structure_relative_terminals_genus}}(q),$ with a sandwich sequence $\langle (\FDisk^{\insf}_{j},\FDisk^{\outsf}_{j})\colon j\in [b] \rangle$, such that the following conditions~hold:
	\begin{enumerate}
		\item $W'$ is controlled by $W$;
		\item $W'$ has depth at most $\depth_{\ref{lst_fi_genus}}(t, q)$ with respect to $\rend$;
		\item $F \subseteq \bigcup_{j \in [b]} \inte(\Theta^{\insf}_{j})$; and
		\item for each $\mathfrak{R} \in \{ \rend_{S} \cap \rend_{T}, \rend_{S} \setminus \rend_{T}, \rend_{T} \setminus \rend_{S} \}$, either $\bigcup \mathfrak{R} \subseteq \bigcup_{j \in [b]} \inte(\FDisk^{\insf}_{j})$ or $W'$ $q$-represents $\mathfrak{R}.$
	\end{enumerate}
	Moreover, it holds that
	\begin{align*}
		\mathsf{annulus}_{\ref{structure_relative_terminals_genus}}(r, t, q) &\in (r + 2^{\Oh(q \log q)} \cdot t)^{\Oh(1)} \cdot (\sfh + \sfc + 1)^{\Oh(1)},\\
		\mathsf{breadth}_{\ref{structure_relative_terminals_genus}}(q) &\in \Oh(q), \quad \text{and}\\
		\depth_{\ref{structure_relative_terminals_genus}}(t, q) &\in 2^{\Oh(q \log q)} \cdot t.
	\end{align*}
\end{lemma}
\begin{proof}
	Same as that of \cref{structure_relative_terminals}, with \cref{lst_fi} replaced with \cref{lst_fi_genus} and \cref{obs:tight_rendition} replaced with \cref{obs:tight_rendition_surfaces}.
\end{proof}

\subsection{A ``local'' structure theorem for surfaces}

Similarly to the planar case, below we prove a ``local'' statement akin to \cref{local_structure_2}.
In fact, to prove \cref{local_structure_genus}, after employing some arguments particular to surface-embeddable graphs, the bulk of the work relies on recycling the arguments from the proof of \cref{local_structure_2}.

\begin{theorem}\label{local_structure_genus} There exist functions $\repr_{\ref{local_structure_genus}} \colon \N^{4} \to \N,$ and $\sep_{\ref{local_structure_genus}}, \cov_{\ref{local_structure_genus}} \colon \N^{3} \to \N$ such that for all integers $g,d,k \geq 0$ and $q \geq 1$ the following holds.

Suppose $G$ is a connected graph that has an embedding on a surface of genus $g$ with representativity at least $\repr_{\ref{local_structure_genus}}(g, d, k, q).$
Suppose further that $F \subseteq V(G)$ is a vertex subset of size strictly less than $q$ and $S, T \subseteq V(G).$
Then one of the following conditions holds:
	\begin{enumerate}
		\item $G$ contains a $d$-scattered $S$-$T$-linkage of order $k$ that avoids $\Ball^{\lfloor \nicefrac{d}{2} \rfloor}_{G}(F)$; or
		\item there exists a non-empty set $\{ (A_{i}, B_{i}) \colon i \in [n] \}$  of separations of $G,$ for some $n \in \N_{\geq 1},$ such that
		\begin{itemize}
			\item for every $i \in [n],$ $|A_{i} \cap B_{i}| \leq \sep_{\ref{local_structure_genus}}(d, k, q)$;
			\item for every $i \in [n],$ the graph $G[A_{i}]$ is planar; and
			\item for all distinct $i,j \in [n],$ we have $(A_{i} \setminus B_{i}) \cap (A_{j} \setminus B_{j}) = \emptyset.$
		\end{itemize}
		Further, assuming that $$A \coloneqq \bigcup_{i \in [n]} A_{i} \quad\text{and}\quad B \coloneqq V(G) \setminus \bigcup_{i \in [n]} (A_{i} \setminus B_{i}),$$ there exists a set $Y \subseteq A \cap B$ of size at most $\cov_{\ref{local_structure_genus}}(d, k, q)$ such that every $S$-$T$-path in $G$ that avoids $\Ball^{d}_{G}(F)$ and intersects $B,$ also intersects $\Ball^{d}_{G[A]}(Y).$
	\end{enumerate}
    Moreover, it holds that
    \begin{align*}
    &\repr_{\ref{local_structure_genus}}(g, d, k, q) \in g^{\Oh(1)} \cdot d^{\Oh(1)} \cdot 2^{\Oh(\max\{k, q \} \log \max \{k, q\})} \quad\text{and}\\
    &\sep_{\ref{local_structure_genus}}(d, k, q), \cov_{\ref{local_structure_genus}}(d, k, q) \in d^{3} \cdot 2^{\Oh(\max\{k, q\} \log \max\{k, q\})}.
    \end{align*}
\end{theorem}
\begin{proof}[Proof sketch.] Let $\Sigma$ be the surface of genus $g$ on which $G$ embeds with the assumed representativity.
    Also, let $\sfh, \sfc \geq 0$ be integers such that $\Sigma$ is homeomorphic to $\Sigma^{(\sfh, \sfc)}.$

    The proof exploits the highly representative embedding of $G$ on $\Sigma$, combined with \cref{prop:representativity_gives_minor}, \cref{structure_relative_terminals_genus} and \cref{lemma:walloid_in_disk} to produce a large walloid, embedded in a disk of $\Sigma,$ that already captures all the neccesary information regarding the terminals sets $S$ and $T.$
    This will allow us to localize the proof to this disk, and consequently to a planar subgraph of $G.$
    Then,  we may recycle the arguments from the proof of \cref{local_structure_2} to conclude.

    Let us first discuss the numbers involved.
    Similarly to the proof of \cref{local_structure_2}, we define an auxiliary function $$\mathsf{t}(d, k) \coloneqq \max \{ f_{\ref{combing_lemma}}(d, 2k), g_{\ref{combing_lemma}}(d, 2k), (k + 2)(d + 1) \} + d \in d^{3} \cdot 2^{\Oh(k)},$$
    which determines the order of the final walloid we will need.

    Also, in order to force the vertices of $F$ to be confined within the inner vortex disks of the vortex segments we will apply \cref{structure_relative_terminals_genus} with the representation parameter $q$ being $k' \coloneqq \max \{ k, q\}.$

    This affects the adhesion of the linear decompositions, which is set to $$\depth_{\ref{structure_relative_terminals_genus}}(\mathsf{t}(d, k), k') \in d^{3} \cdot 2^{\Oh(k' \log k')}.$$
Consequently, the order of separation required to separate each bag of the linear decompositions is $$\sep_{\ref{local_structure_genus}}(d, k, q) \coloneqq 2 \cdot \depth_{\ref{structure_relative_terminals_genus}}(\mathsf{t}(d, k), k') + 1 \in d^{3} \cdot 2^{\Oh(k' \log k')}.$$
    For these conditions to hold, we need to set our requirement about the representativity to be
    $$\repr_{\ref{local_structure_genus}}(g, d, k, q) \coloneqq \repr_{\ref{prop:representativity_gives_minor}}(g, \mathsf{n}(\mathsf{annulus}_{\ref{structure_relative_terminals_genus}}(t^\star, t^\star, \sfh, \sfc, k'))) \in g^{\Oh(1)} \cdot d^{\Oh(1)} \cdot 2^{\Oh(k' \log k')},$$
    where $t^\star \coloneqq f_{\ref{lemma:walloid_in_disk}}(\sfh, \sfc, \mathsf{t}(d, k)),$ and given an elementary $x$-$\Sigma$-annulus wall, $x \in \N_{\geq 4},$ $\mathsf{n}(x)$ returns its total number of vertices, which is roughly $\Oh(x^{2}).$

    Finally, the size of the hitting set that we will obtain follows from the bounds in \cref{lem:harvesting_lemma}, yielding $$\cov_{\ref{local_structure_genus}}(d, k, q) \coloneqq 2k (3 \cdot \depth_{\ref{structure_relative_terminals_genus}}(\mathsf{t}(d, k), k') + 1) \cdot \mathsf{breadth}_{\ref{structure_relative_terminals_genus}}(k') \in d^{3} \cdot 2^{\Oh(k' \log k')},$$ where the depth of the walloid is determined by the requirements described above.

    The claimed asymptotic bounds follow directly from the bounds on the corresponding functions.

    \medskip
    We proceed with a sketch of the proof.
    First apply \cref{prop:representativity_gives_minor} on $G$ and $\Sigma,$ asking for an $\mathsf{annulus}_{\ref{structure_relative_terminals_genus}}(t^\star, t^\star, \sfh, \sfc, k')$-$\Sigma$-annulus wall $W^1$ as a minor of $G.$
    Note that since $W^1$ is subcubic, we may assume that (a subdivision of) $W^1$ is in fact a subgraph of $G.$

	With $W^1$ in hand, we may invoke \cref{structure_relative_terminals_genus}, and obtain a $\Sigma$-rendition $\rend$ of $G$ and a $\rend$-well-grounded $(t^\star, t^\star, a, b)$-walloid $W^\star \subseteq G$ for some $a \in \N$ and $b \in [\max \{1, \mathsf{breadth}_{\ref{structure_relative_terminals_genus}}(k') \}],$ along with a sandwich sequence $\Zz \coloneqq \langle\FDisk^{\insf}_{j},\FDisk^{\outsf}_{j}\colon j\in [b]\rangle$ such that the following conditions hold:
	\begin{enumerate}
		\item $W^\star$ has depth at most $\depth_{\ref{structure_relative_terminals_genus}}(\mathsf{t}(d, k), k')$ with respect to $\rend;$
		\item $F \subseteq \bigcup_{j \in [b]} \inte(\Theta^\insf_j)$; and
		\item for each $\mathfrak{R} \in \{ \rend_{S} \cap \rend_{T}, \rend_{S} \setminus \rend_{T}, \rend_{T} \setminus \rend_{S} \},$ either $$\bigcup \mathfrak{R} \subseteq \bigcup_{j \in [b]} \mathsf{int}(\FDisk^{\mathsf{in}}_{j})$$ or $W^\star$ $k$-represents $\mathfrak{R}$ (since $k' \geq k$).
	\end{enumerate}

    From here, consider the embedding $\Gamma_{W^\star}$ of $W^\star$ on $\Sigma$ induced by the embedding of $G$ on $\Sigma.$
    Applying \cref{lemma:walloid_in_disk} on $W^\star,$ first tells us that up to homeomorphism $\Gamma_{W^\star}$ is unique, and also yields a $(\mathsf{t}(d, k), \mathsf{t}(d, k), a, b)$-walloid $W \subseteq W^\star$ such that
    \begin{itemize}
        \item the base cycles of $W$ are the first $t$ exceptional cycles of $W^\star$;
        \item every segment $W_{1}$ of $W$ is a subgraph of a segment $W_{2}$ of $W^\star$ and $W_{2} - W_{1}$ is a subgraph of the base annulus of $W^\star$; and
        \item the outermost cycle of $W$ bounds a closed disk $\Delta \subseteq \Sigma$ such that the drawing of $W$ induced by $\Gamma_{W^\star}$ on $\Sigma^{(\sfh, \sfc)}$ is an embedding of $W$ in $\Delta.$
    \end{itemize}
    Now note that by the second and third condition above, we have that for every $j \in [b],$ $\Theta^\insf_j \subseteq \Theta^\outsf_j \subsetneq \Delta,$ and that $\Zz$ is a valid sandwich sequence for $W$ as well.

    Finally, let $t \coloneqq \mathsf{t}(d, k) - d,$ let $\Cc \coloneqq \{ C_{t - d}, \ldots, C_{t} \}$ denote the nest consisting of the $d$ outermost base cycles of $W,$ and let $W'$ be the $(t, t, a, b)$-walloid $W' \subseteq W$ obtained from $W$ by removing from its base cycles the cycles of $\Cc$ and appropriatelly resizing all its segments by removing some vertical paths.
    In particular, we still have that $\Zz$ is a valid sandwich sequence for $W'.$

    From this point forward we may apply the arguments of \cref{local_structure_2} from \cref{claim:local_1} onwards on the graph $G' \coloneqq G \cap \Delta$ with the vertex subset $F$ (note that by the guarantees of \cref{structure_relative_terminals_genus} and \cref{obs:special_vertices_hidden}, $F \subseteq V(G')$ and that $\Ball^{d}_{G}(F) = \Ball^{d}_{G'}(F)$), equipped with the walloid $W',$ and obtain one of the following two outcomes:
    \begin{enumerate}
        \item a $d$-scattered $S$-$T$-linkage $\Qq$ of order $k$ in $G'$ that avoids $\Ball^{\lfloor \nicefrac{d}{2} \rfloor}_{G'}(F)$; or
        \item a non-empty set $\{ (A_{i}, B_{i}) \colon i \in [n] \}$ of separations of $G',$ for some $n \in \N_{\geq 1},$ such that
		\begin{itemize}
			\item for every $i \in [n],$ $|A_{i} \cap B_{i}| \leq \sep_{\ref{local_structure_genus}}(d, k, q)$;
			\item for every $i \in [n],$ $A_{i} \subseteq \inte(\Delta)$ (since $\Theta^\outsf_{j} \subsetneq \Delta$ for all $j \in [b]$); and
			\item for all distinct $i,j \in [n],$ we have $(A_{i} \setminus B_{i}) \cap (A_{j} \setminus B_{j}) = \emptyset.$
		\end{itemize}
		Further, assuming that $$A \coloneqq \bigcup_{i \in [n]} A_{i} \quad\text{and}\quad B \coloneqq V(G') \setminus \bigcup_{i \in [n]} (A_{i} \setminus B_{i}),$$ there exists a set $Y \subseteq A \cap B$ of size at most $\cov_{\ref{local_structure_genus}}(d, k, q)$ such that every $S$-$T$-path in $G'$ that avoids $\Ball^{d}_{G'}(F)$ and intersects $B$ also intersects $\Ball^{d}_{G'[A]}(Y).$
    \end{enumerate}
    To conclude the proof it remains to show that both outcomes can be lifted back to the original graph $G.$

    \smallskip
    \noindent\textbf{Outcome i).} A careful examination of the proof of \cref{local_structure_2} shows that every path of $\Qq$ is in fact found within the graph drawn in the disk $\Delta'$ bounded by the simple cycle of $W'$ which contains the drawing of~$W'.$
    This implies that the nest $\Cc$ acts an an insulation layer of $d$ many cycles between the exterior of $\Delta$ and the interior of $\Delta'$; any path in $G,$ say $R,$ that is not a path in $G \cap \Delta,$ and connects a vertex of a path of $\Qq$ with a vertex of another path of $\Qq,$ must intersect every cycle of $\Cc$ at an internal vertex.
    This implies that the length of $R$ is larger than $d,$ which allows us to conclude that $\Qq$ is in fact $d$-scattered in $G.$

    \smallskip
    \noindent\textbf{Outcome ii).} First note that, since $A_{i} \subseteq \inte(\Delta),$ each separation $(A_{i}, B_{i})$ can be lifted to a separation $(A_{i}, B'_{i})$ of $G$ where $$B'_{i} \coloneqq B_i \cup V(G - V(G \cap \Delta))).$$
    Moreover, we have $A_{i} \cap B_{i} = A_{i} \cap B'_{i}$ and $G[A_{i}]$ is planar as desired.
    Also, let $$B' \coloneqq V(G) \setminus \bigcup_{i \in [n]} (A_{i} \setminus B'_{i}).$$
    Now, a careful look at \cref{claim:local_2} and \cref{claim:local_3} from the proof of \cref{local_structure_2} allows us to conclude that $\Ball^{d}_{G[A]}(Y)$ is still a valid hitting set for every $S$-$T$-path in $G$ that avoids $\Ball^{d}_{G}(F)$ and intersects $B'.$
\end{proof}

\subsection{A ``global'' structure theorem for surfaces}

We now turn \cref{local_structure_genus} into a global theorem akin \cref{global_structure}.

\begin{theorem}\label{global_structure_genus}
    There exist functions $\vig_{\ref{global_structure_genus}},$ $\adh_{\ref{global_structure_genus}} \colon \N^{3} \to \N$ such that for all integers $g,d,k \geq 0$ and $q \geq 1$ the following holds.
	Suppose $G$ is a connected graph that has an embedding on a surface of genus $g$ with representativity at least $\repr_{\ref{local_structure_genus}}(g, d, k, q).$
    Suppose further that $F \subseteq V(G)$ is a vertex subset of size strictly less than $q$ and $S, T \subseteq V(G).$
    Then, either
	\begin{itemize}
		\item $G$ contains a $d$-scattered $S$-$T$-linkage of order $k$ that avoids $\Ball^{\lfloor \nicefrac{d}{2} \rfloor}_{G}(F)$; or
		\item $G$ admits a guarded tree-decomposition $(T, \beta, \gamma)$ of adhesion at most $\adh_{\ref{global_structure_genus}}(d, k, q)$ and vigilance at most $\vig_{\ref{global_structure_genus}}(d, k, q)$ that distance-$d$ guards all $S$-$T$-paths in $G$ that avoid $\Ball^{d}_{G}(F).$
	\end{itemize}
	Moreover, it holds that
	$$\adh_{\ref{global_structure_genus}}(d, k, q), \vig_{\ref{global_structure_genus}}(d, k, q) \in 2^{d^{\Oh(1)} \cdot 2^{\Oh(\max\{ k, q \} \log \max \{ k, q \})}}.$$
\end{theorem}
\begin{proof}We first define the functions involved.
    Let $x \coloneqq \sep_{\ref{local_structure_genus}}(d, k, q).$
    Then, we define:
    \begin{align*}
        \adh_{\ref{global_structure_genus}}(d, k, q) &\coloneqq \adh_{\ref{global_structure}}(d, k, q + x) + x\quad\text{and}\\
        \vig_{\ref{global_structure_genus}}(d, k, q) &\coloneqq \max\{ \cov_{\ref{local_structure_genus}}(d, k, q), \vig_{\ref{global_structure}}(d, k, q + x) \} + x.
    \end{align*}
    The claimed asymptotic bounds follow from the corresponding asymptotic bounds on the functions involved in the right hand sides above.

    We commence with an application of \cref{local_structure_genus} on $G.$
    This yields one of two outcomes:
	\begin{enumerate}
		\item a $d$-scattered $S$-$T$-linkage of order $k$ that avoids $\Ball^{\lfloor \nicefrac{d}{2} \rfloor}_{G}(F)$; or
		\item a non-empty set $\{ (A_{i}, B_{i}) \colon i \in [n] \}$  of separations of $G,$ for some $n \in \N_{\geq 1},$ such that
		\begin{itemize}
			\item for every $i \in [n],$ $|A_{i} \cap B_{i}| \leq x$;
			\item for every $i \in [n],$ the graph $G[A_{i}]$ is planar; and
			\item for all distinct $i,j \in [n],$ we have $(A_{i} \setminus B_{i}) \cap (A_{j} \setminus B_{j}) = \emptyset.$
		\end{itemize}
		Further, assuming that $$A \coloneqq \bigcup_{i \in [n]} A_{i} \quad\text{and}\quad B \coloneqq V(G) \setminus \bigcup_{i \in [n]} (A_{i} \setminus B_{i}),$$ there exists a set $Y \subseteq A \cap B$ of size at most $y$ such that every $S$-$T$-path in $G$ that avoids $\Ball^{d}_{G}(F)$ and intersects $B,$ also intersects $\Ball^{d}_{G[A]}(Y).$
	\end{enumerate}
    Note that with outcome i) we immediately conclude.
    So we may assume we have obtained outcome~ii).

    Fix an index $i \in [n].$
    Let $F_{i} \coloneqq F \cap A_{i}$ and $F^\star_{i} \coloneqq F_{i} \cup (A_{i} \cap B_{i}).$
    Note that $|F^\star_{i}| \leq |F| + |A_{i} \cap B_{i}| < q + x.$
    Also, let $G_{i} \coloneqq G[A_{i}],$ $S_{i} \coloneqq S \cap A_{i},$ and $T_{i} \coloneqq T \cap A_{i}.$
    We now apply \cref{global_structure} on $G_{i}$ with vertex subset $F^\star_{i},$ terminal sets $S_{i}, T_{i},$ and parameters $d, k$ and $q + x.$
    We obtain one of two outcomes:
    \begin{enumerate}
        \item a $d$-scattered $S_{i}$-$T_{i}$-linkage of order $k$ in $G_i$ that avoids $\Ball^{\lfloor \nicefrac{d}{2} \rfloor}_{G_{i}}(F^\star_{i})$; or
		\item a guarded tree-decomposition of $G_{i}$ of adhesion at most $\adh_{\ref{global_structure}}(d, k, q + x)$ and vigilance at most $\vig_{\ref{global_structure}}(d, k, q + x)$ that distance-$d$ guards all $S_{i}$-$T_{i}$-paths in $G_{i}$ that avoid $\Ball^{d}_{G_{i}}(F^\star_{i}).$
    \end{enumerate}

    We treat these two outcomes seperately and we argue that if for any $i \in [n],$ we obtain the first outcome, then we conclude with the first outcome of the lemma; otherwise we obtain the second.
    
    Let $\Qq_{i}$ be a $d$-scattered $S_{i}$-$T_{i}$-linkage of order $k$ in $G_i$ that avoids $\Ball^{\lfloor \nicefrac{d}{2} \rfloor}_{G_{i}}(F^\star_i),$ for some $i \in [n].$
    In the two claims that follow we argue that $\Qq_{i}$ is $d$-scattered in $G$ and that it avoids $\Ball^{\lfloor \nicefrac{d}{2} \rfloor}_{G}(F)$ as well.

    \begin{claim} Every $Q \in \Qq_{i}$ avoids $\Ball^{\lfloor \nicefrac{d}{2} \rfloor}_{G}(F).$
    \end{claim}
    \begin{claimproof} Assume towards contradiction that $Q$ intersects $\Ball^{\lfloor \nicefrac{d}{2} \rfloor}_{G}(F),$ i.e. there exists a vertex $u \in F$ such that $\dist_{G}(u,Q) \leq \lfloor \nicefrac{d}{2} \rfloor.$
    There are two cases to analyze.
    If $u \in (B_{i} \setminus A_{i}) \cap F,$ as $(A_{i}, B_{i})$ is a separation of $G$ and $V(Q) \subseteq A_{i},$ it must be that $\dist_{G_{i}}(Q, A_{i} \cap B_{i}) \leq \lfloor \nicefrac{d}{2} \rfloor.$
    This is clearly a contradiction with the fact that $Q$ avoids $\Ball^{\lceil \nicefrac{d}{2} \rceil}_{G_{i}}(A_{i} \cap B_{i}).$

    Then, it must be that $v \in A_{i} \cap F.$
    Let $R$ be a shortest $u$-$v$-path in $G.$
    By our assumption, the length of $R$ must be at most $\lfloor \nicefrac{d}{2} \rfloor.$
    Clearly, $R$ cannot be a path in $G_{i}$ since $\dist_{G_{i}}(u, v) > \lfloor \nicefrac{d}{2} \rfloor$ by the assumption that $Q$ avoids $\Ball^{\lfloor \nicefrac{d}{2} \rfloor}_{G_{i}}(F_{i}).$
    Therefore, $R$ must intersect $A_{i} \cap B_{i}$ and we conclude as in the previous case.
    \end{claimproof}

    \begin{claim} For every pair $Q, Q' \in \Qq_{i}$ we have $\dist_{G}(Q, Q') > d.$
    \end{claim}
    \begin{claimproof} The proof proceeds similarly to the previous claim.
        Assume towards contradiction that this is not the case.
        Let $R$ be a shortest $(u \in V(Q))$-$(v \in V(Q'))$-path in $G.$
        The length of $R$ by assumption is at most $d.$
        Clearly, $R$ cannot be a path in $G_{i}$ since $\Qq_{i}$ is $d$-scattered in $G_{i}.$
        Therefore, $R$ intersects $A_{i} \cap B_{i}.$
        This implies that either the prefix of $R$ terminating at the first intersection with $A_{i} \cap B_{i},$ or the suffix of $R$ starting from the last intersection with $A_{i} \cap B_{i},$ has length at most $\lfloor \nicefrac{d}{2} \rfloor,$ and both of these are paths in~$G_{i}.$
        However, this clearly contradicts the fact that $\Qq_{i}$ avoids $\Ball^{\lceil \nicefrac{d}{2} \rceil}_{G_{i}}(A_{i} \cap B_{i}).$
    \end{claimproof}
    
    Therefore, we may assume that for every $i \in [n],$ we have a guarded tree-decomposition $(T_{i}, \beta_{i}, \gamma_{i})$ as in outcome ii) above.
    First observe that by increasing both the adhesion and vigilance of $(T_{i}, \beta_{i}, \gamma_{i})$ by at most $x,$ we may define the guarded tree-decomposition $(T_i, \beta'_i, \gamma'_{i})$ of $G_{i}$ obtained from $(T_{i}, \beta_{i}, \gamma_{i})$ such that for every node $t \in V(T),$ $\beta'_{i}(t) \coloneqq \beta_{i}(t) \cup (A_{i} \cap B_{i})$ and $\gamma'_{i}(t) \coloneqq \gamma_{i}(t) \cup (A_{i} \cap B_{i}).$
    Clearly by definition, $(T_{i}, \beta'_{i}, \gamma'_{i})$ distance-$d$ guards all $S_{i}$-$T_{i}$-paths in $G_{i}$ that avoid $\Ball^{d}_{G_{i}}(F_{i}).$

    We now define a guarded tree-decomposition $(T, \beta, \gamma)$ for $G$ as follows.
    We first define $T$ by taking the disjoint union of all $T_{i}$ and adding a new root node $r$ which we make adjacent to the root node $r_{i}$ of~$T_{i},$ for all $i \in [n].$
    Now, for each node $t_{i}$ of $T_{i},$ we define $\beta(t_{i}) \coloneqq \beta'_{i}(t_{i})$ and $\gamma_{i} \coloneqq \gamma'_{i}(t_{i}).$
    Finally, for the root node~$r,$ we define $\beta(r) \coloneqq B$ and $\gamma(r) \coloneqq Y.$
    Since for each root node $r_{i}$ we have $\beta(r_{i}) \supseteq A_{i} \cap B_{i},$ it follows that $(T, \beta, \gamma)$ is a valid guarded tree-decomposition of $G.$

    What remains to argue is that $(T, \beta, \gamma)$ does in fact distance-$d$ guard every $S$-$T$-path in $G$ that avoids $\Ball^{d}_{G}(F),$ as desired.
    We do this in the claim that follows.

    \begin{claim} For every $S$-$T$ path $P$ in $G$ that avoids $\Ball^{d}_{G}(F)$ and every node $t \in V(T),$ if $P$ intersects $\beta(t),$ then $P$ also intersects $\Ball^{d}_{G}(\gamma(t)).$
    \end{claim}
    \begin{claimproof}
    There are three cases to distinguish.

    \smallskip
    \noindent\textbf{Case 1, $t = r$:} In this case we conclude directly by the guarantees of \cref{local_structure_genus}.

    \smallskip
    \noindent\textbf{Case 2, $t = t_{i} \in V(T_{i})$ and $P$ is a path in $G_{i}$:} Since $\Ball^{d}_{G}(F) \supseteq \Ball^{d}_{G_{i}}(F_{i}),$ $P$ also avoids $\Ball^{d}_{G_{i}}(F_{i})$ and therefore by the assumptions on $(T'_{i}, \beta'_{i}, \gamma'_{i}),$ $\Ball^{d}_{G_{i}}(\gamma(t_{i}))$ intersects $P.$
    Since $\Ball^{d}_{G_{i}}(\gamma(t_{i})) \subseteq \Ball^{d}_{G}(\gamma(t_{i})),$ so does $\Ball^{d}_{G}(\gamma(t_{i})).$

    \smallskip
    \noindent\textbf{Case 3, $t = t_{i} \in V(T_{i})$ and $P$ is not a path in $G_{i}$:} In this case, since $(A_{i}, B_{i})$ is a separation of $G,$ it must be that $P$ intersects $A_{i} \cap B_{i}$ which by definition is included in $\gamma(t_{i}).$
    \end{claimproof}

    With this the proof concludes.
\end{proof}

\subsection{The final proof}

We are finally in the position to prove Weak Coarse Menger's Conjecture for surface-embeddable graphs. The following statement is a more detailed variant of \cref{thm:main}.

\begin{theorem}\label{menger_genus}
    There exists a function $f_{\ref{menger_genus}}\colon \N^3\to \N$ such that for all $g, d, k\in \N$ the following holds.
    Let $G$ be a graph embeddable in a surface of genus $g$ and let $S,T\subseteq V(G)$ be vertex subsets.
    Suppose that one cannot find $k$ $S$-$T$-paths in $G$ that are pairwise at distance more than $d$ apart.
    Then there is a vertex subset $X\subseteq V(G)$ with $|X|\leq f_{\ref{menger_genus}}(g,k,d)$ such that every $S$-$T$-path in $G$ is at distance at most $d$ from some vertex of $X$.

    Moreover, we have that
    \begin{align*}
        f_{\ref{menger_genus}}(g, d, k) \coloneqq~&\big((2g - 1)(\rep_{\ref{local_structure_genus}}(g, d, k, 2g + 1) - 1) + \vig_{\ref{global_structure_genus}}(d, k, 2g + 1)(k - 1)^{2} \big)(k - 1)\\[3pt]
        \in~&2^{d^{\Oh(1)} \cdot 2^{\Oh((k+g) \log (k+g)))}}.
    \end{align*}
\end{theorem}
\begin{proof} First observe that if $G$ is not connected, there can be at most $k - 1$ connected components of $G$ that contain an $S$-$T$-path.
    Now, each of these $k - 1$ connected components we may treat independently and in the end take the union of the hitting sets we obtain for each of them, in order to distance-$d$ hit all $S$-$T$-paths in $G.$
    This allows us to assume that $G$ is connected.
    
    We first apply the Representativity Lemma (\cref{lemma:representativity}) to $G$ with $r \coloneqq \repr_{\ref{local_structure_genus}}(g, d, k, 2g + 1).$
    This yields a set of vertices $A \subseteq V(G)$ with $|A| \leq (2g - 1)(r - 1)$ (or $A=\emptyset$, if $g=0$) and a partition $\Pp$ of $V(G) \setminus A$ into non-empty vertex subsets (or $\Pp = \emptyset$ if $V(G) = A$) such that
    \begin{itemize}
        \item $\bigcup_{H \in \Pp} G[H \cup A] = G$; and
        \item for every $H \in \Pp,$ there exists a graph $H^\star$ and a vertex subset $A^\star \subseteq V(H^\star)$ with $|A^\star| \leq 2g$ satisfying the following two conditions:
        \begin{enumerate}
            \item $H^\star - A^\star = G[H]$ and $N_{H^\star}(A^\star) = N_{G}(A) \cap H$; and
            \item $H^\star$ has an embedding on a surface of genus at most $g$ with representativity at least $r.$
        \end{enumerate}
    \end{itemize}
    First observe that there can be at most $k - 1$ distinct sets $H \in \Pp$ such that $H^\star - \Ball^{d}_{H^\star}(A^\star)$ contains an $S$-$T$-path.
    Indeed, since $H^\star - A^\star = G[H],$ any $S$-$T$-path in $H^\star - A^\star$ is an $S$-$T$-path in $G.$
    In fact, since $\bigcup_{H \in \Pp} G[H \cup A] = G,$ any two such paths belong to different connected components of $G - A.$
    Then, by \cref{obs:hitA}, any two such paths are disjoint from $\Ball^{d}_{G}(A),$ and hence at distance more than $d$ in $G.$

    Let $\Pp'$ be the subset of $\Pp$ that contains only those vertex subsets $H$ for which $H^\star - \Ball^{d}_{H^\star}(A^\star)$ contains an $S$-$T$-path.
    By the previous observation $|\Pp'| \leq k - 1.$
    Now, for each $H \in \Pp',$ apply \cref{global_structure_genus} on $H^\star$ with $F$ being $A^\star$ and $q$ being $2g + 1.$
    This yields one of two possible outcomes:
    	\begin{itemize}
		\item $H^\star$ contains a $d$-scattered $S$-$T$-linkage $\Qq_{H^\star}$ of order $k$ that avoids $\Ball^{\lfloor \nicefrac{d}{2} \rfloor}_{H^\star}(A^\star)$; or
		\item $H^\star$ has a guarded tree-decomposition $(T_{H^\star}, \beta_{H^\star}, \gamma_{H^\star})$ of adhesion at most $\adh_{\ref{global_structure}}(d, k, 2g + 1)$ and vigilance at most $\vig_{\ref{global_structure}}(d, k, 2g + 1)$ that distance-$d$ guards all $S$-$T$-paths in $H^\star$ that avoid $\Ball^{d}_{H^\star}(A^\star).$
	\end{itemize}
    In case we get the first outcome above, \cref{obs:packA} implies that $\Qq_{H^\star}$ is in fact a $d$-scattered linkage of order $k$ in $G$, which contradicts our assumptions.
    Therefore, we get the second outcome.
    Now, we apply \cref{thm:tree-EP} on $H^\star$ and $(T_{H^\star}, \beta_{H^\star}, \gamma_{H^\star})$ with $\Hh$ being the family of all $S$-$T$-paths in $H^\star$ that avoid $\Ball^{d}_{H^\star}(A^\star).$
    This call cannot yield a $d$-scattered subfamily of size $k,$ since as before, \cref{obs:packA} would yield a contradiction.
    Therefore, we may assume that for every $H \in \Pp',$ we get a set $X_{H^\star} \subseteq V(H^\star)$ of size at most $\vig_{\ref{global_structure}}(d, k, 2g + 1) \cdot (k - 1)$ such that $\Ball^d_{H^\star}(X_{H^\star})$ intersects every $S$-$T$-path in $H^\star$ that avoids $\Ball^d_{H^\star}(A^\star).$
    This means that the only $S$-$T$-paths in $H^\star$ that may not be hit by $\Ball^d_{H^\star}(X_{H^\star})$ are those that intersect $\Ball^d_{H^\star}(A^\star).$
    Therefore, if we set $X'_{H^\star} \coloneqq X_{H^\star} \setminus A^\star,$ we get that the desired hitting set at distance-$d$ for $H^\star$ is $X'_{H^\star} \cup A^\star.$

    Now, let $$X \coloneqq \bigcup_{H \in \Pp'} (X_{H^\star} \setminus A^\star).$$
    Note that $|X| \leq \vig_{\ref{global_structure}}(d, k, 2g + 1) \cdot (k - 1)^{2}.$
    We argue that $X \cup A$ is the desired distance-$d$ hitting set for $G.$
    Indeed, \cref{cor:HitA} now implies that for every $H\in \Pp'$, we have $H^\star - \Ball^{d}_{H^\star}(X'_{H^\star} \cup A^\star) = G[H \cup A] - \Ball^d_{G}(X \cup A),$ and therefore, $G[H \cup A] - \Ball^d_{G}(X \cup A)$ cannot contain any $S$-$T$-path.
    Since $\bigcup_{H \in \Pp} G[H \cup A] = G$ and $|X\cup A|$ is bounded as promised, we conclude.
\end{proof}


\phantomsection
\addcontentsline{toc}{section}{References}
\bibliographystyle{abbrv}
\bibliography{arxiv_v1_literature}

\end{document}